\documentclass[11pt,a4paper,reqno]{amsart}
\pdfoutput=1
\usepackage{ae,aecompl}
\usepackage[english]{babel}
\usepackage{amssymb,amsaddr}
\usepackage{pictex,dcpic,xcolor}

\title{Dirac Operators on Quantum Projective Spaces}
\date{January 29, 2009}

\author[Francesco D'Andrea and Ludwik D\k{a}browski]{Francesco D'Andrea$^1$ and Ludwik D\k{a}browski$^{\,2}$}
\address{$^1$D{\'e}p.~de Math{\'e}matique, U.C.~Louvain, Louvain-La-Neuve, B-1348, Belgique \\
         $^2$Scuola Internazionale Superiore di Studi Avanzati, Trieste, I-34127, Italia\\
         ~}


\allowdisplaybreaks[2]
\numberwithin{equation}{section}
\newtheorem{thm}{Theorem}[section]
\newtheorem{lemma}[thm]{Lemma}
\newtheorem{prop}[thm]{Proposition}

\newcommand{\arxiv}[1]{[arxiv:{\color{blue}#1}]}
\newcommand{\doi}[1]{[doi:{\color{blue}#1}]}

\newcommand{\N}{\mathbb{N}}
\newcommand{\Z}{\mathbb{Z}}

\newcommand{\C}{\mathbb{C}}

\newcommand{\A}{\mathcal{A}}
\newcommand{\M}{\mathcal{M}}

\newcommand{\U}{\mathcal{U}}

\newcommand{\HH}{\mathcal{H}}

\newcommand{\Uq}[1]{U_q(\mathfrak{su}(#1))}
\newcommand{\Aq}[1]{\mathcal{A}(\CP_q^{#1})}
\newcommand{\Oq}{\mathcal{A}(SU_q(\ell+1))}
\newcommand{\Kq}[1]{U_q(\mathfrak{u}(#1))}
\newcommand{\CP}{\mathbb{C}\mathrm{P}}
\newcommand{\Gr}{\mathtt{Gr}^\ell_q}

\newcommand{\az}{\triangleright}
\newcommand{\za}{\triangleleft}
\newcommand{\adj}{\stackrel{\mathrm{ad}}{\za}}
\newcommand{\inner}[1]{\left<#1\right>}
\newcommand{\D}{D\mkern-11.5mu/\,}
\newcommand{\de}{\partial}
\newcommand{\wprod}{\wedge_q\mkern-1mu}
\newcommand{\deb}{\bar{\partial}}
\newcommand{\Ja}{\mathcal{J}_0}
\newcommand{\Jb}{\mathcal{J}}
\newcommand{\maut}{K_{2\rho}}
\newcommand{\mL}{\mathcal{L}}

\begin{document}

\begin{abstract}
We construct a family of self-adjoint operators $D_N$, $N\in\Z$, which have compact resolvent and bounded commutators with the coordinate algebra of the quantum projective space $\CP^\ell_q$, for any $\ell\geq 2$ and $0<q<1$. They provide $0^+$-dimensional equivariant even spectral triples. If $\ell$ is odd and $N=\frac{1}{2}(\ell+1)$, the spectral triple is real with KO-dimen\-sion~\mbox{$2\ell\!\!\mod 8$}.
\end{abstract}

\maketitle

\thispagestyle{empty}

\section{Introduction}
In recent years several examples of noncommutative riemannian spin manifolds, described in terms of spectral triples \cite{Con94,Con96}, have been constructed. Among them there are lowest dimensional quantum groups and their homogeneous spaces (see \cite{Dab08} for references), and $q$-deformed compact simply connected simple Lie groups~\cite{NT07}. An equivariant Dirac operator $D$ satisfying the crucial property of bounded commutators with the coordinates has been constructed on\linebreak $q$-deformed irreducible flag manifolds in~\cite{Kra04} (and shown to yield a finite dimensional differential calculus which coincides with the one of \cite{HK04}). The other essential property of a spectral triple --- that the resolvent of $D$ is compact --- though expected, has not yet been demonstrated.

In this paper we analyse a class of $q$-deformed irreducible flag manifolds: namely quantum projective spaces $\CP^\ell_q$ for any $\ell\in \N$. We first give an explicit description of the antiholomorphic part of the differential calculus (the Dolbeault complex) and use it to construct a family (numbered by $N\in\Z$) of self-adjoint operators $D_N$ on $\HH_N$, which have bounded commutators with the coordinate algebra $\Aq{\ell}$. Then, being $\CP^\ell_q$ a homogeneous $SU_q(\ell+1)$-space, by preserving the equivariance at all steps and by relating $D_N$ to certain Casimir operator of $SU_q(\ell+1)$, we are able to study the asymptotic behaviour of the spectrum of $D_N$. We find the exponentially growing spectrum, which guarantees the compact resolvent property of $D_N$. Thus $(\Aq{\ell},\HH_N,D_N)$ are bona fide spectral triples on noncommutative homogeneous manifolds $\CP^\ell_q$. This generalizes the simplest case $\CP^1_q$ (that coincides with the standard Podle\'s sphere) and the case $\CP_q^2$ (that is spin$^c$ but not spin). The spectral triple with \mbox{$N=\frac{1}{2}(\ell+1)$}, that exists if $\ell$ is odd, is the analogue of the real spectral triple in~\cite{DS03}, and the one with $N=0$ is the analogue of the spectral triple in~\cite{DDL08b}.

It should be mentioned that the de Rham complex for $\CP_q^\ell$ (on the formal level) appears in \cite{HK06b} and (in local coordinates) in \cite{CHZ96}. The relevant differential operator in local coordinates on $\CP^1_q$ appears already in \cite{DS94}, where, in particular, the relation with the $q$-derivative is mentioned.

In the classical limit ($q=1$), when $\ell$ is odd and $N=\frac{1}{2}(\ell+1)$, we obtain the canonical Dirac operator (for the Fubini-Study metric) acting on the space of square integrable spinors on $\CP^\ell$, while for $N=0$ we get the Dolbeault-Dirac operator on the Hilbert space of antiholomorphic forms on $\CP^\ell$. Their spectra agree with the formula in~\cite{DHMOC08}, cf.~also \cite{CFG89,SS93,AB98} if $\ell$ is odd.

The plan of the paper is the following.
In Sec.~\ref{sec:zero} we briefly recall what is known about $\CP^1_q$, to prepare the discussion of the general case.
In Sec.~\ref{sec:uno}, we describe the basic properties of $\Uq{\ell+1}$ and --- guided by the equivariance condition --- introduce a $q$-deformation of the Grassmann algebra and of left invariant vector fields on $\CP^\ell$. The former will be relevant in the construction of the algebra of antiholomorphic forms, the latter in the definition of the exterior derivative.
In Sec.~\ref{sec:due}, we describe the quantum $SU(\ell+1)$ group and the action of $\Uq{\ell+1}$ on it, as well as the subalgebras of `functions' on the quantum unitary sphere $S^{2\ell+1}_q$, and on the quantum complex projective space $\CP^\ell_q$.
Sec.~\ref{sec:tre} is dedicated to the differential calculus, and Sec.~\ref{sec:quattro} to spectral triples.
General notions on spectral triples are recalled in Appendix \ref{app:A}.
Finally, in Appendix \ref{app:B} we discuss the limit $q\to 1$ and compare our results with the literature.

\section{The `exponential' Dirac operator on $\CP^1_q$}\label{sec:zero}
In this section, we briefly recall the geometry of the $q$-deformed $\CP^1$, cf.~\cite{DS03,SW04}.
We use the notations of \cite{DDLW07,DDL08}: $0<q<1$ is a real deformation parameter,
$K,K^{-1},E,F$ are the generators of the Hopf $*$-algebra $\Uq{2}$, $\alpha,\beta$ are the
generators of the dual Hopf $*$-algebra $\A(SU_q(2))$, which is
an $\Uq{2}$-bimodule $*$-algebra for the left $\az$ and right $\za$ canonical
actions (cf~App.~\ref{app:A}). For each $N\in\Z$, a left $\Uq{2}$-module $\Gamma_N$ is given by
$$
\Gamma_N=\big\{ a\in\A(SU_q(2))\,|\,a\za K=q^{-\frac{N}{2}}a\big\} \;,
$$
and $\Aq{1}:=\Gamma_0$ is a left $\Uq{2}$-module $*$-algebra called the coordinate
algebra of the standard Podle\'s sphere. For each $N\in\Z$, $\Gamma_N$ is also an $\Aq{1}$-bimodule.
As a left $\Uq{2}$-module, we have the following decomposition
$$
\Gamma_N\simeq\bigoplus_{n-|N|\in 2\N}V_n \;,
$$
where $V_n$ is the spin $\frac{1}{2}n$ irreducible $*$-representation of $\Uq{2}$.
This is a unitary equivalence if we put on $\Gamma_N$ the inner product
coming from the Haar state of $SU_q(2)$ (see~\cite{KS97}). The Casimir element
$$
\mathcal{C}_q=\left(\frac{\smash[t]{q^{\frac{1}{2}}}K-\smash[t]{q^{-\frac{1}{2}}K^{-1}}}{q-q^{-1}}\right)^2+FE
$$
has eigenvalues
$$
\mathcal{C}_q\big|_{V_n}=[\tfrac{n+1}{2}]^2\cdot id
$$
with multiplicity $\dim V_n=n+1$. Here
$$
[x]:=\frac{q^x-q^{-x}}{q-q^{-1}}
$$
is the $q$-analogue of $x$.

Antiholomorphic $0$ and $1$-forms are $\Omega^0=\Aq{1}$ and $\Omega^1=\Gamma_{-2}$,
with the Dolbeault operator and its Hermitian conjugate given by
\begin{align*}
\deb &: \Omega^0\to\Omega^1\;,\qquad a\mapsto \mL_F\, a \;,\\
\deb^\dag\! &: \Omega^1\to\Omega^0\;,\qquad a\mapsto \mL_E\, a \;,
\end{align*}
where
$$
\mL_ha:=a\za S^{-1}(h) \;,\qquad\forall\;a\in\A(SU_q(2))\,,\;h\in\Uq{2}\,,
$$
and $S$ is the antipode of $\Uq{2}$. It is shown already in \cite{DS94}, by using local `coordinates', that $\deb$
is related to the well-known $q$-derivative operator; cf.~(4.19) therein.

The Dolbeault-Dirac operator $D$ on $\Omega^0\oplus\Omega^1$ is given by
$$
D(\omega_0,\omega_1):=(\deb^\dag\omega_1,\deb\omega_0)
=-(q^{-1}\omega_1\za E,q\,\omega_0\za F)
$$
and satisfies $D^2\omega=\omega\za (\mathcal{C}_q-[\frac{1}{2}]^2)$.
Being an even spectral triple, the spectrum of $D$ must be symmetric with respect to the origin.
It is computed from the spectrum of $\mathcal{C}_q$, by using the above
decomposition of $\Gamma_N$ and the fact that for central elements the
left and right canonical actions are equal. It immediately follows
that $D$ has a $1$-dimensional kernel, and its non-zero eigenvalues are
$\pm\sqrt{[k][k+1]}$ with multiplicity $2k+1$, for all $k\in\N+1$.

To get the Dirac operator (for the Fubini-Study metric) we must tensor
$\Omega^0\oplus\Omega^1$ with the square root of the canonical bundle of holomorphic
$1$-forms, i.e.~with $\Gamma_1$. We get the space
$$
(\Omega^0\oplus\Omega^1)\otimes_{\Aq{1}}\Gamma_1\simeq\Gamma_1\oplus\Gamma_{-1} \;.
$$
The Dirac operator $\D$ is obtained by twisting $D$ with the Grassmannian
connection of $\Gamma_{-1}$. This goes as follows. Given any $\Aq{1}$-bimodule
$\mathfrak{M}\subset\A(SU_q(2))$, the map
\begin{align*}
\phi&:\mathfrak{M}\otimes_{\Aq{1}}\Gamma_1\to\mathfrak{M}^2p_B\;,\qquad a\mapsto a(\alpha,\beta) \;,\\
\phi^{-1}\!\!&:\mathfrak{M}^2p_B\to\mathfrak{M}\otimes_{\Aq{1}}\Gamma_1\;,\qquad (a_1,a_2)\mapsto a_1\alpha^*+a_2\beta^* \;,
\end{align*}
is an isomorphism of left $\Aq{1}$-modules, where $p_B$ is the $q$-analogue of
Bott projection
$$
p_B:=(\alpha,\beta)^\dag (\alpha,\beta) \;.
$$
The Dirac operator $\D$ on $\Gamma_1\oplus\Gamma_{-1}$ is
$$
\D:=\phi^{-1}(D\otimes 1_2)\phi \;,
$$
where $\phi$ in this case sends $\Gamma_1\oplus\Gamma_{-1}\simeq (\Omega^0\oplus\Omega^1)\otimes_{\Aq{1}}\Gamma_1$
to $(\Omega^0\oplus\Omega^1)^2p_B$.
We compute $\D$ explicitly. For $v_+\in\Gamma_1$ and $v_-\in\Gamma_{-1}$ we have
\begin{align*}
\D(v_+,v_-)=-\left(\,
q^{-1}v_-(\alpha,\beta)\za E\tbinom{\alpha^*}{\beta^*}\,,\,
q\,v_+(\alpha,\beta)\za F\tbinom{\alpha^*}{\beta^*}\,\right) \;.
\end{align*}
But $(\alpha,\beta)\za E=0$, $(\alpha,\beta)^\dag\za F=0$ and thus
$$
0=1\za F=(\alpha,\beta)\tbinom{\alpha^*}{\beta^*}\za F
=q^{-\frac{1}{2}}(\alpha,\beta)\za F\tbinom{\alpha^*}{\beta^*} \;.
$$
Thus $F,E$ acts non-trivially only on the $v_+,v_-$ part and using
$(\alpha,\beta)\za K(\alpha,\beta)^\dag=q^{\frac{1}{2}}$
we get
$$
\D(v_+,v_-)=-q^{\frac{1}{2}}(q^{-1}v_-\za E,q\,v_+\za F) \;.
$$
Hence, $\D$ has an expression similar to $q^{\frac{1}{2}}D$, although living on a
different Hilbert space. This is exactly the Dirac operator of \cite{DS03} (but for
the factor $q^{\frac{1}{2}}$), as proved in \cite{SW04}.
Since $\D^2(v_+,v_-)=(v_+,v_-)\za q\,\mathcal{C}_q$, it immediately follows that
$\D$ has eigenvalues $\pm q^{\frac{1}{2}}[k]$ with multiplicity $2k$, for all $k\in\N+1$.

A crucial difference between $D$ and $\D$ is that the latter admits
a real structure $J$. This is the operator
$$
J(v_+,v_-):=K\az (v_-^*,-v_+^*)\za K \;.
$$
One can show that $JK\az $ is left $\Uq{2}$-covariant, $J^2=-1$, $J$ is an isometry, $J\D=\D J$
and the commutant and first order conditions are satisfied (cf.~\cite{SW04}), meaning that
we have a real spectral triple with KO-dimension $2$.

The analogous of the operator $D$ for $\CP^2_q$ has been constructed in \cite{DDL08b}. 

\section{Preliminaries about $\Uq{\ell+1}$}\label{sec:uno}
For $0<q<1$, we denote $\Uq{\ell+1}$ the `compact' real form of the
Hopf algebra denoted $\breve{U}_q(\mathfrak{sl}(\ell+1,\C))$ in Sec.~6.1.2 of~\cite{KS97}.
As a $*$-algebra it is generated by $\{K_i=K_i^*,K_i^{-1},E_i,F_i=E_i^*\}_{i=1,2,\ldots,\ell}$
with relations
\begin{align*}
[K_i,K_j]&=0 \;,\\
K_iE_iK_i^{-1}&=qE_i\;, \\
K_iE_jK_i^{-1}&=q^{-1/2}E_j &&\mathrm{if}\;|i-j|=1\;, \\
K_iE_jK_i^{-1}&=E_j &&\mathrm{if}\;|i-j|>1\;, \\
[E_i,F_j]&=\delta_{ij}\,\frac{K_i^2-K_i^{-2}}{q-q^{-1}}\;, \\
E_i^2E_j-(q+q^{-1})&E_iE_jE_i+E_jE_i^2=0 &&\mathrm{if}\;|i-j|=1\;,\\
[E_i,E_j]&=0 &&\mathrm{if}\;|i-j|>1\;,
\end{align*}
plus conjugated relations. If we define the $q$-commutator as
$$
[a,b]_q:=ab-q^{-1}ba \;,
$$
the second-last relation can be rewritten in two equivalent forms
$$
[E_i,[E_j,E_i]_q]_q=0 \qquad\mathrm{or}\qquad
[[E_i,E_j]_q,E_i]_q=0 \;,
$$
for any $|i-j|=1$. Coproduct, counit and antipode are given by
\begin{gather*}
\Delta(K_i)=K_i\otimes K_i\;,\qquad
\Delta(E_i)=E_i\otimes K_i+K_i^{-1}\otimes E_i\;, \\
\epsilon(K_i)=1\;,\qquad
\epsilon(E_i)=0\;,\qquad
S(K_i)=K_i^{-1}\;,\qquad
S(E_i)=-qE_i\;.
\end{gather*}
Using self-evident notation, we call $\Uq{\ell}$ the Hopf
$*$-subalgebra of $\Uq{\ell+1}$ generated by the elements
$\{K_i=K_i^*,K_i^{-1},E_i,F_i=E_i^*\}_{i=1,2,\ldots,\ell-1}$.
Its commutant is the Hopf $*$-subalgebra $U_q(\mathfrak{u}(1))\subset\Uq{\ell+1}$
generated by the element $K_1K_2^2\ldots K_\ell^\ell$ and its inverse.
This is a positive operator in all representations we consider.
Its positive root of order $\frac{2}{\ell+1}$,
\begin{equation}\label{eq:Khat}
\hat{K}:=(K_1K_2^2\ldots K_\ell^\ell)^{\frac{2}{\ell+1}} \;,
\end{equation}
and its inverse will serve to define a Casimir operator.
We enlarge the algebra $\Uq{\ell+1}$ accordingly.

The element
\begin{equation}\label{eq:Ssquare}
\maut=(K_1^\ell K_2^{2(\ell-1)}\ldots K_j^{j(\ell-j+1)}\ldots K_\ell^\ell)^2 \;,
\end{equation}
implements the square of the antipode:
\begin{equation}\label{eq:Sdue}
S^2(h)=\maut h\maut^{-1} \;,\qquad\forall h\in\Uq{\ell+1}\;,
\end{equation}
as one can easily check on generators of the Hopf algebra.
By \cite[Sec.~11.3.4]{KS97}, Ex.~9, we see that the pairing
$\inner{\maut,a}=f_1(a)$ is the character giving the modular automorphism
(cf.~(11.36) in \cite{KS97}).
The expression $f_1.a.f_1$ in \cite{KS97} becomes $\maut\az a\za\maut$ in our notations,
and by (11.26) of \cite{KS97} the Haar state $\varphi:\Oq\to\C$ satisfies
\begin{equation}\label{eq:modprop}
\varphi(ab)=\varphi\bigl(b\,\maut\az a\za\maut) \;,
\end{equation}
for all $a,b\in\Oq$.

We are interested in highest weight $*$-representations of $\Uq{\ell+1}$ such that
$K_j$ are represented by positive operators.
Such irreducible $*$-representations are labeled
by $\ell$ non-negative integers $n_1,\ldots,n_\ell$. For
$n=(n_1,\ldots,n_\ell)\in\N^\ell$ we denote by $V_n$ the vector space
carrying the representation $\rho_n$ with highest
weight $n$; the highest weight vector $v$ is annihilated by all the $E_j$'s
and satisfies $\rho_n(K_i)v=q^{n_i/2}v$, $i=1,\ldots,\ell$.

\subsection{The Casimir operator}\label{sec:Cas}
Casimir operators for $\Uq{\ell+1}$ are discussed in \cite{Bin91,Cha91}. We repeat here the
construction from scratch, adapting their notations to ours, to make the paper self-contained.
Moreover, some formul{\ae} in the proofs will be useful later on.

For any $j,k\in\{1,\ldots,\ell\}$, with $j<k$, we define the following
elements of $\Uq{\ell+1}$
$$
M_{jk}:=[E_j,[E_{j+1},[E_{j+2},\ldots[E_{k-1},E_k]_q
\ldots]_q]_q]_q\;.
$$
That is, if we set $M_{ii}=E_i$ the $M_{jk}$'s are
obtained by iteration using
\begin{equation}\label{eq:iter}
M_{jk}=[E_j,M_{j+1,k}]_q \;.
\end{equation}
(For $q=1$, $M_{jk}$ and $M^*_{jk}$, together with the Cartan generators, form a basis of the
Lie algebra $\mathfrak{su}(\ell+1)$.) We need the following Lemmas.

\begin{lemma}\label{lemmaA}
The following equalities hold:
\begin{subequations}
\begin{align}
[F_i,M_{jk}] &=\delta_{ij}\,M_{j+1,k}K_i^{-2}
-\delta_{ik}\,K_i^2M_{j,k-1}
-\delta_{ij}\delta_{ik}\tfrac{K_i^2-K_i^{-2}}{q-q^{-1}} \;, \label{eq:prov} \\
[E_i,M_{jk}^*] &=\delta_{ik}\,M_{j,k-1}^*K_i^2
-\delta_{ij}\,K_i^{-2}M_{j+1,k}^*
+\delta_{ij}\delta_{ik}\tfrac{K_i^2-K_i^{-2}}{q-q^{-1}} \;, \label{eq:cons}
\end{align}
\end{subequations}
where we set $M_{jk}:=0$ when the labels are out of the range (i.e.~when $j>k$).
\end{lemma}
\begin{proof}
Since (\ref{eq:prov}) implies (\ref{eq:cons}) by adjunction, we have
to prove only the former. Since for $j=k$ (or $j>k$) this is trivial,
we assume $j<k$. We notice that $F_i$ commutes with all $E_j$'s
but for $i=j$, and $K_i$ commutes with all $E_j$'s but for $i=j\pm 1$.
In particular, this means that $[F_i,M_{jk}]=0$ if $i<j$.
For $i=j$ (so $i<k$) using (\ref{eq:iter}) we get
$$
[F_i,M_{ik}]=[[F_i,E_i],M_{i+1,k}]_q=
[-\tfrac{K_i^2-K_i^{-2}}{q-q^{-1}},M_{i+1,k}]_q=
M_{i+1,k}K_i^{-2} \;,
$$
which agrees with (\ref{eq:prov}).
If $i=j+1$ and $i<k$ using again (\ref{eq:iter}) and the identity
just proved we get
$$
[F_i,M_{i-1,k}]=[E_{i-1},[F_i,M_{ik}]]_q=
[E_{i-1},M_{i+1,k}K_i^{-2}]_q=
[E_{i-1},M_{i+1,k}]K_i^{-2}=0\; ;
$$
this is zero since $E_{i-1}$ commutes with all $E_j$ with $j\geq i+1$.
Using the equation just proved (which is true for $i<k$), by (\ref{eq:iter})
we prove by induction on $j$ that
$$
[F_i,M_{jk}]=[E_j,[F_i,M_{j+1,k}]]_q=0 \;,
$$
for all $j<i<k$. If $j=i-1$ and $k=i$ we have
$$
[F_i,M_{i-1,i}]=[E_{i-1},[F_i,E_i]]_q
=[E_{i-1},-\tfrac{K_i^2-K_i^{-2}}{q-q^{-1}}]_q
=-K_i^2E_{i-1}=-K_i^2M_{i-1,i-1} \;,
$$
and again by induction on $j$ we prove that
$$
[F_i,M_{ji}]=[E_j,[F_i,M_{j+1,i}]]_q=
[E_j,-K_i^2M_{j+1,i-1}]_q=
-K_i^2[E_j,M_{j+1,i-1}]_q=
-K_i^2M_{j,i-1} \;,
$$
for all $j<i-1$. With this (\ref{eq:prov}) is proved for any $i\leq k$.
For $i>k$ it holds trivially. This concludes the proof.
\end{proof}

For all $j,k\in\{1,\ldots,\ell\}$ let
\begin{equation}\label{eq:Njk}
N_{jk}:=(K_jK_{j+1}\ldots K_\ell)\cdot(K_{k+1}K_{k+2}\ldots K_\ell)\cdot\hat{K}^{-1}
\end{equation}
and notice that
\begin{equation}\label{eq:NEN}
N^2_{jk}E_iN_{jk}^{-2}=q^{-\delta_{i,j-1}+\delta_{i,j}-\delta_{i,k}+\delta_{i,k+1}}E_i
\end{equation}
for all $i,j,k\in\{1,\ldots,\ell\}$.

\begin{lemma}\label{lemmaB}
The following equality holds
\begin{equation}\label{eq:last}
[E_i,N^2_{jk}M_{jk}]=
\delta_{i,j-1}qN^2_{jk} M_{ik}
-\delta_{i,k+1}N^2_{jk} M_{ji} \;,
\end{equation}
where we set $M_{jk}:=0$ when the labels are out of the range.
\end{lemma}

\begin{proof}
We assume $j\leq k$, as for $j>k$ the claim is a trivial $0=0$.
First we notice that if $i<j-1$ or $i>k+1$ we have $[E_i,M_{jk}]=0$
since $E_i$ commutes with any $E_n$ with $|n-i|>1$; also
by (\ref{eq:NEN}) we have $[E_i,N_{jk}]=0$, and this proves
(\ref{eq:last}) in the cases $i<j-1$ and $i>k+1$.

By (\ref{eq:iter}) and (\ref{eq:NEN}) we have the recursive
definition of $N^2_{jk} M_{jk}$,
$$
N^2_{ik} M_{ik}=
q^{-1}K_i^2[E_i,N^2_{i+1,k} M_{i+1,k}]\;,
$$
which gives (\ref{eq:last}) in the case $i=j-1$.
On the other hand if $k=i-1$, since $[E_i,E_l]=0$ for $i>l+1$ we have
\begin{align*}
[M_{j,i-1},E_i]_q&=
[[E_j,[E_{j+1},\ldots[E_{i-2},E_{i-1}]_q\ldots]_q]_q,E_i]_q \\
&=[E_j,[E_{j+1},\ldots[E_{i-2},[E_{i-1},E_i]_q]_q\ldots]_q\ldots]_q
=M_{j,i} \;,
\end{align*}
and this together with (\ref{eq:NEN}) gives (for $j\leq i-1$)
$$
[N^2_{j,i-1}M_{j,i-1},E_i]=
N^2_{j,i-1}[M_{j,i-1},E_i]_q=
N^2_{j,i-1}M_{j,i}
$$
which is (\ref{eq:last}) in the case $i=k+1$. It remains to consider
$j\leq i\leq k$, which by (\ref{eq:NEN}) is equivalent to the following
set of equations
\begin{subequations}
\begin{align}
[E_i,M_{ji}]_q&=0 \qquad\mathrm{if}\;j<i \;, \label{eq:lastD} \\
[M_{ik},E_i]_q&=0 \qquad\mathrm{if}\;i<k \;, \label{eq:lastE} \\
[E_i,M_{jk}] &=0\qquad\mathrm{if}\;i=j=k\;\mathrm{or}\;j<i<k \;.\label{eq:lastC}
\end{align}
\end{subequations}
The case $i=j=k$ is trivial.
Furthermore by Serre's relations
\begin{align*}
E_iM_{i-1,i}= E_i[E_{i-1},E_i]_q
 &=q^{-1}[E_{i-1},E_i]_qE_i=q^{-1}M_{i-1,i} E_i \;, \\
E_iM_{i,i+1}=E_i[E_i,E_{i+1}]_q
 &=q[E_i,E_{i+1}]_qE_i=qM_{i,i+1} E_i \;.
\end{align*}
Then, for any $j+1<i=k$ we prove by induction that
\begin{subequations}\label{eq:pro}
\begin{equation}
E_iM_{j,i}=[E_j,E_iM_{j+1,i}]_q
=q^{-1}[E_j,M_{j+1,i}E_i]_q
=q^{-1}[E_j,M_{j+1,i}]_qE_i
=q^{-1}M_{j,i}E_i
\end{equation}
which is (\ref{eq:lastD}) and for any $j=i<k-1$ that
\begin{equation}
E_iM_{i,k}=[M_{i,k-1},E_k]_q
=[E_iM_{i,k-1},E_k]_q
=q[M_{i,k-1}E_i,E_k]_q
=qM_{i,k}E_i \;,
\end{equation}
\end{subequations}
which is (\ref{eq:lastE}).

Consider now $j\leq n<k$ and notice that for any such
$i,j,k$ we can write
$$
M_{jk}=[M_{jn},M_{n+1,k}]_q \;.
$$
Using this equation in the cases $n=i,i-1$,
together with (\ref{eq:pro}) and (\ref{eq:iter})
we have
\begin{align*}
qE_iM_{jk}-q^{-1}M_{jk}E_i &=
[M_{ji},[E_i,M_{i+1,k}]_q]=
[M_{ji},M_{ik}] \;,
\\
qM_{jk}E_i-q^{-1}E_iM_{jk} &=
[[M_{ji-1},E_i]_q,M_{ik}]=
[M_{ji},M_{ik}] \;.
\end{align*}
The difference of the two lines has to be zero, so
$$
0=qE_iM_{jk}-q^{-1}M_{jk}E_i
-qM_{jk}E_i+q^{-1}E_iM_{jk}=
[2]\,[E_i,M_{jk}]\;.
$$
This concludes the proof.
\end{proof}

A Casimir operator $\mathcal{C}_q$ for $\Uq{\ell+1}$ is given by the formula
\begin{equation}\label{eq:Cq}
\mathcal{C}_q=\sum_{i=1}^\ell\tfrac{q^{\ell+2-2i}}{(q-q^{-1})^2}
N^2_{i,i-1}+\tfrac{q^{-\ell}}{(q-q^{-1})^2}\,\hat{K}^{-2}
+\sum_{1\leq j\leq k\leq\ell}q^{\ell+1-2j}M_{jk}^*N^2_{jk}M_{jk}
-\tfrac{[\ell+1]}{(q-q^{-1})^2} \;\,.
\end{equation}

\begin{prop}
The operator $\mathcal{C}_q$ is real ($\mathcal{C}_q=\mathcal{C}_q^*$) and central. In the irreducible representation $\rho_n:\Uq{\ell+1}\to\mathrm{End}(V_n)$ with highest weight $n=(n_1,\ldots,n_\ell)$, it is proportional to the identity with proportionality constant:
\begin{equation}\label{eq:SpCqa}
\rho_n(\mathcal{C}_q)=
\frac{1}{2}\sum_{i=1}^{\ell+1}\left[\tfrac{\sum_{j=1}^{i-1}jn_j-\sum_{j=i}^\ell(\ell+1-j)n_j}{\ell+1}+i-\tfrac{\ell+2}{2}\right]^2
+\frac{\ell+1-[\ell+1]}{(q-q^{-1})^2} \;.
\end{equation}
\end{prop}

\begin{proof}
The properties $\mathcal{C}_q=\mathcal{C}_q^*$ and $K_i\mathcal{C}_q=\mathcal{C}_qK_i$ are evident.
If further
\begin{equation}\label{eq:central}
[E_i,\mathcal{C}_q]=0\qquad\forall\;i=1,\ldots,\ell\;,
\end{equation}
then by adjunction  $\mathcal{C}_q$ commutes with all the generators of
$\Uq{\ell+1}$ and so it is central.

Using Lemmas \ref{lemmaA}-\ref{lemmaB} we get
\begin{align*} &
\left[E_i,\sum\nolimits_{j\leq k}q^{\ell+1-2j}M_{jk}^*N^2_{jk}M_{jk}\right] \\ &
\qquad =\sum\nolimits_{j\leq k}q^{\ell+1-2j}[E_i,M_{jk}^*]N^2_{jk}M_{jk}+
\sum\nolimits_{j\leq k}q^{\ell+1-2j}M_{jk}^*[E_i,N^2_{jk}M_{jk}] \\ &
\qquad =
\sum\nolimits_{j\leq i-1}q^{\ell+1-2j}M_{j,i-1}^*K_i^2N^2_{ji}M_{ji}
-\sum\nolimits_{k\geq i+1}q^{\ell+1-2j}K_i^{-2}M_{i+1,k}^*N^2_{ik}M_{ik}
\\ &\qquad
+\sum\nolimits_{k\geq i+1}q^{\ell-2i-1}qM_{i+1,k}^*N^2_{i+1,k}M_{i,k}
-\sum\nolimits_{j\leq i-1}q^{\ell+1-2j}M_{j,i-1}^*N^2_{j,i-1}M_{ji} \\ & \qquad
+\boxed{\frac{K_i^2-K_i^{-2}}{q-q^{-1}}q^{\ell+1-2i}N_{ii}^2E_i}\;.
\end{align*}
Since $K_i^2M_{i+1,k}^*=qM_{i+1,k}^*K_i^2$,
$N_{i,k}=K_iN_{i+1,k}$ and $N_{j,i-1}=K_iN_{j,i}$,
all the terms cancel but the framed one. Using (\ref{eq:NEN}) we get
\begin{align*}
\left[E_i,\sum\nolimits_{j=1}^\ell\tfrac{q^{\ell+2-2j}}{(q-q^{-1})^2}N_{j,j-1}^2\right]
&=\sum\nolimits_{j=1}^\ell\tfrac{q^{\ell+2-2j}}{(q-q^{-1})^2}
(q^{2(\delta_{i,j-1}-\delta_{i,j})}-1)N_{j,j-1}^2E_i \\
&=\begin{cases}
q^{\ell+1-2i}\frac{N_{i+1,i}^2-N_{i,i-1}^2}{q-q^{-1}}E_i
& \mathrm{if}\;i<\ell\;, \\
-\tfrac{q^{\ell+1-2i}}{q-q^{-1}}N_{i,i-1}^2E_i
& \mathrm{if}\;i=\ell\;.
\end{cases}
\end{align*}
Thus if $i<\ell$ last commutator cancel with the framed equation and we get
$[E_i,\mathcal{C}_q]=[E_i,A]$, with
$A:=q^{-\ell}(q-q^{-1})^{-2}\hat{K}^{-2}$,
while if $i=\ell$ we have
$$
[E_\ell,\mathcal{C}_q]=[E_\ell,A]
-K_\ell^{-2}(q-q^{-1})^{-1}q^{-\ell+1}
N_{\ell\ell}^2E_\ell \;.
$$
Observe that
$$
[E_i,\hat{K}^{-2}]=\delta_{i\ell}
q(q-q^{-1})\hat{K}^{-2}E_\ell \;,
$$
which implies that $[E_i,\mathcal{C}_q]=0$ for all $i$.
This concludes the first part of the proof.

By Schur's Lemma $\rho_n(\mathcal{C}_q)$ is proportional
to the identity. We can compute the proportionality constant
by applying it to the highest weight vector $v_n$. By construction
$\rho_n(M_{jk})v_n=0$, being annihilated by all $E_i$'s, and
$\rho_n(K_i)v_n=q^{n_i/2}v_n$, thus
$$
\rho_n(\mathcal{C}_q)=
\sum_{i=1}^{\ell+1}\tfrac{q^{\ell+2-2i}}{(q-q^{-1})^2}\,
q^{-\frac{2}{\ell+1}(\sum_{j=1}^{i-1}jn_j-\sum_{j=i}^\ell
(\ell+1-j)n_j)}-\tfrac{[\ell+1]}{(q-q^{-1})^2} \;,
$$
times the identity operator on $V_n$.
If we call $i'=\ell+2-i$ and $j'=\ell+1-j$, last equation can be rewritten as
$$
\rho_n(\mathcal{C}_q)=
\sum_{i'=1}^{\ell+1}\tfrac{q^{2i'-\ell-2}}{(q-q^{-1})^2}\,
q^{-\frac{2}{\ell+1}(\sum_{j'=i'}^\ell (\ell+1-j')n_{j'}-
\sum_{j'=1}^{i'-1}j'n_{j'})}-\tfrac{[\ell+1]}{(q-q^{-1})^2} \;,
$$
and the sum of last two equations gives
\begin{align*}
2\rho_n(\mathcal{C}_q)&=
\sum_{i=1}^{\ell+1}
\frac{
q^{\ell+2-2i-\frac{2}{\ell+1}(\sum_{j=1}^{i-1}jn_j-\sum_{j=i}^\ell(\ell+1-j)n_j)}
+q^{-\ell-2+2i+\frac{2}{\ell+1}(\sum_{j=1}^{i-1}jn_j-\sum_{j=i}^\ell(\ell+1-j)n_j)}
-2
}{(q-q^{-1})^2} \\ & \qquad\quad
+2\,\frac{\ell+1-[\ell+1]}{(q-q^{-1})^2} \\
&=
\sum_{i=1}^{\ell+1}
\left[
-\tfrac{\ell+2}{2}+i+\tfrac{\sum_{j=1}^{i-1}jn_j-\sum_{j=i}^\ell(\ell+1-j)n_j}{\ell+1}
\right]^2+2\,\frac{\ell+1-[\ell+1]}{(q-q^{-1})^2}
\;.
\end{align*}
This concludes the proof.
\end{proof}

\noindent
From Weyl's character formula \cite{IN66} we know that the multiplicity
of the eigenvalue (\ref{eq:SpCqa}) is
\begin{equation}\label{eq:Wchf}
\dim V_n=\frac{\prod_{1\leq r\leq s\leq\ell}\left(s-r+1
+\sum\nolimits_{i=r}^sn_i\right)}{\prod\nolimits_{r=1}^\ell r!} \;.
\end{equation}

We shall need later certain class of $V_n$, with $n=(n_1,0,0,\ldots,0,n_\ell)+\underline{e}_{\,k}$,
where $\underline{e}_{\,k}$ is the $\ell$-tuple with $k$-th component equal to one and all the others equal
to zero. That is, $n$ has components $n_i=n_1\delta_{i,1}+n_\ell\delta_{i,\ell}+\delta_{i,k}$, for
$i=1,\ldots,\ell$ and $k$ is a fixed number in $\{1,2,\ldots,\ell\}$.

\begin{lemma}\label{lemma:multeig}
For any $1\leq k\leq\ell$, the dimension of the irreducible representation
$V_n$ with highest weight $n_i=n_1\delta_{i,1}+n_\ell\delta_{i,\ell}+\delta_{i,k}$ is
\begin{equation}\label{eq:mult}
\dim V_n=\frac{k(n_1+n_\ell+\ell+1)}{(n_1+k)(n_\ell+\ell+1-k)}
\binom{n_1+\ell}{\ell}\binom{n_\ell+\ell}{\ell}\binom{\ell}{k} \;.
\end{equation}
The eigenvalue $\lambda_{n_1,n_\ell,N}$ of $\mathcal{C}_q$ in such a representation is
given by
\begin{equation}\label{eq:eigA}
2\lambda_{n_1,n_\ell,N}=
[n_1+k][n_1-\tfrac{2N}{\ell+1}+\ell+2-k]
+[n_\ell][n_\ell+\tfrac{2N}{\ell+1}+\ell]
+[\ell+1][\tfrac{N}{\ell+1}]^2 \;,
\end{equation}
where we call $N:=n_1-n_\ell+k$.
\end{lemma}

\begin{proof}
We divide the product in the numerator of (\ref{eq:Wchf}) in the following
cases
\begin{align*}
\{1\leq r\leq s\leq\ell\} &=
\{1<r\leq s<k\}
\cup\{k<r\leq s<\ell\}
\cup\{1<r\leq k\leq s<\ell\}\cup \\ &
\quad\cup\{1=r\leq s<\ell\}
\cup\{1<r\leq s=\ell\}
\cup\{r=1,s=\ell\} \;.
\end{align*}
With a simple computation we get
\begin{align*}
& \textstyle{\prod\limits_{1<r\leq s<k}}\Big(s-r+1+\textstyle{\sum\limits_{i=r}^s}\,n_i\Big)
   \!=\textstyle{\prod\limits_{r=1}^{k-2}}r! \;, &&
  \textstyle{\prod\limits_{k<r\leq s<\ell}}\Big(s-r+1+\textstyle{\sum\limits_{i=r}^s}\,n_i\Big)
   \!=\textstyle{\prod\limits_{r=1}^{\ell-k-1}}r! \;,\\
& \textstyle{\prod\limits_{1<r\leq k\leq s<\ell}}\Big(s-r+1+\textstyle{\sum\limits_{i=r}^s}\,n_i\Big)
   \!=\frac{\prod_{r=k}^{\ell-1}r!}{\prod_{r=1}^{\ell-k}r!} \;, &&
\textstyle{\prod\limits_{1=r\leq s<\ell}}\Big(s-r+1+\textstyle{\sum\limits_{i=r}^s}\,n_i\Big)
   \!=\frac{(n_1+\ell)!}{n_1!(n_1+k)} \;,\\
& \textstyle{\prod\limits_{1<r\leq s=\ell}}\Big(s-r+1+\textstyle{\sum\limits_{i=r}^s}\,n_i\Big)
   \!=\frac{(n_\ell+\ell)!}{n_\ell!(n_\ell+\ell+1-k)} \;, &&
\textstyle{\prod\limits_{r=1,s=\ell}}\Big(s-r+1+\textstyle{\sum\limits_{i=r}^s}\,n_i\Big)
   \!=n_1+n_\ell+\ell+1 \;,
\end{align*}
and plugging these in Weyl's character formula we get (\ref{eq:mult}).

The eigenvalue is given by (\ref{eq:SpCqa}):
\begin{align*}
2\lambda_{n_1,n_\ell,N} &=
\left[n_1-\tfrac{1}{\ell+1}(n_1-n_\ell+k)+\tfrac{\ell+2}{2}\right]^2
+\sum_{i=2}^k\left[\tfrac{1}{\ell+1}(n_1-n_\ell+k)+i-1-\tfrac{\ell+2}{2}\right]^2 \\
+\sum_{i=k+1}^{\ell}&\left[\tfrac{1}{\ell+1}(n_1-n_\ell+k)+i-\tfrac{\ell+2}{2}\right]^2
+\left[n_\ell+\tfrac{1}{\ell+1}(n_1-n_\ell+k)+\tfrac{\ell}{2}\right]^2
+2\frac{\ell+1-[\ell+1]}{(q-q^{-1})^2} \\
&=\left[n_1-\tfrac{1}{\ell+1}(n_1-n_\ell+k)+\tfrac{\ell+2}{2}\right]^2
-\left[\tfrac{1}{\ell+1}(n_1-n_\ell+k)+k-\tfrac{\ell+2}{2}\right]^2 \\ &\quad
+\left[n_\ell+\tfrac{1}{\ell+1}(n_1-n_\ell+k)+\tfrac{\ell}{2}\right]^2
-\left[\tfrac{1}{\ell+1}(n_1-n_\ell+k)+\tfrac{\ell}{2}\right]^2 \\ &\quad
+\sum_{i=1}^{\ell+1}\left[\tfrac{1}{\ell+1}(n_1-n_\ell+k)+i-\tfrac{\ell+2}{2}\right]^2 
+2\frac{\ell+1-[\ell+1]}{(q-q^{-1})^2} \;.
\end{align*}
But for all $t$,
$$
\sum_{i=1}^{\ell+1}\left[i-\tfrac{\ell+2}{2}+t\right]^2
=\frac{(q^{2t}+q^{-2t})[\ell+1]-2(\ell+1)}{(q-q^{-1})^2}=
[\ell+1][t]^2-2\,\frac{\ell+1-[\ell+1]}{(q-q^{-1})^2} \;.
$$
Thus, with the substitution $n_1-n_\ell+k=N$, we get
\begin{align*}
2\lambda_{n_1,n_\ell,N}
&=\left[n_1-\tfrac{N}{\ell+1}+\tfrac{\ell+2}{2}\right]^2
-\left[\tfrac{N}{\ell+1}+k-\tfrac{\ell+2}{2}\right]^2 \\ &\quad
+\left[n_\ell+\tfrac{N}{\ell+1}+\tfrac{\ell}{2}\right]^2
-\left[\tfrac{N}{\ell+1}+\tfrac{\ell}{2}\right]^2 \\ &\quad
+[\ell+1][\tfrac{N}{\ell+1}]^2 \;.
\end{align*}
Using in the first two lines the algebraic identity
$[x]^2-[y]^2=[x-y][x+y]$, we prove \eqref{eq:eigA}.
\end{proof}

\begin{lemma}
A Casimir operator $\mathcal{C}'_q$ for $\Uq{\ell}$ is defined by
\begin{equation}\label{eq:CqK}
\hat{K}^{\frac{2}{\ell}}\mathcal{C}'_q=
\sum_{i=1}^\ell\tfrac{q^{\ell+1-2i}}{(q-q^{-1})^2}N^2_{i,i-1}
+\sum_{1\leq j\leq k\leq\ell-1}q^{\ell-2j}M^*_{jk}N^2_{jk}M_{jk}
-\tfrac{[\ell]}{(q-q^{-1})^2}\hat{K}^{\frac{2}{\ell}} \;\,.
\end{equation}
Its spectrum is given by (\ref{eq:SpCqa}), with a replacement $\ell\to\ell-1$.
\end{lemma}

\begin{proof}
Take (\ref{eq:Cq}), replace $\ell$ with $\ell-1$, $\hat{K}$ with
$$
\hat{K}'=(K_1K_2^2\ldots K_{\ell-1}^{\ell-1})^{\frac{2}{\ell}}=
\hat{K}^{\frac{\ell+1}{\ell}}K_\ell^{-2} \;,
$$
$N_{j,k}$ with
$$
N'_{j,k}=(K_jK_{j+1}\ldots K_{\ell-1})\cdot(K_{k+1}K_{k+2}\ldots K_{\ell-1})\cdot\hat{K}'^{-1}
=\hat{K}^{-\frac{1}{\ell}}N_{j,k} \;,
$$
$M^*_{jk}$ with $M'^*_{jk}\equiv M^*_{jk}$ (for $k\leq\ell-1$),
obtaining the Casimir 
\begin{align*}
\mathcal{C}'_q &=\sum_{i=1}^{\ell-1}\tfrac{q^{\ell+1-2i}}{(q-q^{-1})^2}
\,N'^2_{i,i-1}+\tfrac{q^{-\ell+1}}{(q-q^{-1})^2}\,\hat{K}'^{-2}
+\sum_{1\leq j\leq k\leq\ell-1}q^{\ell-2j}M'^*_{jk}N'^2_{jk}M'_{jk}
-\tfrac{[\ell]}{(q-q^{-1})^2} \\
&=\hat{K}^{-\frac{2}{\ell}}\bigg(\sum_{i=1}^{\ell-1}\tfrac{q^{\ell+1-2i}}{(q-q^{-1})^2}
\,N^2_{i,i-1}+\tfrac{q^{-\ell+1}}{(q-q^{-1})^2}\,K_\ell^4\hat{K}^{-2}
+\!\!\sum_{1\leq j\leq k\leq\ell-1}\!\!q^{\ell-2j}M^*_{jk}N^2_{jk}M_{jk}\bigg)
-\tfrac{[\ell]}{(q-q^{-1})^2} \;.
\end{align*}
Since $K_\ell^4\hat{K}^{-2}=N_{\ell,\ell-1}^2$, the last equation is exactly (\ref{eq:CqK}).
\end{proof}

\smallskip

\noindent
The relation between the Casimir (\ref{eq:Cq}) of $\Uq{\ell+1}$ and the Casimir (\ref{eq:CqK}) of $\Uq{\ell}$ is
\begin{equation}\label{eq:CCprime}
\mathcal{C}_q=
q\hat{K}^{\frac{2}{\ell}}\Big(\mathcal{C}'_q+\tfrac{[\ell]}{(q-q^{-1})^2}\Big)
+\tfrac{q^{-\ell}}{(q-q^{-1})^2}\,\hat{K}^{-2}
+\sum\nolimits_{i=1}^\ell q^{\ell+1-2i}M_{i,\ell}^*N^2_{i,\ell}M_{i,\ell}
-\tfrac{[\ell+1]}{(q-q^{-1})^2} \;.
\end{equation}

\subsection{The quantum Grassmann algebra}\label{sec:3.2}
Irreducible representations of $\Uq{\ell}$ are labeled by
$n=(n_1,\ldots,n_{\ell-1})$, with $q^{\frac{1}{2}n_i}$ the eigenvalue of $K_i$ corresponding
to the highest weight vector (the vector that is annihilated by all $E_i$'s, $i=1,\ldots,\ell-1$).
These are all the finite-dimensional irreducible highest weight representations
such that $K_i$ are positive operators (thus, having a well-defined $q\to 1$ limit).
Before discussing representations, we need some preliminaries.

For a fixed $k=1,\ldots,\ell-1$, we denote $\Lambda_k$ the following
set of multi-indices (the set of $k$-partitions of $\ell$):
\begin{equation}\label{eq:ineq}
\Lambda_k:=\big\{\,
\underline{i}=(i_1,i_2,\ldots,i_k)\in\Z^k\,\big|\,
1\leq i_1<i_2<\ldots<i_k\leq\ell\,
\big\} \;.
\end{equation}
Let
\begin{equation}\label{eq:jdi}
j\#\underline{i}:=\sum\nolimits_{h=1}^k(\delta_{i_h,j}-\delta_{i_h,j+1})
\end{equation}
be the number of times $j$ appears in the string $\underline{i}$
minus the number of times $j+1$ appears in it. 
Due to the inequality in (\ref{eq:ineq}), this is either $0$ or $\pm 1$.

Given two multi-indices $\underline{i}'\in\Lambda_{k'}$ and
$\underline{i}''\in\Lambda_{k''}$ with empty intersection,
$\underline{i}'\cap\underline{i}''=\emptyset$
(that is $i'_r\neq i''_s\;\forall\;r,s$), we denote
$\underline{i}'\cup\underline{i}''\in\Lambda_{k'+k''}$
the ordered set with elements $\{i'_r,i'_s\}_{r,s}$.
In this case, for all $j$,
\begin{equation}\label{eq:dsum}
j\#(\underline{i}'\cup\underline{i}'')=
j\#\underline{i}'\,+\,j\#\underline{i}'' \;.
\end{equation}

For any $\underline{i}\in\Lambda_k$ there is only one 
$(\ell-k)$-tuple $\underline{i}^c=(i^c_1,\ldots,i^c_{\ell-k})\in\Lambda_{\ell-k}$,
such that $i_r\neq i^c_s\;\forall\;r,s$.
This is given by $\underline{i}^c=(1,2,\ldots,\ell)\smallsetminus\underline{i}$
(as ordered sets), and since
\begin{equation}
(\underline{i}^c)^c=\underline{i} \;,
\end{equation}
the map $\Lambda_k\to\Lambda_{\ell-k}$, $\underline{i}\mapsto\underline{i}^c$
is a bijection (in fact, an involution).
As $\underline{i}\cap\underline{i}^c=\emptyset$ and
$\underline{i}\cup\underline{i}^c=\{1,2,\ldots,\ell\}$, by (\ref{eq:dsum}) we have
\begin{equation}\label{eq:iic}
j\#\underline{i}\,+\,j\#\underline{i}^c=0 \;.
\end{equation}

If $j\#\underline{i}=+1$ it means that there is $r$ such that $i_r=j$
and $i_{r+1}>j+1$; in this case, we denote
\begin{subequations}\label{eq:jpm}
\begin{equation}
\underline{i}^{\,j,+}=(i_1,\ldots,i_{r-1},j+1,i_{r+1},\ldots,i_k) \;.
\end{equation}
If $j\#\underline{i}=-1$ it means that there is $r$ such that $i_r=j+1$
and $i_{r-1}<j$;
in this case, we denote
\begin{equation}
\underline{i}^{\,j,-}=(i_1,\ldots,i_{r-1},j,i_{r+1},\ldots,i_k) \;.
\end{equation}
\end{subequations}
Both $\underline{i}^{\,j,+}$ and $\underline{i}^{\,j,-}$ satisfy the inequality
in (\ref{eq:ineq}), so $\underline{i}^{\,j,+},\underline{i}^{\,j,-}\in\Lambda_k$.

\begin{lemma}\label{lemma:diesis}
For all $j,l,\underline{i}$ the following identities hold:
\begin{subequations}
\begin{align}
& l\#\underline{i}^{\,j,+}=l\#\underline{i}-2\delta_{j,l}+\delta_{j,l+1}+\delta_{j,l-1}
&& \mathrm{whenever}\;j\#\underline{i}=1 \label{eq:diesisP} \;,\\
\rule{0pt}{16pt}
& \delta_{j\#\underline{i},+1}\delta_{l\#\underline{i}^{\,j,+},-1}-
  \delta_{l\#\underline{i},-1}\delta_{j\#\underline{i}^{\,l,-},+1}=
  \delta_{j,l}\,j\#\underline{i} \label{eq:diesisA} \;,\\
\rule{0pt}{16pt}
& (\underline{i}^{j,+})^c=(\underline{i}^c)^{j,-}
&& \mathrm{whenever}\;j\#\underline{i}=1 \label{eq:diesisB} \;.
\end{align}
\end{subequations}
\end{lemma}
\begin{proof}
If $j\#\underline{i}=+1$ (resp. $-1$) and $r$ is the integer
such that $i_r=j$ (resp. $j+1$), then
\begin{align*}
l\#\underline{i}-l\#\underline{i}^{\,j,+} &=\sum\nolimits_{h=1}^k
(\delta_{l,i_h}-\delta_{l+1,i_h}-\delta_{l,i_h+\delta_{r,h}}+\delta_{l+1,i_h+\delta_{r,h}})
=2\delta_{j,l}-\delta_{j,l+1}-\delta_{j,l-1} \;,\\
l\#\underline{i}-l\#\underline{i}^{\,j,-} &=\sum\nolimits_{h=1}^k
(\delta_{l,i_h}-\delta_{l+1,i_h}-\delta_{l,i_h-\delta_{r,h}}+\delta_{l+1,i_h-\delta_{r,h}})
=-2\delta_{j,l}+\delta_{j,l+1}+\delta_{j,l-1} \;.
\end{align*}
The former equation is just (\ref{eq:diesisP}).
Next, if $|j-l|=1$ we have
$$
\delta_{j\#\underline{i},+1}\delta_{l\#\underline{i}^{\,j,+},-1}
=\delta_{j\#\underline{i},+1}\delta_{l\#\underline{i}+1,-1}
$$
and this is zero since $l\#\underline{i}$ can never be $-2$.
Similarly $\delta_{l\#\underline{i},-1}\delta_{j\#\underline{i}^{\,l,-},+1}=0$.

If $|j-l|>1$, then $l\#\underline{i}^{\,j,+}=l\#\underline{i}$,
$j\#\underline{i}^{\,l,-}=j\#\underline{i}$, and
$$
\delta_{j\#\underline{i},+1}\delta_{l\#\underline{i}^{\,j,+},-1}
=\delta_{j\#\underline{i},+1}\delta_{l\#\underline{i},-1}
=\delta_{j\#\underline{i}^{\,l,-},+1}\delta_{l\#\underline{i},-1} \;.
$$
If $j=l$ then $\delta_{j\#\underline{i},+1}\delta_{l\#\underline{i}^{\,j,+},-1}=
\delta_{j\#\underline{i},+1}$,
$\delta_{l\#\underline{i},-1}\delta_{j\#\underline{i}^{\,l,-},+1}=-
\delta_{j\#\underline{i},-1}$, and their difference is
$$
\delta_{j\#\underline{i},+1}-\delta_{j\#\underline{i},-1}=j\#\underline{i} \;.
$$
This proves (\ref{eq:diesisA}).

We pass to (\ref{eq:diesisB}). We assume $j\#\underline{i}=1$, that by (\ref{eq:iic})
is equivalent to the condition $j\#\underline{i}^c=-1$. We prove the equation by
induction on the length $k$ of $\underline{i}$. For $k=1$ we have:
$$
\underline{i}=(j)\;,\qquad
\underline{i}^{j,+}=(j+1)\;,\qquad
\underline{i}^c=(1,2,\ldots,j-1,j+1,j+2,\ldots,\ell)\;,
$$
and
$$
(\underline{i}^c)^{j,-}=(1,\ldots,j-1,j,j+2,\ldots,\ell)=
(\underline{i}^{j,+})^c\;.
$$
Suppose now
$(\underline{i}^{j,+})^c=(\underline{i}^c)^{j,-}$ is true for any
$\underline{i}\in\Lambda_k$ with $j\#\underline{i}=1$ (for a fixed
$k\geq 1$). Let $\underline{i}'\in\Lambda_{k+1}$ with $j\#\underline{i}'=1$.
We write $\underline{i}'=(\underline{i},i'_{k+1})$ if $i'_{k+1}\neq j$ and
$\underline{i}'=(i'_1,\underline{i})$ if $i'_1\neq j$. In both cases
$\underline{i}\in\Lambda_k$ has $j\#\underline{i}=1$. We consider the former
case, the latter being symmetric. We have
$$
\underline{i}'^{\,j,+}=(\underline{i}^{j,+},i'_{k+1}) \;,
$$
and (as ordered sets)
$$
\underline{i}'^{\,c}=(1,2,\ldots,\ell)\smallsetminus
(\underline{i},i'_{k+1})=\underline{i}^c\smallsetminus (i'_{k+1}) \;,
$$
and being $i'_{k+1}\neq j+1$ (as $j\#\underline{i}'=1$), by inductive hypothesis
$$
(\underline{i}'^{\,c})^{j,-}=(\underline{i}^c)^{j,-}\smallsetminus (i'_{k+1})
=(\underline{i}^{j,+})^c\smallsetminus (i'_{k+1})=
(\underline{i}^{j,+},i'_{k+1})^c=(\underline{i}'^{\,j,+})^c \;.
$$
This concludes the proof.
\end{proof}

\smallskip

\noindent
We denote by $W_k\simeq\C^{\binom{\ell}{k}}$ the linear space of vectors
$w=(w_{\underline{i}})_{\underline{i}\in\Lambda_k}$ with
components $w_{\underline{i}}\in\C$ labeled by the above multi-index.
A transformation $T\in\mathrm{End}(W_k)$ sends a vector $w$
into the vector $Tw$ with components $(Tw)_{\underline{i}}$.

\begin{prop}
Let $1\leq k\leq\ell-1$. The irreducible $*$-representation of $\Uq{\ell}$ with highest weight
$$
\delta^k=(\,\stackrel{k-1\;\mathrm{times}}{\overbrace{0,\ldots,0}},1,\!\!\!\stackrel{\ell-k-1\;\mathrm{times}}{\overbrace{0,\ldots,0}}\!\!) \;,
$$
is given explicitly by the map $\,\sigma_k:\Uq{\ell}\to\mathrm{End}(W_k)$ defined on
generators by
\begin{subequations}
\begin{align}
\bigl\{\sigma_k(K_j)w\bigr\}_{\underline{i}}&=q^{\frac{1}{2}j\#\underline{i}}\,w_{\,\underline{i}}\;,\label{ar:K} \\
\bigl\{\sigma_k(E_j)w\bigr\}_{\underline{i}}&=\delta_{j\#\underline{i},+1}\,w_{\,\underline{i}^{\,j,+}}\;,\label{ar:E} \\
\bigl\{\sigma_k(F_j)w\bigr\}_{\underline{i}}&=\delta_{j\#\underline{i},-1}\,w_{\,\underline{i}^{\,j,-}}\;,\label{ar:F}
\end{align}
\end{subequations}
where $j\#\underline{i}$ is defined in (\ref{eq:jdi}) and
$\underline{i}^{\,j,\pm}$ are defined in (\ref{eq:jpm}).

For $k=0$ or $\ell$, we denote $W_0=W_\ell=\C$ the vector space underlying the trivial representation
(i.e.~$\sigma_0=\sigma_\ell=\epsilon$ is the counit of $\Uq{\ell}$).
\end{prop}

\begin{proof}
The vector $w$ with $w_{(1,2,\ldots,k)}=1$ and all other components equal to zero is annihilated
by all the $E_i$'s. Since $\sigma_k(K_j)w=q^{\delta_{j,k}/2}w$, the vector $w$
is the highest weight vector of the irreducible representation with highest weight $\delta^k$, that is
then a subrepresentation of the (would-be) representation $\sigma_k$. Having both the same dimension
$\binom{\ell}{k}$ (by Weyl's character formula), they coincide if $\sigma_k$ is a representation.
Since $\sigma_k(K_j)$ is Hermitian and the Hermitian conjugated of $\sigma_k(E_j)$
is $\sigma_k(F_j)$, the would-be representation is unitary.
We now prove that the defining relations of $\Uq{\ell}$ are satisfied (being $\sigma_k(h)^\dag=\sigma_k(h^*)$,
if a relation holds, its conjugated holds too).
We omit the representation symbol $\sigma_k$.

That $K_j$'s commute each other is trivial, being all diagonal. Further,
by (\ref{eq:diesisP}):
\begin{align*}
\{K_lE_jK_l^{-1}w\}_{\underline{i}} &=
\delta_{j\#\underline{i},+1}\,q^{\frac{1}{2}(l\#\underline{i}-l\#\underline{i}^{\,j,+})}w_{\,\underline{i}^{\,j,+}} \\
&=q^{\delta_{j,l}-\frac{1}{2}(\delta_{j,l+1}+\delta_{j,l-1})}\delta_{j\#\underline{i},+1}\,w_{\,\underline{i}^{\,j,+}} \\
&=\{q^{\delta_{j,l}-\frac{1}{2}(\delta_{j,l+1}+\delta_{j,l-1})}E_jw\}_{\underline{i}} \;.
\end{align*}
This means $K_jE_jK_j^{-1}=qE_j$, $K_lE_jK_l^{-1}=q^{-\frac{1}{2}}E_j$ if $|j-l|=1$
and $K_lE_jK_l^{-1}=E_j$ otherwise.
Next, by by (\ref{eq:diesisA}):
\begin{align*}
\{[E_j,F_l]w\}_{\underline{i}} &=
\big(
\delta_{j\#\underline{i},+1}\delta_{l\#\underline{i}^{\,j,+},-1}-
\delta_{l\#\underline{i},-1}\delta_{j\#\underline{i}^{\,l,-},+1}
\big)w_{\underline{i}} \\
&=\delta_{j,l}\,(j\#\underline{i})\,w_{\underline{i}}
=\delta_{j,l}\,[j\#\underline{i}]\,w_{\underline{i}}
=\Big\{\tfrac{K_j^2-K_j^{-2}}{q-q^{-1}}\,w\Big\}_{\underline{i}}\;,
\end{align*}
where we used the fact that being $j\#\underline{i}\in\{0,\pm 1\}$, it is always
$[j\#\underline{i}]=j\#\underline{i}\,$.

Concerning Serre's relation, it is enough to show that
$E_j^2=E_jE_{j\pm 1}E_j=0$ for all $j$. We have
$$
\{E_j^2w\}_{\underline{i}}=
\delta_{j\#\underline{i},+1}\delta_{j\#\underline{i}^{\,j,+},+1}
\,w_{(\underline{i}^{j,+})^{l,+}} \;,
$$
and this is zero since by (\ref{eq:diesisP}) $j\#\underline{i}^{\,j,+}=j\#\underline{i}-2=-1$
whenever $j\#\underline{i}=1$. Again by (\ref{eq:diesisP}), for $l=j\pm 1$ we have
$$
j\#\underline{i}=1\;\;\wedge\;\;
l\#\underline{i}^{\,j,+}=1
\quad\Rightarrow\quad
j\#(\underline{i}^{\,j,+})^{\,l,+}=
j\#\underline{i}^{\,j,+}+1=
j\#\underline{i}-1=0 \;.
$$
So for $l=j\pm 1$,
$$
\{E_jE_lE_jw\}_{\underline{i}}=
\delta_{j\#\underline{i},+1}\delta_{l\#\underline{i}^{\,j,+},+1}
\delta_{j\#(\underline{i}^{\,j,+})^{\,l,+},+1}
w_{((\underline{i}^{j,+})^{l,+})^{j,+}}=0\;.
$$
This concludes the proof.
\end{proof}

For $q=1$, we have $W_k\simeq\wedge^kW_1$ for all $0\leq k\leq\ell$. We are going to define a deformation
$\wprod$ of the wedge product such that $\wprod:W_h\otimes W_k\to W_{h+k}$
intertwines the Hopf tensor product of the representations $\sigma_h$ and $\sigma_k$
with the representation $\sigma_{h+k}$, where $W_{h+k}:=0$ if $h+k>\ell$.
Before that, we need a brief interlude on permutations. We take notations from \cite{BB05}.

Let $S_n$ be the group of permutations of the set with $n$ elements
$\{1,2,\ldots,n\}$; if $p\in S_n$, we denote $p(i)$ 
the image of $i\in\{1,\ldots,n\}$ through $p$, and with $p^i$
the position of $i$ after the permutation (we use the superscript to avoid
confusion, since in \cite{BB05} they call $p_i=p(i)$). For example,
let $n=3$; if $p$ is the permutation
$$
p=\footnotesize\left(\!\begin{array}{ccc}1 & 2 & 3 \\ 3 & 1 & 2\end{array}\!\right)
$$
then $(p(1),p(2),p(3))=(3,1,2)$, while $p^1=2$, $p^2=3$ and $p^3=1$.
A generic permutation is
$$
p=\footnotesize\left(\!\begin{array}{cccc}1 & 2 & \ldots & n \\ p(1) & p(2) & \ldots & p(n)\end{array}\!\right)
$$
and the relation between $p^i$ and $p(i)$ is the following: $p^i=p^{-1}(i)$
where $p^{-1}$ is the inverse of $p$ in $S_n$ ($p^i$ is the new position of $i$,
$p(p^i)$ is the number in position $p^i$, then $p(p^i)=i$).
In the above example, one can easily check that
$$
\footnotesize\left(\!\begin{array}{ccc}1 & 2 & 3 \\ p^1 & p^2 & p^3\end{array}\!\right)
\normalsize =
\footnotesize\left(\!\begin{array}{ccc}1 & 2 & 3 \\ 2 & 3 & 1\end{array}\!\right)
$$
is the inverse of $p$. The product of two permutations $p\cdot p'$ is their composition,
where as usual the permutation on the right act first: thus $pp'$ sends $i$ into $p(p'(i))$.
For example if
$$
p=\footnotesize\left(\!\begin{array}{ccc}1 & 2 & 3 \\ 2 & 1 & 3\end{array}\!\right)
\normalsize \;,\qquad
p'=\footnotesize\left(\!\begin{array}{ccc}1 & 2 & 3 \\ 1 & 3 & 2\end{array}\!\right)
\normalsize \;,
$$
then $p'$ sends $1\mapsto 1$, $p$ sends $1\mapsto 2$, so $pp'$ sends $1\mapsto 2$, etc.;
on the other hand, $p$ sends $1\mapsto 2$, $p'$ sends $2\mapsto 3$, so $p'p$ sends
$1\mapsto 3$, etc.; the result is that
$$
pp'=\footnotesize\left(\!\begin{array}{ccc}1 & 2 & 3 \\ 2 & 3 & 1\end{array}\!\right)
\normalsize \;,\qquad
p'p=\footnotesize\left(\!\begin{array}{ccc}1 & 2 & 3 \\ 3 & 1 & 2\end{array}\!\right)
\normalsize \;.
$$

A representation $\pi:S_n\to\mathrm{End}(\Z^n)$ is given by
$\pi(p)(\underline{i}):=(i_{p^1},i_{p^2},\ldots,i_{p^n})$ for all
$p\in S_n$ and any $n$-tuple $\underline{i}=(i_1,i_2,\ldots,i_n)\in\Z^n$.
Notice that $\pi(p)$ sends $\Lambda_n\subset\Z^n$ into $\Lambda_n$
only if $p$ is the trivial permutation. We call
\begin{equation}\label{eq:repSn}
\underline{i}_{\,\pi(p)}:=(i_{p(1)},i_{p(2)},\ldots,i_{p(n)})=
\pi(p^{-1})(\underline{i}) \;.
\end{equation}
With this notation $(\underline{i}_{\,\pi(p)})_{\pi(p')}=\underline{i}_{\,\pi(pp')}$.

Recall that $S_n$ is generated by \emph{simple transpositions}, that
are the transformations that exchange two consecutive elements and keep
all others fixed. The \emph{length} $||p||$ of a permutation $p$ is the
minimal number of simple transpositions needed to express $p$.
The length of $p$ coincides with the \emph{number of inversions} in $p$,
that is the number of pairs $(i,j)$ such that $1\leq i<j\leq n$ and
$p(i)>p(j)$ (cf.~\cite[Proposition 1.5.2]{BB05}):
$$
||p||=\mathrm{card}\{(i,j)\,:\,i<j\;\mathrm{and}\;p(i)>p(j)\} \;.
$$
This can be written also as
$$
||p||=\mathrm{card}\{(i,j)\,:\,i<j\;\mathrm{and}\;p^i>p^j\} \;,
$$
since $p$ and $p^{-1}$ have always the same length (if $p=ss'\ldots s''$
is a reduced expression of $p$, then $p^{-1}=s''\ldots s's$ is a
reduced expression of $p^{-1}$).

Given a maximal parabolic subgroup $S_k\times S_{n-k}$ of
$S_n$ we can consider the left cosets
$$
S^{(k)}_n:=S_n/(S_k\times S_{n-k}) \;.
$$
The quotient $S^{(k)}_n$ coincides with the set of permutations (cf.~\cite{BB05}, Sec.~2.4)
$$
S^{(k)}_n=\{p\in S_n\,|\,p(1)<p(2)<\ldots<p(k)\;\mathrm{and}\;p(k+1)<p(k+2)<\ldots<p(n) \}\;.
$$
The map $p\mapsto p^{-1}$ gives a bijection between $S^{(k)}_n$ and the right cosets
$S_{k,n-k}:=(S_k\times S_{n-k})\backslash S_n$, that is then given by those permutations
$p$ that satisfy $p^1<p^2<\ldots<p^k$ and $p^{k+1}<p^{k+2}<\ldots<p^n$.
Elements of $S_{k,n-k}$ are called \emph{$(k,n-k)$-shuffles} (see e.g.~\cite{KS97}, Sec.~13.2).

We'll need the following lemma (cf.~\cite{BB05}, Proposition 2.4.4, Cor.~2.4.5 and 2.12).

\begin{lemma}\label{lemma:3.9}
Any $p\in S_n$ can be uniquely factorized 
$p=p'p''$ with $p'\in S^{(k)}_n$ and $p''\in S_k\times S_{n-k}$,
or as $p=\tilde{p}'\tilde{p}''$ with $\tilde{p}'\in S_k\times S_{n-k}$ and $\tilde{p}''\in S_{k,n-k}$,
and in both cases
$$
||p||=||p'||+||p''||
=||\tilde{p}'||+||\tilde{p}''|| \;.
$$
\end{lemma}

Elements of these two quotients correspond to maps $\Lambda_{h+k}\to\Lambda_h\times\Lambda_k$:
$\forall\;\underline{i}\in\Lambda_{h+k}$ we have
$\{\pi(p)(\underline{i})\in\Lambda_h\times\Lambda_k\iff p\in S_{h,k}\}$
and
$\{\underline{i}_{\,\pi(p)}\in\Lambda_h\times\Lambda_k\iff p\in S^{(h)}_{h+k}\}$.
With this preparation, we define $\wprod:W_h\otimes W_k\to W_{h+k}$ by the formula
\begin{equation}\label{eq:wprod}
(v\wprod w)_{\underline{i}}=\sum\nolimits_{p\in S^{(h)}_{h+k}}(-q^{-1})^{||p||}\,
v_{i_{p(1)},\ldots,i_{p(h)}}\,w_{i_{p(h+1)},\ldots,i_{p(h+k)}}
\end{equation}
for all $v=(v_{\underline{i}'})\in W_h$ and $w=(w_{\underline{i}''})\in W_k$,
and for all $h,k=0,\ldots,\ell$ with $h+k\leq\ell$. Moreover, we set $v\wprod w:=0$ if $h+k>\ell$.
By definition of $S^{(h)}_{h+k}$, this map is well defined.

Remark: when there is no risk of confusion, we write $i_1,\ldots,i_k$ without
parenthesis instead of $(i_1,\ldots,i_k)$. For example in the equation above,
$v_{i_{p(1)},\ldots,i_{p(h)}}$ means $v_{(i_{p(1)},\ldots,i_{p(h)})}$; also
$\underline{i}\smallsetminus j$ means $\underline{i}\smallsetminus (j)$.

\begin{prop}\label{prop:grass}
The above product satisfies the following properties:
\begin{enumerate}
\item $\wprod$ is surjective. \label{enu:one}
\item $\wprod$ is associative. \label{enu:two}
\item $\wprod$ is a left $\Uq{\ell}$-module map, that is \label{enu:three}
\begin{equation}\label{eq:wedgecov}
\sigma_{h+k}(x)(v\wprod w)=
\{\sigma_h(x_{(1)})v\}\wprod\{\sigma_k(x_{(2)})w\}
\end{equation}
for all $v\in W_h$, $w\in W_k$ and for all $x\in\Uq{\ell}$.
Here we use Sweedler-type notation, $\Delta(x)=x_{(1)}\otimes x_{(2)}$,
for the coproduct.
\end{enumerate}
\end{prop}

\begin{proof}
The map $\wprod$ is clearly \emph{surjective}, since
for all $\underline{i}\in\Lambda_n$, $n=h+k$, we can find $v\in W_h$ and $w\in W_k$
such that $(v\wprod w)_{\underline{i}}\neq 0$. For instance, take as $v$ the vector with $v_{i_1,...,i_h}=1$
and all other components zero, and as $w$ the vector with $w_{i_{h+1},...,i_{h+k}}=1$
and all other components zero; this yields $(v\wprod w)_{\underline{i}}=v_{i_1,...,i_h}
w_{i_{h+1},...,i_{h+k}}=1$, which shows point \ref{enu:one}.

\medskip

Concerning point \ref{enu:two}, we need to prove that
for any $k\in\{3,4,\ldots,\ell\}$, no matter how we parenthesize the product $v^1\wprod v^2\wprod ...\wprod v^k$
of vectors in $W_1$, the result will be the same, and is given by the formula
\begin{equation}\label{eq:wprodind}
(v^1\wprod v^2\wprod\ldots\wprod v^k)_{\underline{i}}=\sum_{p\in S_k}(-q^{-1})^{||p||}v^1_{i_{p(1)}}
v^2_{i_{p(2)}}\ldots v^k_{i_{p(k)}} \;.
\end{equation}
We prove (\ref{eq:wprodind}) by induction on $k$. It's easy to check when $k=3$:
\begin{align*}
\{(v^1\wprod v^2)&\wprod v^3\}_{i_1,i_2,i_3}=\{v^1\wprod (v^2\wprod v^3\}_{i_1,i_2,i_3} \\
&=v^1_1v^2_2v^3_3
-q^{-1}v^1_2v^2_1v^3_3
-q^{-1}v^1_1v^2_3v^3_2
+q^{-2}v^1_2v^2_3v^3_1
+q^{-2}v^1_3v^2_1v^3_2
-q^{-3}v^1_3v^2_2v^3_1 \;.
\end{align*}
Now we assume the claim is true for a generic $k\geq 3$, and prove that
$(v^1\wprod ...\wprod v^k)\wprod v^{k+1}$ is equal to
$v^1\wprod (v^2\wprod ...\wprod v^{k+1})$ and given by
the expression (\ref{eq:wprodind}) (with $k$ replaced by $k+1$).
We have
\begin{multline*}
\{(v^1\wprod ...\wprod v^k)\wprod v^{k+1}\}_{\underline{i}}\,
=\sum_{p'\in S^{(k)}_{k+1}}(-q^{-1})^{||p'||}
(v^1\wprod ...\wprod v^k)_{i_{p'(1)},\ldots,i_{p'(k)}}
v^{k+1}_{i_{p'(k+1)}} \\
=\!\sum_{p'\in S^{(k)}_{k+1},p''\in S_k}\!(-q^{-1})^{||p'||+||p''||}v^1_{i_{p'(p''(1))}}
v^2_{i_{p'(p''(2))}}\ldots v^k_{i_{p'(p''(k))}} 
v^{k+1}_{i_{p'(k+1)}} \,.
\end{multline*}
The composition $p=p'\circ (p''\times id)$ is clearly an element
of $S_{k+1}$. But also the converse is true: by Lemma
\ref{lemma:3.9} for any $p\in S_{k+1}$
there is a unique decomposition $p=p'\circ (p''\times id)$
with $p'\in S^{(k)}_{k+1}$ and $p''\times id\in S_k\times S_1$
($S_1=\{id\}$ is the trivial group), and satisfies
$||p||=||p'||+||p''||$. Thus,
$$
\{(v^1\wprod ...\wprod v^k)\wprod v^{k+1}\}_{\underline{i}}
=\sum_{p\in S_{k+1}}(-1)^{||p||}\,v^1_{i_{p(1)}}
v^2_{i_{p(2)}}\ldots v^{k+1}_{i_{p(k+1)}} \;.
$$
The proof can be mirrored: we get the analogous claim for
$v^1\wprod (v^2\wprod ...\wprod v^{k+1})$ by using the decomposition
$S_{k+1}=S^{(1)}_{k+1}\cdot (S_1\times S_k)$. This proves the
inductive step.

Now, the map $\wprod$ being surjective, any $w\in W_k$, $w'\in W_{k'}$
and $w''\in W_{k''}$ can be written as
$$
w=v^1\wprod\ldots\wprod v^k \;,\qquad
w'=v^{k+1}\wprod\ldots\wprod v^{k+k'} \;,\qquad
w''=v^{k+k'+1}\wprod\ldots\wprod v^{k+k'+k''} \;,
$$
with $v^j\in W_1$. Since the product of any number of vectors
in $W_1$ is associative, we have
$$
(w\wprod w')\wprod w''=
v^1\wprod\ldots\wprod v^{k+k'+k''}
=w\wprod (w'\wprod w'')
$$
and this concludes the proof of associativity.

\medskip

We pass to point \ref{enu:three}, the claim that $\wprod:W_h\otimes W_k\to W_{h+k}$
intertwines the Hopf tensor product of the representations $\sigma_h$ and $\sigma_k$
with the representation $\sigma_{h+k}$. It is enough to prove that
\begin{equation}\label{eq:intcov}
\sigma_n(x)(v^1\wprod v^2\wprod\ldots\wprod v^n)=
\{\sigma_1(x_{(1)})v^1\}\wprod
\{\sigma_1(x_{(2)})v^2\}\wprod\ldots\wprod
\{\sigma_1(x_{(n)})v^n\} \;,
\end{equation}
for any number $n$ of vectors $v^j\in W_1$. Equation (\ref{eq:intcov}), together
with associativity and surjectivity of $\wprod$, implies (\ref{eq:wedgecov}).
Indeed, any $v\in W_h$ (resp.~$w\in W_k$) can be written as products $v=v^1\wprod\ldots\wprod v^h$
(resp.~$w=v^{h+1}\wprod\ldots\wprod v^{h+k}$) of vectors $v^j\in W_1$, and using
(\ref{eq:intcov}) and coassociativity of the coproduct we get
\begin{align*}
\sigma_{h+k}(x)(v\wprod w)
&=\sigma_{h+k}(x)(v^1\wprod\ldots\wprod v^{h+k}) \\
&=\{\sigma_1(x_{(1)})v^1\}\wprod
\{\sigma_1(x_{(2)})v^2\}\wprod\ldots\wprod
\{\sigma_1(x_{(n)})v^{h+k}\} \\
&=\bigl\{\sigma_h(x_{(1)})(v^1\wprod\ldots\wprod v^h)\bigr\}
\wprod
\bigl\{\sigma_k(x_{(2)})(v^{h+1}\wprod\ldots\wprod v^{h+k})\bigr\} \\
&=\{\sigma_h(x_{(1)})v\}\wprod\{\sigma_k(x_{(2)})w\} \;.
\end{align*}
Last step is to prove (\ref{eq:intcov}). Dealing with $*$-representa\-tions,
it is enough to do the check for $x=K_j,E_j$ (and for all $j=1,\ldots,\ell-1$).
For $x=K_j$, the property (\ref{eq:intcov}) follows from the simple
observation that, by (\ref{eq:jdi}),
$j\#\underline{i}=\sum\nolimits_{r=1}^n j\#(i_r)$ for all
$\underline{i}\in\Lambda_n$.
For $x=E_j$, the $n$th power of the coproduct is
$$
\Delta_n(E_j)=\sum\nolimits_{r=1}^n\,(K_j^{-1})^{\otimes (r-1)}\otimes E_j\otimes K_j^{\otimes (n-r)}
$$
and the right hand side of (\ref{eq:intcov}) becomes
$$
\mathrm{rhs}=\sum_{r=1}^n
\sigma_1(K_j^{-1})v^1\wprod\ldots\wprod
\sigma_1(K_j^{-1})v^{r-1}\wprod
\sigma_1(E_j)v^{r}\wprod
\sigma_1(K_j)v^{r+1}\wprod\ldots\wprod
\sigma_1(K_j)v^n \;.
$$
By (\ref{ar:K}) and (\ref{ar:E}) we have $\{\sigma_1(K_j)v\}_i=q^{\frac{1}{2}(\delta_{i,j}-\delta_{i,j+1})}v_i$
and $\{\sigma_1(E_j)v\}_i=\delta_{i,j}v_{i+1}$ for all $v\in W_1$. Thus, using (\ref{eq:wprodind})
we rewrite the $\underline{i}$th component of previous equation as
$$
(\mathrm{rhs})_{\underline{i}}=\sum_{r=1}^n
\sum_{p\in S_n|i_{p(r)}=j}\!\!\!(-q^{-1})^{||p||}
q^{\frac{1}{2}\left(\sum_{s<r}-\sum_{s>r}\right)\delta_{i_{p(s)},j+1}}
v^1_{i_{p(1)}}\ldots v^{r-1}_{i_{p(r-1)}}v^r_{j+1}v^{r+1}_{i_{p(r+1)}}\ldots v^n_{i_{p(n)}} \;.
$$
Being $p$ a permutation, $j\in\underline{i}_{\,\pi(p)}$
if and only if $j\in\underline{i}$. We have three cases:
1) $j\notin\underline{i}$,
2) $j,j+1\in\underline{i}$ and
3) $j\in\underline{i}$ but $j+1\notin\underline{i}$; the first
two cases correspond to $j\#\underline{i}\neq 1$, the third
to $j\#\underline{i}=1$.

If $j\notin\underline{i}$ the equation is trivially zero (the sum is
empty). In the second case we get zero too, but to prove it requires
some work. Firstly, notice that if $r,s$ are the integer such that
$i_r=j$ and $i_s=j+1$, then $s=r+1$ (by contradiction, assume there
exists $h$ such that $r<h<s$, then there is also $i_h$ such that $j<i_h<j+1$,
that is a contradiction being $i_h$ an integer). Assume then
that $\underline{i}=(i_1,\ldots,i_{r-1},j,j+1,j_{r+2},\ldots,i_n)$,
and call
$$
A=\{p\in S_n\,|\,p^r<p^{r+1}\}\;,\qquad
B=\{p\in S_n\,|\,p^r>p^{r+1}\}\;.
$$
We have
\begin{align*}
(\mathrm{rhs})_{\underline{i}}
=\sum\nolimits_{p\in A}(-q^{-1})^{||p||}q^{-\frac{1}{2}}
v^1_{i_{p(1)}}\ldots
v^{p^r-1}_{i_{p(p^r-1)}}v^{p^r}_{j+1}v^{p^r+1}_{i_{p(p^r+1)}}\ldots
v^{p^{r+1}-1}_{i_{p(p^{r+1}-1)}}v^{p^{r+1}}_{j+1}v^{p^{r+1}+1}_{i_{p(p^{r+1}+1)}}\ldots
v^n_{i_{p(n)}} \notag \\
 +\sum\nolimits_{p'\in B}(-q^{-1})^{||p'||}q^{\frac{1}{2}}
v^1_{i_{p'(1)}}\ldots
v^{p'^{r+1}-1}_{i_{p'(p'^{r+1}-1)}}v^{p'^{r+1}}_{j+1}v^{p'^{r+1}+1}_{i_{p'(p'^{r+1}+1)}}\ldots
v^{p'^r-1}_{i_{p'(p'^r-1)}}v^{p'^r}_{j+1}v^{p'^r+1}_{i_{p'(p'^r+1)}}\ldots
v^n_{i_{p'(n)}} \;.
\end{align*}
The effect of the composition $p^{-1}\mapsto p^{-1}\circ s_r$,
with $s_r$ the simple transposition exchanging $r$ with $r+1$, is to
exchange $p^{-1}(r)$ with $p^{-1}(r+1)$ in the complete expression of $p^{-1}$
(cf.~\cite{BB05}, Pag.~20, with $x=p^{-1}$). But this is equivalent to
the transformation $p\mapsto s_r\circ p$ (as $s_r^2=1$), whose effect
is then to exchange $p^r$ with $p^{r+1}$ (as $p^j=p^{-1}(j)$), thus
giving a bijection $A\to B$.
The change of variable $p'=s_r\circ p$ turns the second sum in last equation
into the first, but for a global sign (by~\cite[1.26]{BB05}, $||p'^{-1}||=||p^{-1}||+1$,
and then $||p'||=||p||+1$, for all $p'\in B$). Hence, the two sums cancel and the result is zero.
Therefore, $(\mathrm{rhs})_{\underline{i}}$ is zero unless $j\#\underline{i}=1$.
We have
\begin{align*}
(\mathrm{rhs})_{\underline{i}} &=\delta_{j\#\underline{i},+1}\sum_{p\in S_n}(-q^{-1})^{||p||}
\Big[v^1_{i_{p(1)}}\ldots v^{x-1}_{i_{p(x-1)}}v^x_{j+1}v^{x+1}_{i_{p(x+1)}}\ldots v^n_{i_{p(n)}}
\Big]_{x:\,i_{p(x)}=j} \\
&=\delta_{j\#\underline{i},+1}(v^1\wprod v^2\wprod\ldots\wprod v^n)_{\underline{i}^{\,j,+}} \;,
\end{align*}
that by (\ref{ar:E}) is exactly the $\underline{i}$th component of the
left hand side of (\ref{eq:intcov}). This concludes the proof.
\end{proof}

We set $\Gr:=\oplus_{k=0}^\ell W_k$, equipped with $\wprod$, that by Proposition \ref{prop:grass} is a
graded associative algebra -- generated by $W_1$ -- and a left $\Uq{\ell}$-module algebra. This is a $q$-analogue of the $2^\ell$ dimensional
Grassmann algebra. Indeed, for its dimension we have $\,\dim\Gr=\sum_{k=0}^\ell\dim W_k=\sum_{k=0}^\ell\binom{\ell}{k}=2^\ell$.
We list explicitly the wedge product between elements with degree $0$ and $1$.
If $a\in W_0$ and $v\in W_1$ then $a\wprod v=av$ and $v\wprod a=va$;
if $v,w\in W_1$:
$$
(v\wprod w)_{i_1,i_2}=v_{i_1}w_{i_2}-q^{-1}v_{i_2}w_{i_1}
\;,\qquad\forall\;1\leq i_1<i_2\leq\ell\;;
$$
if $v\in W_1$, $w\in W_2$:
$$
(v\wprod w)_{i_1,i_2,i_3}=v_{i_1}w_{i_2,i_3}-q^{-1}v_{i_2}w_{i_1,i_3}
+q^{-2}v_{i_3}w_{i_1,i_2}
\;,\qquad\forall\;1\leq i_1<i_2<i_3\leq\ell\;.
$$
Also, the formula for the product of $v\in W_1$ and $w\in W_k$ will be useful later.
Any $p\in S^{(1)}_{k+1}$ has the form $p:(i_1,...,i_{k+1})
\mapsto (i_r)\times (i_1,...,i_{r-1},i_{r+1},...,i_{k+1})$ for some $1\leq r\leq k+1$,
and $||p||=r-1$. Thus
$$
(v\wprod w)_{\underline{i}}=\sum_{r=1}^{k+1} (-q)^{1-r}\,
v_{i_r}\,w_{\underline{i}\smallsetminus i_r} \;,
$$
and similarly
$$
(w\wprod v)_{\underline{i}}=\sum_{r=1}^{k+1} (-q)^{r-k-1}\,
w_{\underline{i}\smallsetminus i_r}v_{i_r} \;.
$$

\medskip

To discuss the first order condition, we'll need the following antilinear
map $J:W_k\to W_{\ell-k}$,
\begin{equation}\label{eq:JonW}
(Jw)_{\underline{i}}=
(-q^{-1})^{|\underline{i}|}q^{\frac{1}{4}\ell(\ell+1)}
\,\overline{w_{\,\underline{i}^c}} \;,
\end{equation}
where $|\underline{i}|:=i_1+i_2+\ldots+i_{\ell-k}$ and $\bar{z}$ is the
complex conjugate of $z\in\C$.

\begin{prop}\label{prop:Jsquare}
Let $0\leq k\leq\ell$. The map $J:W_k\to W_{\ell-k}$ has square
$$
J^2=(-1)^{\lfloor\frac{\ell+1}{2}\rfloor} \;,
$$
with $\lfloor t\rfloor$ the integer part of $t$.
It is equivariant, in the following sense:
\begin{equation}\label{eq:Jx}
\sigma_{\ell-k}(x^*)J=J\sigma_k(S(x))
\end{equation}
for all $x\in\Uq{\ell}$.
\end{prop}
\begin{proof}
Clearly, for $w\in W_k$,
$$
(J^2w)_{\underline{i}}=
(-q^{-1})^{i_1+i_2+\ldots+i_k+i^c_1+i^c_2+\ldots+i_{\ell-k}^c}q^{\frac{1}{2}\ell(\ell+1)}
\,w_{(\underline{i}^c)^c} \;,
$$
but $(\underline{i}^c)^c=\underline{i}$ and $\{i_r,i^c_s\}_{r,s}$ is the
set as \emph{all} integers between $1$ and $\ell$, so their sum is $\frac{1}{2}\ell(\ell+1)$,
and
$$
(J^2w)_{\underline{i}}=
(-q^{-1})^{\frac{1}{2}\ell(\ell+1)}q^{\frac{1}{2}\ell(\ell+1)}
\,w_{\underline{i}}=(-1)^{\frac{1}{2}\ell(\ell+1)}\,w_{\underline{i}} \;.
$$
Note that $\frac{1}{2}\ell(\ell+1)$ has the same parity of $\frac{1}{2}(\ell+1)$ if $\ell$
is odd, and it has the same parity of $\frac{1}{2}\ell$ if $\ell$ is
even. In both cases it has the same parity as $\lfloor\frac{\ell+1}{2}\rfloor$.
This proves the claim about $J^2$.

We pass to (\ref{eq:Jx}). Let
$c_{\underline{i},\ell}:=(-q^{-1})^{i_1+i_2+\ldots+i_k}q^{\frac{1}{4}\ell(\ell+1)}$. Firstly,
by (\ref{eq:iic}) (we omit the representation symbols)
$$
(K_jJw)_{\underline{i}}=q^{\frac{1}{2}(j\#\underline{i})}
(Jw)_{\underline{i}}=q^{-\frac{1}{2}(j\#\underline{i}^c)}(Jw)_{\underline{i}}
=q^{-\frac{1}{2}(j\#\underline{i}^c)}c_{\underline{i},\ell}\,
\overline{w_{\underline{i}^c}}=c_{\underline{i},\ell}\,\overline{(K_j^{-1}w)_{\underline{i}^c}} \;,
$$
that is $K_j^*Jw=J\,S(K_j)w$.

Now we use (\ref{eq:iic}), (\ref{eq:diesisB}), and the observation that
$-q^{-1}c_{\underline{i},\ell}=c_{\underline{i}^{j,+},\ell}$, to compute
\begin{align*}
\{J(-q^{-1}F_jw)\}_{\underline{i}}&=
c_{\underline{i},\ell}\overline{\{-q^{-1}F_jw\}_{\underline{i}^c}}=
-q^{-1}c_{\underline{i},\ell}\delta_{j\#\underline{i}^c,-1}\overline{w_{(\underline{i}^c)^{j,-}}} \\ &=
c_{\underline{i}^{j,+},\ell}\delta_{j\#\underline{i},1}\overline{w_{(\underline{i}^{j,+})^c}}=
\delta_{j\#\underline{i},1}(Jw)_{\underline{i}^{j,+}}=
\{E_jJw\}_{\underline{i}} \;.
\end{align*}
Since $J^2=\pm 1$, we have also $-q^{-1}F_jJ=JE_j$.
Hence, we have $x^*J=JS(x)$ for arbitrary generator $x=K_j,E_j,F_j$ of $\Uq{\ell}$,
and this concludes the proof.
\end{proof}

To any $x\in W_1$, we associate an operation of left `exterior product'
$\mathfrak{e}^L_x:W_k\to W_{k+1}$ (resp.~right `exterior product' $\mathfrak{e}^R_x:W_k\to W_{k+1}$)
via the rule
\begin{equation}\label{eq:extpro}
\mathfrak{e}_x^Lw=x\wprod w \;,\qquad
\mathfrak{e}_x^Rw=(-q)^kw\wprod x \;.
\end{equation}
We define the left (resp.~right) `contraction' as the adjoint $\mathfrak{i}_x^L$
of $\mathfrak{e}_x^L$ (resp.~$\mathfrak{i}_x^R$ of $\mathfrak{e}_x^R$) with respect to the inner product
on $W_k$ given by
$$
\inner{v,w}:=\sum\nolimits_{\underline{i}\in\Lambda_k}
\overline{v_{\underline{i}}}\,w_{\,\underline{i}} \;,
$$
for all $v,w\in W_k$.

\begin{prop}\label{prop:3.12}
We have
$$
J\mathfrak{e}_x^L J^{-1}=-q\mathfrak{i}^R_x\;,
$$
for all $x\in W_1$. As a consequence, denoting  $L(j,\underline{i}')$ the position
of $j$ inside the string $\underline{i}'$, we have
\begin{equation}\label{eq:internal}
(\mathfrak{i}_x^Rv)_{\underline{i}}=\sum_{j\notin\underline{i}}
(-q)^{L(j,\underline{i}\cup j)-1}\,\overline{x_j}\,v_{\underline{i}\cup j} \;\;,
\end{equation}
for all $v\in W_{k+1}$ and $\underline{i}\in\Lambda_k$.
\end{prop}

\begin{proof}
For any $v\in W_{k+1}$ and $\underline{i}\in\Lambda_k$,
\begin{align}
(J\mathfrak{e}^L_xJ^{-1}v)_{\underline{i}} &=
(-q^{-1})^{|\underline{i}|}q^{\frac{1}{4}\ell(\ell+1)}
\overline{(\mathfrak{e}_x^LJ^{-1}v)_{\underline{i}^c}} \notag \\
&=
(-q^{-1})^{|\underline{i}|}q^{\frac{1}{4}\ell(\ell+1)}
\sum\nolimits_{r=1}^{\ell-k}(-q^{-1})^{r-1}\,
\overline{(J^{-1}v)_{\underline{i}^c\smallsetminus i^c_r}}\,\overline{x_{i_r^c}} \notag \\
&=\sum\nolimits_{r=1}^{\ell-k}(-q)^{i^c_r-r+1}\,
v_{\underline{i}\cup i^c_r}\,\overline{x_{i_r^c}} \;\;. \label{eq:take}
\end{align}
Thus,
\begin{align*}
\inner{J\mathfrak{e}_x^LJ^{-1}v,w} &=
\sum_{\underline{i}\in\Lambda_k} \sum_{r=1}^{\ell-k}
(-q)^{i^c_r-r+1}\,\overline{v_{\underline{i}\cup i^c_r}}\,x_{i_r^c}\,
w_{\underline{i}}
=\sum_{\underline{i}\in\Lambda_k} \sum_{j\notin\underline{i}}
(-q)^{j-L(j,\underline{i}^c)+1}\,\overline{v_{\underline{i}\cup j}}\,x_j\,
w_{\underline{i}} \;\;, \\
\inner{v,-q\mathfrak{e}_x^Rw} &=
\sum_{\underline{i}'\in\Lambda_{k+1}}\sum_{r'=1}^{k+1}
(-q)^{r'}\,\overline{v_{\underline{i}'}}\,x_{i'_{r'}}\,
w_{\underline{i}'\smallsetminus i'_{r'}}
=\sum_{\underline{i}'\in\Lambda_{k+1}}\sum_{j\in\underline{i}'}
(-q)^{L(j,\underline{i}')}\,\overline{v_{\underline{i}'}}\,x_j\,
w_{\underline{i}'\smallsetminus j} \;\;,
\end{align*}
where $L(j,\underline{i}^c)$ (resp.~$L(j,\underline{i}')$) is the position
of $j$ inside $\underline{i}^c$ (resp.~$\underline{i}'$). We have
$$
\sum_{\underline{i}\in\Lambda_k} \sum_{j\notin\underline{i}}f_{\underline{i},j}
=\sum_{\underline{i}'\in\Lambda_{k+1}}\sum_{j\in\underline{i}'}
f_{\underline{i}'\smallsetminus j,j}
$$
for any $f$ and modulo a proportionality constant (any $\underline{i}'\in\Lambda_{k+1}$
can be written as a union $\underline{i}\cup j$, but the decomposition is not unique).
To check that the normalization is correct we take $f$ with all components equal to $1$,
and get
\begin{align*}
\sum_{\underline{i}\in\Lambda_k} \sum_{j\notin\underline{i}} 1   &=
(\ell-k)\sum_{\underline{i}\in\Lambda_k}1=(\ell-k)\tbinom{\ell}{k} \;,\\
\sum_{\underline{i}'\in\Lambda_{k+1}}\sum_{j\in\underline{i}'} 1 &=
(k+1)\sum_{\underline{i}'\in\Lambda_{k+1}}1=(k+1)\tbinom{\ell}{k+1} \;,
\end{align*}
and the two quantities above coincide. It remains to show that
\begin{equation}\label{eq:claimj}
L(j,\underline{i}^c)+L(j,\underline{i}\cup j)=j+1
\end{equation}
for all $\underline{i}\in\Lambda_k$ and $j\notin\underline{i}$.
From this it follows immediately $\inner{J\mathfrak{e}_x^LJ^{-1}v,w}=
\inner{v,-q\mathfrak{e}_x^Rw}$, so that the adjoint $-q\mathfrak{i}_x^R$ of $-q\mathfrak{e}_x^R$
is exactly $J\mathfrak{e}^L_x J^{-1}$.

We now prove (\ref{eq:claimj}) by induction. Let $k=1$ and $\underline{i}=(i_1)$.
We have $L(j,\underline{i}\cup j)=1$ if $j<i_1$ and $L(j,\underline{i}\cup j)=2$
if $j>i_1$. Concerning the left hand side, $\underline{i}^c=\{1,2,\ldots,i_1-1,i_1+1,\ldots,\ell\}$
and the position of $j$ inside $\underline{i}^c$ is $j$ itself if
$j<i_1$, and $j-1$ (for the empty position corresponding to $i_1$)
if $j>i_1$. In both cases, the sum is $j+1$.

Now we assume (\ref{eq:claimj}) is true for $k\geq 1$ generic, and prove that
it is true for $k+1$. Let $\underline{i}\in\Lambda_{k+1}$.
Call $\underline{i}''=(i_1,\ldots,i_k)$,
$\underline{i}=\underline{i}''\cup i_{k+1}$ and
$\underline{i}^c=\underline{i}''^{\,c}\smallsetminus i_{k+1}$. If
$i_{k+1}>j$ the inductive step follows from
$$
L(j,\underline{i}\cup j)=L(j,\underline{i}''\cup j)
\;,\qquad
L(j,\underline{i}^c)=L(j,\underline{i}''^{\,c})
\;.
$$
If $i_{k+1}<j$ then $i_r<j$ for all $r$, and
$$
L(j,\underline{i}\cup j)=k+2 \;,\qquad
L(j,\underline{i}^c)=j-k-1 \;.
$$
The sum is again $j+1$.
To conclude, if we take (\ref{eq:take}) and use (\ref{eq:claimj}), we
get (\ref{eq:internal}).
\end{proof}

Remark: for $q=1$, $\mathfrak{e}_x^L=\mathfrak{e}_x^R=:\mathfrak{e}_x$,
and for all $x,y\in W_1$, we have
$\mathfrak{e}_x\mathfrak{e}_y+\mathfrak{e}_y\mathfrak{e}_x=0$
and $\mathfrak{i}_x\mathfrak{e}_y+\mathfrak{e}_y\mathfrak{i}_x=\inner{x,y}\cdot
id_{\Gr}$. From this it follows that the map $x\mapsto\mathfrak{i}_x+\mathfrak{e}_x$
gives a representation of the Clifford algebra generated by $W_1$. For $q\neq 1$
this is no more true (for example, $(\mathfrak{e}^L_x)^2=
\mathfrak{e}^L_{x\wprod x}$, and $x\wprod x$ is not always zero). Fortunately,
we don't need this property in the sequel.

We conclude with a lemma on the quantum dimension $\dim_qW_k$ of $W_k$, that is defined as
the trace of 
$$
\prod\nolimits_{j=1}^{\ell-1}K_j^{2j(\ell-j)}=\maut\hat{K}^{-\ell-1} \;,
$$
which is the analogue of the element in (\ref{eq:Ssquare}) for the Hopf algebra $\Uq{\ell}$
(cf.~Sec.~7.1.6 of \cite{KS97}). Recall that $\hat{K}$ is defined in \eqref{eq:Khat}.
The geometrical meaning of $\dim_qW_k$ is the value of the $U_q(\mathfrak{su}(\ell))$
invariant of the unknot coloured by the representation $W_k$ (cf.~\cite{Res87}).

\begin{lemma}\label{lemma:dimq}
The quantum dimension of $W_k$ is given by
$$
\dim_qW_k=\sum\nolimits_{\underline{i}\in\Lambda_k}q^{k(\ell+1)-2|\underline{i}|} \;.
$$
It is symmetric under the exchange $q\to q^{-1}$, and its explicit value is
$$
\dim_qW_k=\tfrac{[\ell]!}{[k]!\,[\ell-k]!} \;,
$$
where $[n]!:=[n][n-1]\ldots [1]$ if $n\geq 1$, and $[0]!:=1$.
\end{lemma}

\begin{proof}
The general matrix element of $\maut\hat{K}^{-\ell-1}$ along the diagonal is
$$
\sigma_k(\maut\hat{K}^{-\ell-1})_{\underline{i},\underline{i}}=
q^{\sum_{j=1}^{\ell-1}j(\ell-j)\cdot(j\#\underline{i})} \;.
$$
But{\allowdisplaybreaks%
\begin{align*}
\sum\nolimits_{j=1}^{\ell-1}j(\ell-j)\cdot(j\#\underline{i})
&=\sum\nolimits_{j=1}^{\ell-1}j(\ell-j)\sum\nolimits_{h=1}^k(\delta_{j,i_h}-\delta_{j+1,i_h}) \\
&=\sum\nolimits_{h=1}^k\sum\nolimits_{j=1}^{\ell-1}j(\ell-j)(\delta_{j,i_h}-\delta_{j+1,i_h}) \\
&=\sum\nolimits_{h=1}^k\left(\sum\nolimits_{j=1}^{\ell-1}j(\ell-j)\delta_{j,i_h}-
\sum\nolimits_{j=2}^\ell(j-1)(\ell-j+1)\delta_{j,i_h}\right) \\
&=\sum\nolimits_{h=1}^k\left(\sum\nolimits_{j=1}^\ell j(\ell-j)\delta_{j,i_h}-
\sum\nolimits_{j=1}^\ell(j-1)(\ell-j+1)\delta_{j,i_h}\right) \\
&=\sum\nolimits_{h=1}^k\sum\nolimits_{j=1}^\ell\Big( j(\ell-j)-(j-1)(\ell-j+1)\Big)\delta_{j,i_h} \\
&=\sum\nolimits_{h=1}^k\sum\nolimits_{j=1}^\ell(\ell+1-2j)\delta_{j,i_h}
=\sum\nolimits_{h=1}^k(\ell+1-2i_h) \\
&=k(\ell+1)-2|\underline{i}| \;.
\end{align*}
}Thus:
\begin{equation}\label{eq:zaK}
\sigma_k(\maut\hat{K}^{-\ell-1})_{\underline{i},\underline{i}}=q^{k(\ell+1)-2|\underline{i}|} \;.
\end{equation}
This proves the first formula for $\dim_qW_k$.

A $q$-analogue of the hook formula for the quantum dimension of a general
irreducible representation of $U_q(\mathfrak{su}(\ell))$ is discussed in
the unpublished paper \cite{Res87}. We derive a simpler formula for the
representations we are interested in.
Let
$$
c^\ell_k:=\dim_qW_k=\sum_{1\leq i_1<i_2<\ldots<i_k\leq\ell}q^{k(\ell+1)-2(i_1+i_2+\ldots+i_k)} \;,
$$
where the dependence on $\ell$ is put in evidence. We decompose $\Lambda_k$ as
\begin{align*}
\Lambda_k &=\{1\leq i_1<i_2<\ldots<i_k\leq\ell\} \\
&=\{1\leq i_1<i_2<\ldots<i_k\leq\ell-1\}\cup
\{1\leq i_1<i_2<\ldots<i_{k-1}\leq\ell-1 \, ; \, i_k=\ell\} \;.
\end{align*}
This gives the recursive equation
\begin{align*}
c^\ell_k &= q^k\!\!\sum_{1\leq i_1<\ldots<i_k\leq\ell-1}\!\!q^{k\ell-2(i_1+i_2+\ldots+i_k)}
+q^{-(\ell-k)}\!\!\sum_{1\leq i_1<\ldots<i_{k-1}\leq\ell-1}\!\!q^{(k-1)\ell-2(i_1+i_2+\ldots+i_{k-1})} \\
&=q^k c^{\ell-1}_k+q^{-(\ell-k)}c^{\ell-1}_{k-1} \;.
\end{align*}
But since $q^k[\ell-k]+q^{-(\ell-k)}[k]=[\ell]$, we have also
$$
\tfrac{[\ell]!}{[k]!\,[\ell-k]!}=q^k\tfrac{[\ell-1]!}{[k]!\,[\ell-1-k]!}
+q^{-(\ell-k)}\tfrac{[\ell-1]!}{[k-1]!\,[\ell-k]!} \;.
$$
Thus, $c^\ell_k=c^1_1\tfrac{[\ell]!}{[k]!\,[\ell-k]!}$. But $c^1_1=1$,
and this concludes the proof.
\end{proof}

\subsection{Left invariant vector fields over $\CP^\ell_q$}
Classically (for $q=1$), left invariant vector fields on $\CP^\ell\simeq SU(\ell+1)/U(\ell)$ are given
by the right action of elements of $\mathfrak{u}^\perp(\ell)$, the orthogonal
complement of $\mathfrak{u}(\ell)$ inside $\mathfrak{su}(\ell+1)$. That is, any left
invariant vector field is a linear combination of the operators $M_{i,\ell}$
and $M_{i,\ell}^*$ (see Sec.~\ref{sec:Cas}), acting via the right canonical action.
In particular, the former gives the Dolbeault operator $\de$ and the latter $\deb$.
We are interested only in the latter one.
The vector $v=(v_i)$ with components $v_i=M_{i,\ell}^*|_{q=1}$ can be thought of
as an element of $\mathfrak{su}(\ell+1)\otimes_{\C}W_1$ that is right
$\mathfrak{u}(\ell)$-invariant with respect to the tensor product of the right adjoint action --
`dressed' with $S^{-1}$ -- and the action via $\sigma_1$. It is then clear how to generalize it to $q\neq 1$.

We denote by
$$
x\adj h:=S(h_{(1)})xh_{(2)}
$$
the right adjoint action of $\Uq{\ell+1}$ on itself, and set
$$
\mathfrak{X}:=\big\{v\in\Uq{\ell+1}\otimes_{\C}W_1
\,\big|\,v\adj\hat{K}=qv,\;\sigma_1(h_{(2)})v\adj S^{-1}(h_{(1)})=
\epsilon(h)v,\;\forall\;h\in\Uq{\ell}\big\} \;.
$$
Classically, primitive elements of $\mathfrak{X}$ (i.e.~such that
$\Delta(x)=x\otimes 1+1\otimes x$) are proportional to the vector
with components $M_{i,\ell}^*$. We look for a natural $q$-deformation
of these elements belonging to $\mathfrak{X}$ (the subset of primitive
elements is empty for $q\neq 1$).

\begin{lemma}\label{lemma:Xi}
Let $\{e^i\}_{i=1,\ldots,\ell}$ be the canonical basis of $W_1\simeq\C^\ell$
and
$$
X_i:=N_{i\ell}M_{i\ell}^*\;,\qquad
\forall\;i=1,\ldots,\ell\;,
$$
where $N_{i\ell}$ are defined in \eqref{eq:Njk}.
Then the vector
$$
X=\sum_ie^iX_i\in\Uq{\ell+1}\otimes_{\C}W_1
$$
belongs to $\mathfrak{X}$.
\end{lemma}
\begin{proof}
We have
$$
\hat{K}^{-1}X_\ell\hat{K}=
\hat{K}K_\ell^{-1} (\hat{K}^{-1}F_\ell\hat{K})
$$
and
$$
\hat{K}^{-1}F_\ell\hat{K}=
(K_{\ell-1}^{\ell-1}K_\ell^\ell)^{-\frac{2}{\ell+1}}F_\ell
(K_{\ell-1}^{\ell-1}K_\ell^\ell)^{\frac{2}{\ell+1}}=
q^{\big(\ell-\frac{1}{2}(\ell-1)\big)\frac{2}{\ell+1}}F_\ell
=qF_\ell \;.
$$
Since $\hat{K}$ commutes with $F_i$ for all $i=1,\ldots,\ell-1$, we
have $X_i\adj\hat{K}=\hat{K}^{-1}X_i\hat{K}=qX_i$ for all $i=1,\ldots,\ell$.
Next, to show the right $\Uq{\ell}$-invariance, it is enough to work with generators,
that means to verify
$$
X\adj K_j=\sigma_1(K_j)X \;,\qquad
X\adj E_j=\sigma_1(E_j)X \;,\qquad
X\adj F_j=\sigma_1(F_j)X \;,
$$
for all $j=1,\ldots,\ell-1$.
We prove by induction that
$$
X_i\adj K_j:=K_j^{-1}X_iK_j=q^{\frac{1}{2}(\delta_{i,j}-\delta_{i,j+1})}X_i\equiv\{\sigma_1(K_j)X\}_i \;,
$$
for all $i=1,\ldots,\ell$ and for all $j=1,\ldots,\ell-1$. It is true for $i=\ell$,
since $F_\ell$ commutes with $K_j$ for all $j<\ell-1$, and
$K_{\ell-1}^{-1}F_\ell K_{\ell-1}=q^{-\frac{1}{2}}F_\ell$.
The inductive step comes from the recursive relation
$$
X_i=N_{i,\ell}[N_{i+1,\ell}^{-1}X_{i+1},F_i]_q
$$
that together with $F_i\adj K_j=q^{\delta_{i,j}-\frac{1}{2}(\delta_{i,j+1}+\delta_{i,j-1})}$
gives
\begin{align*}
X_i\adj K_j &=N_{i,\ell}[N_{i+1,\ell}^{-1}X_{i+1}\adj K_j,F_i\adj K_j]_q \\
&=q^{\frac{1}{2}(\delta_{i+1,j}-\delta_{i+1,j+1})}
q^{\delta_{i,j}-\frac{1}{2}(\delta_{i,j+1}+\delta_{i,j-1})}
N_{i,\ell}[N_{i+1,\ell}^{-1}X_{i+1},F_i]_q \\
&=q^{\frac{1}{2}(\delta_{i,j}-\delta_{i,j+1})}X_i \;,
\end{align*}
for all $i=1,\ldots,\ell-1$.
Since
$\,K_jM_{i,\ell}^*K_j^{-1}=q^{-\frac{1}{2}(\delta_{i,j}-\delta_{i,j+1})}M_{i,\ell}^*\,$,
$\,N_{i,\ell}E_jN_{i,\ell}^{-1}=q^{\frac{1}{2}(\delta_{i,j}-\delta_{i,j+1})}\,$ and
$\,K_jE_jK_j^{-1}=qE_j$,
we have
\begin{align*}
X_i\adj E_j &=K_jX_iE_j-qE_jX_iK_j \\
&=K_jN_{i,\ell}M_{i,\ell}^*E_j-qE_jN_{i,\ell}M_{i,\ell}^*K_j \\
&=K_jN_{i,\ell}\big\{M_{i,\ell}^*E_j-q
(K_j^{-1}N_{i,\ell}^{-1}E_jN_{i,\ell}K_j)(K_j^{-1}M_{i,\ell}^*K_j)\big\} \\
&=K_jN_{i,\ell}[M_{i,\ell}^*,E_j] \\
\intertext{and by (\ref{eq:cons}) and $\,N_{i,\ell}=K_iN_{i+1,\ell}\,$, this is}
&=K_jN_{i,\ell}\delta_{i,j}K_i^{-2}M_{i+1,\ell}^* \\
&=\delta_{i,j}N_{i+1,\ell}M_{i+1,\ell}^* \\
&=\delta_{i,j}X_{i+1}\equiv\{\sigma_1(E_j)X\}_i \;.
\end{align*}
Finally,
\begin{align*}
X_i\adj F_j &=K_jN_{i,\ell}M_{i,\ell}^*F_j-q^{-1}F_jN_{i,\ell}M_{i,\ell}^*K_j \\
&=K_jN_{i,\ell}^{-1}\big\{N^2_{i,\ell}M_{i,\ell}^*F_j
-q^{-1}(K_j^{-1}N_{i,\ell}F_jN_{i,\ell}^{-1}K_j)N^2_{i,\ell}(K_j^{-1}M_{i,\ell}^*K_j)\big\} \\
&=K_jN_{i,\ell}^{-1}[N^2_{i,\ell}M_{i,\ell}^*,F_j] \\
\intertext{and by the adjoint of (\ref{eq:last}) this is}
&=K_jN_{i,\ell}^{-1}\delta_{i,j+1}N^2_{i,\ell}M_{i-1,\ell}^* \\
&=\delta_{i,j+1}N_{i-1,\ell}M_{i-1,\ell}^* \\
&=\delta_{i,j+1}X_{i-1}\equiv\{\sigma_1(F_j)X\}_i \;.
\end{align*}
This concludes the proof.
\end{proof}

The following lemmas will be useful later.
\begin{lemma}
We have
\begin{equation}\label{eq:debsquare}
X_iX_j=q^{-1}X_jX_i \;,
\end{equation}
for all $1\leq i<j\leq\ell$.
\end{lemma}
\begin{proof}
From $X_i=N_{i,\ell}M_{i,\ell}^*$, and \eqref{eq:NEN}, which implies
$N_{j,\ell}M_{i,\ell}^*N_{j,\ell}^{-1}=q^{1-\frac{1}{2}\delta_{i,j}}M_{i,\ell}^*$,
we deduce
$$
[X_i,X_j]_q=q^{\frac{1}{2}\delta_{i,j}-1}
N_{i,\ell}N_{j,\ell}[M_{j,\ell},M_{i,\ell}]_q^* \;.
$$
Now we prove, by induction on $j$, that $[M_{j,\ell},M_{i,\ell}]_q=0$ for all $j>i$.
For $j=\ell$ we have
$$
[E_\ell,M_{i,\ell}]_q=0
$$
by (\ref{eq:lastD}). Then, if assume the claim true for a generic $j$, $i+1<j\leq\ell$,
using the identity
$$
[[a,b]_q,c]_q=[a,[b,c]_q]_q+q^{-1}[[a,c],b]
$$
we get
$$
[M_{j-1,\ell},M_{i,\ell}]_q=
[[E_{j-1},M_{j,\ell}]_q,M_{i,\ell}]_q=
[E_{j-1},[M_{j,\ell},M_{i,\ell}]_q]_q+q^{-1}
[[E_{j-1},M_{i,\ell}],M_{j,\ell}] \;.
$$
The first term is zero by inductive hypothesis, and the second is
zero since $[E_{j-1},M_{i,\ell}]=0$ for all $i<j-1<\ell$ by (\ref{eq:lastC}).
Thus, $[M_{j,\ell},M_{i,\ell}]_q=0$ implies $[M_{j-1,\ell},M_{i,\ell}]_q=0$
for all $i+1<j\leq\ell$, and this concludes the proof.
\end{proof}

\begin{lemma}
We have
\begin{equation}\label{eq:Xcop}
\Delta(X_i)=X_i\otimes N_{i,i-1}+\hat{K}^{-1}\otimes X_i
+q^{-\frac{1}{2}}(q-q^{-1})\sum_{j=i}^{\ell-1}X_{j+1}\otimes N_{i,j}M_{i,j}^* \;.
\end{equation}
for all $1\leq i\leq\ell$.
\end{lemma}
\begin{proof}
Let $c:=q^{-\frac{1}{2}}(q-q^{-1})$. We prove by induction that
\begin{align}
\Delta(M_{i,\ell}) &=M_{i,\ell}\otimes K_i\ldots K_\ell+(K_i\ldots K_\ell)^{-1}\otimes M_{i,\ell} \notag\\
&\quad +c\sum_{j=i}^{\ell-1}M_{j+1,\ell}(K_i\ldots K_j)^{-1}\otimes M_{i,j}(K_{j+1}\ldots K_\ell) \;.\label{eq:rhs}
\end{align}
This immediately implies the claim for $X_i$. For $i=\ell$ the equality (\ref{eq:rhs})
holds
$$
\Delta(E_\ell)=E_\ell\otimes K_\ell +K_\ell^{-1}\otimes E_\ell \;.
$$
Suppose \eqref{eq:rhs} gives the correct value of 
$\Delta(M_{i+1,\ell})$ for some $i+1\leq\ell$. Then
\begin{align*}
\Delta(M_{i,\ell}) &=[\Delta(E_i),\Delta(M_{i+1,\ell})]_q \\
&=
\left[E_i\otimes K_i,M_{i+1,\ell}\otimes K_{i+1}\ldots K_\ell\right]_q
+\left[K_i^{-1}\otimes E_i,M_{i+1,\ell}\otimes K_{i+1}\ldots K_\ell\right]_q
\\ &\quad
+\left[E_i\otimes K_i,(K_{i+1}\ldots K_\ell)^{-1}\otimes M_{i+1,\ell}\right]_q
+\left[K_i^{-1}\otimes E_i,(K_{i+1}\ldots K_\ell)^{-1}\otimes M_{i+1,\ell}\right]_q
\\ &\quad
+c\sum\nolimits_{j=i+1}^{\ell-1}
\left[E_i\otimes K_i,M_{j+1,\ell}(K_{i+1}\ldots K_j)^{-1}\otimes M_{i+1,j}(K_{j+1}\ldots K_\ell)\right]_q
\\ &\quad
+c\sum\nolimits_{j=i+1}^{\ell-1}
\left[K_i^{-1}\otimes E_i,M_{j+1,\ell}(K_{i+1}\ldots K_j)^{-1}\otimes M_{i+1,j}(K_{j+1}\ldots K_\ell)\right]_q \\
&=
M_{i,\ell}\otimes K_i\ldots K_\ell
+q^{-\frac{1}{2}}(q-q^{-1})M_{i+1,\ell}K_i^{-1}\otimes E_i(K_{i+1}\ldots K_\ell)
\\ &\quad
+\;0\;+(K_i\ldots K_\ell)^{-1}\otimes M_{i,\ell}
\\ &\quad
+c\sum\nolimits_{j=i+1}^{\ell-1}q^{-\frac{1}{2}}
[E_i,M_{j+1,\ell}](K_{i+1}\ldots K_j)^{-1}\otimes M_{i+1,j}K_i(K_{j+1}\ldots K_\ell)
\\ &\quad
+c\sum\nolimits_{j=i+1}^{\ell-1}M_{j+1,\ell}(K_i\ldots K_j)^{-1}\otimes
[E_i,M_{i+1,j}]_q(K_{j+1}\ldots K_\ell) \\
&=M_{i,\ell}\otimes K_i\ldots K_\ell+(K_i\ldots K_\ell)^{-1}\otimes M_{i,\ell}
\\ &\quad
+c\sum\nolimits_{j=i}^{\ell-1}M_{j+1,\ell}(K_i\ldots K_j)^{-1}\otimes
M_{i,j}(K_{j+1}\ldots K_\ell) \;,
\end{align*}
and last equation is exactly the right hand side of (\ref{eq:rhs}).
\end{proof}

\begin{lemma}
We have
\begin{subequations}
\begin{align}
[X_i^*,S^{-1}(X_j)] &=q^{\frac{1}{2}}\,(\hat{K}N_{j,i-1})^{-1}\,S^{-1}(M_{j,i-1}^*)
&&\textup{for all}\;\ell\geq i>j\geq 1 \;,\label{eq:SXXA} \\
[X_i^*,S^{-1}(X_i)] &=-q\,\frac{(K_i\ldots K_\ell)^2-(K_i\ldots K_\ell)^{-2}}{q-q^{-1}}
&&\textup{for all}\;i=1,\ldots,\ell \;.\label{eq:SXXB}
\end{align}
\end{subequations}
\end{lemma}

\begin{proof}
Since $K_iX_jK_i^{-1}=q^{\frac{1}{2}(\delta_{i,j-1}-\delta_{i,j}-\delta_{i,\ell})}
X_j$, we have
$$
N_{i,\ell}X_jN_{i,\ell}^{-1}=q^{\frac{1}{2}(1-\delta_{i,j})}X_j \;,
$$
and
$$
[X_i^*,S^{-1}(X_j)]=q^{\frac{1}{2}(\delta_{i,j}-1)}\,(K_j\ldots K_{i-1})^{-1}\,[M_{i,\ell},S^{-1}(M_{j,\ell}^*)] \;,
$$
for all $i\geq j$. Thus, it is equivalent to prove the equalities
\begin{subequations}
\begin{align}
[M_{i,\ell},S^{-1}(M_{j,\ell}^*)] &=q(K_i\ldots K_\ell)^{-2}\,S^{-1}(M_{j,i-1}^*)
&&\textup{for all}\;\ell\geq i>j\geq 1 \;,\label{eq:ausA} \\
[M_{i,\ell},S^{-1}(M_{i,\ell}^*)] &=-q\,\frac{(K_i\ldots K_\ell)^2-(K_i\ldots K_\ell)^{-2}}{q-q^{-1}}
&&\textup{for all}\;i=1,\ldots,\ell \;.\label{eq:ausB}
\end{align}
\end{subequations}
Let us start with (\ref{eq:ausA}).
We prove it by induction on $i$. For $i=\ell$ ($1\leq j<i$), (\ref{eq:cons}) gives
$$
[E_\ell,S^{-1}(M_{j,\ell}^*)]=qS^{-1}([E_\ell,M_{j,\ell}^*])=
qS^{-1}(M_{j,\ell-1}^*K_\ell^2)=qK_\ell^{-2}S^{-1}(M_{j,\ell-1}^*) \;.
$$
Suppose, as claimed in \eqref{eq:ausA}, that $[M_{i+1,\ell},S^{-1}(M_{j,\ell}^*)]=
q(K_{i+1}\ldots K_\ell)^{-2}\,S^{-1}(M_{j,i}^*)$ for some $i+1\leq\ell$.
Since $M_{i,\ell}=[E_i,M_{i+1,\ell}]_q$ we have
$$
[M_{i,\ell},S^{-1}(M_{j,\ell}^*)]=
[[E_i,S^{-1}(M_{j,\ell}^*)],M_{i+1,\ell}]_q
+[E_i,[M_{i+1,\ell},S^{-1}(M_{j,\ell}^*)]]_q \;.
$$
But the first term is zero by (\ref{eq:cons}), since $j<i<\ell$ are all distinct,
while in the second term we use the inductive hypothesis. Thus
\begin{align*}
[M_{i,\ell},S^{-1}(M_{j,\ell}^*)]
&=q[E_i,(K_{i+1}\ldots K_\ell)^{-2}\,S^{-1}(M_{j,i}^*)]_q
 =(K_{i+1}\ldots K_\ell)^{-2}[E_i,S^{-1}(M_{j,i}^*)] \\
&=q(K_{i+1}\ldots K_\ell)^{-2}S^{-1}([E_i,M_{j,i}^*]) \;.
\end{align*}
Using again (\ref{eq:cons}) we find
$$
[M_{i,\ell},S^{-1}(M_{j,\ell}^*)]
=q(K_{i+1}\ldots K_\ell)^{-2}S^{-1}(M_{j,i-1}^*K_i^2)
=q(K_i\ldots K_\ell)^{-2}S^{-1}(M_{j,i-1}^*) \;,
$$
and this proves the inductive step, and then (\ref{eq:ausA}).

We pass to (\ref{eq:ausB}), and prove it by induction on $i$.
Since $[M_{\ell,\ell},S^{-1}(M_{\ell,\ell}^*)]=-q[E_\ell,F_\ell]$, the claim is true for $i=\ell$.
We now show that the claim \eqref{eq:ausB} for $[M_{i+1,\ell},S^{-1}(M_{i+1,\ell}^*)]$
implies the claim for $[M_{i,\ell},S^{-1}(M_{i,\ell}^*)]$, for any $1<i+1\leq\ell$.
Since $M_{i,\ell}=[E_i,M_{i+1,\ell}]_q$ we have
$$
[M_{i,\ell},S^{-1}(M_{i,\ell}^*)]=
-q[[E_i,M_{i+1,\ell}]_q,[F_i,S^{-1}(M_{i+1,\ell}^*)]_q] \;.
$$
But $E_i$ commutes with $S^{-1}(M_{i+1,\ell}^*)$ and $F_i$ commutes with
$M_{i+1,\ell}$. We then use the following identity
$$
[[a,b]_q,[c,d]_q]=
[[[a,c],b]_q,d]_q+
[c,[a,[b,d]]_q]_q \;,
$$
that is valid whenever $[a,d]=[b,c]=0$, to write
$$
-q^{-1}[M_{i,\ell},S^{-1}(M_{i,\ell}^*)]=[[[a,c],b]_q,d]_q+[c,[a,[b,d]]_q]_q \;,
$$
where in our case $a=E_i$, $b=M_{i+1,\ell}$, $c=F_i$ and $d=S^{-1}(M_{i+1,\ell}^*)$.
By inductive hypothesis,
$$
[b,d]=[M_{i+1,\ell},S^{-1}(M_{i+1,\ell}^*)]=
-q\,\frac{(K_{i+1}\ldots K_\ell)^2-(K_{i+1}\ldots K_\ell)^{-2}}{q-q^{-1}} \;,
$$
so
$$
[a,[b,d]]_q=-E_i(K_{i+1}\ldots K_\ell)^2 \;,
$$
and
$$
[c,[a,[b,d]]_q]_q
=[E_i,F_i](K_{i+1}\ldots K_\ell)^2
=\frac{(K_i\ldots K_\ell)^2-K_i^{-2}(K_{i+1}\ldots K_\ell)^2}{q-q^{-1}} \;.
$$
Similarly,
$$
[a,c]=\frac{K_i^2-K_i^{-2}}{q-q^{-1}} \;,
$$
so
$$
[[a,c],b]_q=-M_{i+1,\ell}K_i^{-2} \;,
$$
and
$$
[[[a,c],b]_q,d]_q=q^{-1}[S^{-1}(M_{i+1,\ell}^*),M_{i+1,\ell}]K_i^{-2}
$$
that again by inductive hypothesis gives
$$
[[[a,c],b]_q,d]_q=\frac{K_i^{-2}(K_{i+1}\ldots K_\ell)^2-(K_i\ldots K_\ell)^{-2}}{q-q^{-1}} \;.
$$
The sum is
$$
[[[a,c],b]_q,d]_q+[c,[a,[b,d]]_q]_q=
\frac{(K_i\ldots K_\ell)^2-(K_i\ldots K_\ell)^{-2}}{q-q^{-1}} \;,
$$
and this proves the inductive step.
\end{proof}

\section{The quantum projective space $\CP^\ell_q$ and equivariant modules}\label{sec:due}
We recall the definition of the Hopf $*$-algebra $\Oq$, deformation of the algebra of representative functions
on $SU(\ell+1)$ (cf.~\cite{KS97}, Sec.~9.2). As a $*$-algebra it is generated by elements
$u^i_j$ ($i,j=1,\ldots,\ell+1$) with commutation relations
\begin{align*}
u^i_ku^j_k &=qu^j_ku^i_k &
u^k_iu^k_j &=qu^k_ju^k_i &&
\forall\;i<j\;, \\
[u^i_l,u^j_k]&=0 &
[u^i_k,u^j_l]&=(q-q^{-1})u^i_lu^j_k &&
\forall\;i<j,\;k<l\;,
\end{align*}
and with determinant relation
$$
\sum\nolimits_{p\in S_{\ell+1}}(-q)^{||p||}
u^1_{p(1)}u^2_{p(2)}\ldots u^{\ell+1}_{p(\ell+1)}=1 \;,
$$
where the sum is over all permutations $p$ of the set
$\{1,2,\ldots,\ell+1\}$ and $||p||$ is the number of inversions
in $p$. The $*$-structure is given by
$$
(u^i_j)^*=(-q)^{j-i}\sum\nolimits_{p\in S_\ell}(-q)^{||p||}u^{k_1}_{p(n_1)}
u^{k_2}_{p(n_2)}\ldots u^{k_\ell}_{p(n_\ell)}
$$
with $\{k_1,\ldots,k_\ell\}=\{1,\ldots,\ell+1\}\smallsetminus\{i\}$,
$\{n_1,\ldots,n_\ell\}=\{1,\ldots,\ell+1\}\smallsetminus\{j\}$
(as ordered sets) and the sum is over all permutations $p$ of the set
$\{n_1,\ldots,n_\ell\}$.
Coproduct, counit and antipode are of `matrix' type:
$$
\Delta(u^i_j)=\sum\nolimits_ku^i_k\otimes u^k_j\;,\qquad
\epsilon(u^i_j)=\delta^i_j\;,\qquad
S(u^i_j)=(u^j_i)^*\;.
$$

The basic representation $\pi:\Uq{\ell+1}\to\mathrm{Mat}_{\ell+1}(\C)$
of $\Uq{\ell+1}$, i.e.~the one with highest weight $(n_1,\ldots,n_\ell)=(0,0,\ldots,0,1)$,
is given in matrix form by
\begin{equation}\label{eq:frep}
\pi^j_k(1)=\delta^j_k\;,\qquad
\pi^j_k(K_i)=\delta^j_k q^{\frac{1}{2}(\delta_{i+1,j}-\delta_{i,j})}\;,\qquad
\pi^j_k(E_i)=\delta^j_{i+1}\delta^i_k\;,
\end{equation}
where $i\in\{1,\ldots,\ell\}$ and $j,k\in\{1,\ldots,\ell+1\}$.
Since $\pi(\hat{K}^{(\ell+1)/2})$ is a positive operator, $\pi(\hat{K}):=
\pi(\hat{K}^{(\ell+1)/2})^{2/(\ell+1)}$ is well defined and the representation can
be extended to the extension of $\Uq{\ell+1}$ by $\hat{K}$.

In our notation (that agrees with~\cite{KS97}), the action of a
matrix $T=((T^j_k))$ on a vector $v=(v^1,\ldots,v^{\ell+1})^t\in\C^{\ell+1}$
gives the vector $Tv$ with components
$$
(Tv)^j=\sum\nolimits_{k=1}^{\ell+1}T^j_kv^k \;.
$$
That is, in $M^j_k$ the label $j$ indicates the row while the label $k$
indicates the column, and $v=(v^j)$ is a column vector.
Modulo a replacement $\ell\to\ell-1$, this representation is equivalent to
the representation $W_\ell$ of $\Uq{\ell}$ given in Sec.~\ref{sec:3.2}: the intertwiner
is the map sending $v=(v^1,\ldots,v^\ell)^t\in\C^\ell$ to the vector
$w\in W_{\ell-1}$ with components given by $w_{1,2,\ldots,j-1,j+1,\ldots,\ell}=v^j$.

With the representation (\ref{eq:frep}) one defines a
pairing between $\Uq{\ell+1}$ and $\Oq$ through the formul{\ae} (cf.~\cite{KS97}, Sec.~9.4)
\begin{equation}\label{eq:pairing}
\inner{h,1}:=\epsilon(h)\;,\qquad
\inner{h,\smash[t]{u^j_k}}:=\pi^j_k(h)\;,\qquad\forall\;h\in \Uq{\ell+1}\;,
\end{equation}
and with this pairing one constructs the left and right canonical actions,
$h\az a=a_{(1)}\inner{h,\smash[b]{a_{(2)}}}$ and
$a\za h=\inner{h,\smash[b]{a_{(1)}}}a_{(2)}$,
which on generators are then given by
$$
h\az u^i_j=\sum\nolimits_ku^i_k\,\pi^k_j(h)\;,\qquad
u^i_j\za h=\sum\nolimits_k\pi^i_k(h)u^k_j\;.
$$
Note that the Casimir $\mathcal{C}_q$, being central,
satisfies $\mathcal{C}_q\az a=a\za \mathcal{C}_q$ for all
$a\in\Oq$.

Remark: by definition of left canonical action, the pairing between $\Uq{\ell+1}$ and $\Oq$
can be written as $\inner{h,a}=\epsilon(h\az a)$. Since $\hat{K}^{(\ell+1)/2}\az$
is a positive diagonalizable (invertible) operator, it is immediate to extend the
left canonical action to its (positive) root $\hat{K}$, and to extend the pairing
to the extension of $\Uq{\ell+1}$ by $\hat{K}$ and $\hat{K}^{-1}$.

Consider the following left action of $\Uq{\ell+1}$ on $\Oq$:
\begin{equation}\label{eq:mL}
\mL_ha:=a\za S^{-1}(h) \;.
\end{equation}
Notice that the map $h\mapsto\mL_h$ is a $*$-representation of $\Uq{\ell+1}$ on $\Oq$,
for the inner product $(a,b):=\varphi(a^*b)$ on $\Oq$ coming from the Haar state.
Indeed, using the right invariance of $\varphi$, we get
\begin{align*}
(\mL_{h^*}a,b)
&=\varphi\bigl(\{a\za S^{-1}(h^*)\}^*b\bigr)
 =\varphi\bigl(\{a^*\za h\}b\bigr) \\
&=\varphi\bigl(\{a^*\za h_{(1)}\}\epsilon(h_{(2)})b\bigr)
 =\varphi\bigl(\{a^*\za h_{(1)}\}\{b\za S^{-1}(h_{(3)})h_{(2)}\}\bigr) \\
&=\varphi\bigl(\{a^*(b\za S^{-1}(h_{(2)}))\}\za h_{(1)}\bigr)
 =\epsilon(h_{(1)})\varphi\bigl(a^*\{b\za S^{-1}(h_{(2)})\}\bigr) \\
&=\varphi\bigl(a^*\{b\za S^{-1}(h)\}\bigr)
 =(a,\mL_hb) \;,
\end{align*}
that means $\mL_{h^*}=(\mL_h)^\dag$ for all $h\in\Uq{\ell+1}$.

The algebra $\A(S^{\ell+1}_q)$ of `functions' on the unitary quantum sphere
$S^{2\ell+1}_q$ is defined as
$$
\A(S^{\ell+1}_q):=\big\{a\in\Oq\,\big|\,
\mL_h a=\epsilon(h)a\;\forall\;h\in\Uq{\ell}\big\} \;.
$$
There is an isomorphism with the $*$-algebra generated by
$\{z_i,z_i^*\}_{i=1,\ldots,\ell+1}$ with relations
\begin{align*}
z_iz_j &=qz_jz_i &&\forall\;i<j \;,\\
z_i^*z_j &=qz_jz_i^* &&\forall\;i\neq j \;,\\
[z_1^*,z_1] &=0 \;,\\
[z_{i+1}^*,z_{i+1}] &=(1-q^2)\sum\nolimits_{j=1}^i z_jz_j^* &&\forall\;i=1,\ldots,\ell \;,\\
z_1z_1^*+z_2z_2^* &+\ldots+z_{\ell+1}z_{\ell+1}^*=1 \;,
\end{align*}
given on generators by the identification $z_i=u_i^{\ell+1}$~\cite{VS91}.

The algebra $\Aq{\ell}$ of `functions' on the quantum projective
space $\CP^\ell_q$ is defined as
$$
\Aq{\ell}:=\big\{a\in\A(S^{\ell+1}_q)\,\big|\,\mL_{\hat{K}}a=a\big\} \;.
$$
Let $\Kq{\ell}$ be the Hopf algebra generated by $\Uq{\ell}$, $\hat{K}$ and
$\hat{K}^{-1}$. If $\lambda$ is an $N$-dimensional $*$-representation of $\Kq{\ell}$,
the tensor product (over $\C$) $\Oq\otimes\C^N\simeq\Oq^N$ carries a natural action
of $\Kq{\ell}$, that is the Hopf tensor product of $\mL$ and $\lambda$.
We call $\mathfrak{M}(\lambda)=\Oq\!\boxtimes_{\lambda}\!\C^N$ the subset of
$\Kq{\ell}$-invariant elements:
$$
\mathfrak{M}(\lambda):=
\big\{v\in\Oq^N\,\big|\,\{\mL_{h_{(1)}}\otimes\lambda(h_{(2)})\}v
=\epsilon(h)v\,,\;\forall\;h\in\Kq{\ell}\big\} \;,
$$
where $v=(v_1,\ldots,v_N)^t$ is a column vector and row by column multiplication
is understood. This is an $\Aq{\ell}$-bimodule and a left $\Aq{\ell}\rtimes\Uq{\ell+1}$-module.

\begin{lemma}\label{lemma:equiv}
Let $v\in\Oq^N$. The following are equivalent:
\begin{itemize}
\item[(i)] $\{\mL_{h_{(1)}}\otimes\lambda(h_{(2)})\}v=\epsilon(h)v$ for all $h\in\Kq{\ell}$;
\item[(ii)] $\lambda(S(h_{(1)}))v\za h_{(2)}=\epsilon(h)v$ for all $h\in\Kq{\ell}$;
\item[(iii)] $v\za h=\lambda(h)v$ for all $h\in\Kq{\ell}$.
\end{itemize}
\end{lemma}
\begin{proof}
Equivalence between (i) and (ii) is straightforward: just replace $h$ with $S(h)$.

Last condition can be rewritten as $\mL_hv=\lambda(S^{-1}(h))v$. Since $S^{-1}(h_{(2)})h_{(1)}=\epsilon(h)$,
$$
\mL_h\otimes id_{\C^N}=(id_{\Oq}\otimes \lambda(S^{-1}(h_{(3)})))(\mL_{h_{(1)}}\otimes\lambda(h_{(2)})) \;.
$$
Thus (i) implies (iii): if $v$ is invariant,
$$
\mL_hv=\lambda(S^{-1}(h_{(3)}))(\mL_{h_{(1)}}\otimes\lambda(h_{(2)}))v
=\lambda(S^{-1}(h_{(2)}))\epsilon(h_{(1)})v=\lambda(S^{-1}(h))v \;.
$$
Vice versa, if $v\za h=\lambda(h)v$, we have
$$
(\mL_{h_{(1)}}\otimes\lambda(h_{(2)}))v=
\lambda(h_{(2)})v\za S^{-1}(h_{(1)})=\lambda(h_{(2)}S^{-1}(h_{(1)}))v=\epsilon(h)v \;.
$$
Hence, (iii) implies (i).
\end{proof}

We introduce an $\Aq{\ell}$-valued sesquilinear map:
$$
\mathfrak{M}(\lambda)\times\mathfrak{M}(\lambda)\to\Aq{\ell} \;,\qquad
(v,v')\mapsto v^\dag v' \;,
$$
where $v^\dag$ is the conjugate transpose of $v$ and row by column multiplication
is understood. Indeed, if $v,v'\in\mathfrak{M}(\lambda)$, then $v^\dag v'\in\Aq{\ell}$:
\begin{align*}
v^\dag v'\za h
 &=(v^\dag\za h_{(1)})(v'\za h_{(2)})
 =(v^\dag\za h_{(1)})\lambda(h_{(2)})\lambda(S(h_{(3)}))(v'\za h_{(4)}) \\
 &=(v^\dag\za h_{(1)})\lambda(h_{(2)})\epsilon(h_{(3)})v'
 =\bigl\{\lambda(h_{(2)}^*)v\za S(h_{(1)})^*\bigr\}^*v' \\
 &=\bigl\{\lambda(S(t_{(1)}))v\za t_{(2)})\bigr\}^*v'
 =\epsilon(t)^*v^\dag v'=\epsilon(h)v^\dag v' \;,
\end{align*}
for all $h\in\Kq{\ell}$, where we denoted $t:=S(h)^*$.

Composing this map with the Haar functional $\varphi:\Oq\to\C$, we get a non-degenerate
inner product on $\mathfrak{M}(\lambda)$, $\inner{v,v'}:=\varphi(v^\dag v')$.
Thus $\mathfrak{M}(\lambda)$ are noncommutative `homogeneous Hermitian vector bundles' over $\CP^\ell_q$.

\subsection{Line bundles on $\CP^\ell_q$}
Let
$$
[n]!:=[n][n-1]\ldots [1] \;,
$$
for any $n\geq 1$ (and $[0]!:=1$). For $j_1,\ldots,j_{\ell+1}\in\N$, we define the $q$-multinomial coefficients as
$$
[j_1,\ldots,j_{\ell+1}]!:=\frac{[j_1+\ldots+j_{\ell+1}]!}{[j_1]!\ldots[j_{\ell+1}]!}
q^{-\sum_{r<s}j_rj_s} \;.
$$
The following lemma is a generalization of a similar lemma for $\CP^2_q$ (cf.~\cite{DL08}).

\begin{lemma}
The generators $z_i$ of $\A(S^{\ell+1}_q)$ satisfy
$$
\sum_{j_1+\ldots+j_{\ell+1}=N}
[j_1,\ldots,j_{\ell+1}]!\,
z_1^{j_1}\ldots z_{\ell+1}^{j_{\ell+1}}
(z_1^{j_1}\ldots z_{\ell+1}^{j_{\ell+1}})^*=1 \;,
$$
for all $N\in\N$.
\end{lemma}
\begin{proof}
The equality
$$
[j_1+\ldots+j_{\ell+1}]=\sum_{i=1}^{\ell+1}\,[j_i]\,q^{\sum_{k=1}^{i-1}j_k-\sum_{k=i+1}^{\ell+1}j_k}
$$
implies
\begin{equation}\label{eq:qmulti}
[j_1,\ldots,j_{\ell+1}]!=\sum_{i=1}^{\ell+1}\,[j_1,\ldots,j_i-1,\ldots,j_{\ell+1}]!\,
q^{-2\sum_{k=i+1}^{\ell+1}j_k} \;.
\end{equation}
Let $c_N$ be the polynomial we want to compute,
$$
c_N:=\sum_{j_1+\ldots+j_{\ell+1}=N}
[j_1,\ldots,j_{\ell+1}]!\,
z_1^{j_1}\ldots z_{\ell+1}^{j_{\ell+1}}
(z_1^{j_1}\ldots z_{\ell+1}^{j_{\ell+1}})^* \;.
$$
We prove the lemma by induction on $N$. For $N=1$,
we get the spherical relation of the algebra $c_1=\sum_{i=1}^{\ell+1}z_iz_i^*=1$.
From this, it follows that $c_N$ can be rewritten as
$$
c_N=\sum_{j_1+\ldots+j_{\ell+1}=N}
[j_1,\ldots,j_{\ell+1}]!\,
z_1^{j_1}\ldots z_{\ell+1}^{j_{\ell+1}}\left(\textstyle{\sum_{i=1}^{\ell+1}}z_iz_i^*\right)
(z_1^{j_1}\ldots z_{\ell+1}^{j_{\ell+1}})^* \;,
$$
and using the commutation rules of the algebra ($z_iz_j=qz_jz_i$ for all $i<j$)
we get
\begin{align*}
c_N &=\sum_{i=1}^{\ell+1}\sum_{j_1+...\,+j_{\ell+1}=N}
[j_1,...\,,j_{\ell+1}]!\,q^{-2\sum_{k=i+1}^{\ell+1}j_k}
z_1^{j_1}...\, z_i^{j_i+1}...\, z_{\ell+1}^{j_{\ell+1}}
(z_1^{j_1}...\, z_i^{j_i+1}...\, z_{\ell+1}^{j_{\ell+1}})^* \\
&=\sum_{i=1}^{\ell+1}\sum_{j_1+...\,+j_{\ell+1}=N+1}
[j_1,...\,,j_i-1,...\,,j_{\ell+1}]!\,q^{-2\sum_{k=i+1}^{\ell+1}j_k}
z_1^{j_1}...\, z_i^{j_i}...\, z_{\ell+1}^{j_{\ell+1}}
(z_1^{j_1}...\, z_i^{j_i}...\, z_{\ell+1}^{j_{\ell+1}})^* \\
\intertext{and using (\ref{eq:qmulti}) we get}
&=\sum_{j_1+\ldots+j_{\ell+1}=N+1}
[j_1,\ldots,j_i,\ldots,j_{\ell+1}]!
z_1^{j_1}\ldots z_i^{j_i}\ldots z_{\ell+1}^{j_{\ell+1}}
(z_1^{j_1}\ldots z_i^{j_i}\ldots z_{\ell+1}^{j_{\ell+1}})^* \\
&=c_{N+1} \;.
\end{align*}
This proves the inductive step.
\end{proof}

As a consequence, the vector $\Psi_N=(\psi^N_{j_1,\ldots,j_{\ell+1}})$ with components
\begin{equation}\label{eq:Psi}
\psi^N_{j_1,\ldots,j_{\ell+1}}:=[j_1,\ldots,j_{\ell+1}]!^{\frac{1}{2}}
(z_1^{j_1}\ldots z_{\ell+1}^{j_{\ell+1}})^* \;,\qquad
\forall\;j_1+\ldots+j_{\ell+1}=N \;,
\end{equation}
satisfies $\Psi_N^\dag\Psi_N=1$. Thus $P_N:=\Psi_N\Psi_N^\dag$ is a projection.

We observe that the constraint \eqref{eq:Psi}, $j_1,\ldots,j_{\ell+1}\in\N$ and $j_1+j_2+\ldots+j_{\ell+1}=N$,
holds if and only if
$$
1\leq j_1+1<j_1+j_2+2<\ldots<j_1+\ldots+j_\ell+\ell\leq j_1+\ldots+j_{\ell+1}+\ell=N+\ell \;.
$$
Therefore, the size of $P_N$, that is the number of $j_1,\ldots,j_{\ell+1}$ which sum up to $N$,
equals the number of $\ell$-partitions of $N+\ell$, which is $\binom{N+\ell}{\ell}$.

Since $\Psi_N\za K_i=\Psi_N$ for $i\neq\ell$, and $\Psi_N\za K_\ell=q^{-N/2}\Psi_N$,
we have
$$
\mL_{\hat{K}}\Psi_N=
\Psi_N\za\hat{K}^{-1}=q^{\frac{\ell}{\ell+1}N}\Psi_N \;.
$$
Let
$$
\Gamma_N:=\{a\in\A(S^{\ell+1}_q)\,|\,\mL_{\hat{K}}a=q^{\frac{\ell}{\ell+1}N}a\} \;.
$$
The map
\begin{align*}
\phi&:\Gamma_N\to\Aq{\ell}^{\binom{N+\ell}{\ell}}P_N \;,\qquad a\mapsto a\,\Psi_N^\dag \;,\\
\phi^{-1}\!\!&:\Aq{\ell}^{\binom{N+\ell}{\ell}}P_N \to\Gamma_N\;,\qquad v\mapsto v\cdot\Psi_N \;,
\end{align*}
with $v=(v_{j_1,\ldots,j_\ell}^N)$ a row vector and $v\cdot\Psi_N$ the scalar product,
is an isomorphism of left $\Aq{\ell}$-modules.
The above discussion can be mirrored, exchanging the role of $z_i$ and $z_i^*$, to get a
projective module description of the modules $\Gamma_N$ with $N<0$. Notice that for
$q=1$, the projection giving $\Gamma_N$ is simply the transpose of the projection
giving $\Gamma_{-N}$.
The special case $\mathcal{S}:=\Gamma_{\frac{1}{2}(\ell+1)}$,
that exists only if $\ell$ is odd, is the $q$-analogue of the module of
sections on the square root of the canonical bundle over $\CP^\ell$.
In the case $\ell=1$ (i.e.~for the standard Podle\'s sphere), such a module
has been discussed in Sec.~\ref{sec:zero}.

\section{Antiholomorphic forms on $\CP^\ell_q$}\label{sec:tre}

Recall that $\Kq{\ell}$ is generated by $\Uq{\ell}$, $\hat{K}$
and its inverse. Since it is a central extension of $\Uq{\ell}$, any representation of
$\lambda:\Uq{\ell}\to\mathrm{End}(V)$ can be lifted to a representation
$\tilde{\lambda}:\Kq{\ell}\to\mathrm{End}(V)$ by defining $\tilde{\lambda}(\hat{K})$
to be any non-zero (and non-negative) multiple of the identity endomorphism.
For the representation $\sigma_k:\Uq{\ell}\to\mathrm{End}(W_k)$ studied in Sec.~\ref{sec:3.2},
and for any $N\in\N$, we define a lift $\sigma^N_k$ of $\sigma_k$ as
$$
\sigma^N_k(\hat{K})=q^{k-\frac{\ell}{\ell+1}N}\cdot id_{W_k} \;.
$$

\subsection{The algebra of forms}

We call $\Omega^k_N:=\mathfrak{M}(\sigma^N_k)$ and notice that $\Omega^k:=\Omega^k_0$
is the $q$-analogue of the module antiholomorphic $k$-forms, and
$$
\Omega^k_N\simeq\Omega^k\otimes_{\Aq{\ell}}\Gamma_N
$$
consists of $k$-forms twisted with the `line bundle' $\Gamma_N$. For odd $\ell$, chiral spinors are
given by
$$
\mathcal{S}_k:=\Omega^k_{\frac{1}{2}(\ell+1)}\simeq\Omega^k\otimes_{\Aq{\ell}}\mathcal{S} \;.
$$
When $q=1$ we have $\Omega^k=\wedge^k\Omega^1$, where the antisymmetric tensor product
is over the algebra $\A(\CP^\ell)$. Being $\Omega^1$ the module of antiholomorphic $1$-forms
(classically, $\deb a=(a\za X_1,\ldots, a\za X_\ell)^t\in\Omega^1$ for all $a\in\A(\CP^\ell)$,
where $X_i$ are defined in Lemma \ref{lemma:Xi}), $\Omega^k$ are
antiholomorphic $k$-forms, and $\mathcal{S}_k$ are chiral spinors. This justifies
the our terminology. The geometrical meaning of $\za\hat{K}$ acting on forms is also clear:
eigenspaces consist of homogeneous forms, and the eigenvalue is $q^{\mathrm{degree}}$ of the
forms.

Notice that $\Omega^0_0=\Omega^\ell_{\ell+1}=\Aq{\ell}$, and that
$\Omega^k_N$ is an $\Aq{\ell}$-bimodule and a left $\Aq{\ell}\rtimes\Uq{\ell+1}$-module (for all $k,N$).
An associative product $\wprod:\Omega^k_N\times\Omega^{k'}_{N'}\to\Omega^{k+k'}_{N+N'}$
is obtained by composing the product in the Grassmann algebra $\Gr$ and
the product in $\Oq$ (and extended linearly). For $\omega=av\in\Omega^k_N$
and $\omega'=a'v'\in\Omega^{k'}_{N'}$, with $a,a'\in\Oq$, $v\in W_k$, $v'\in W_{k'}$,
the product is given by
$$
\omega\wprod\omega':=a\cdot a'(v\wprod v') \;.
$$
In particular when $N=N'=0$, we have that $\Omega^\bullet:=\bigoplus_{k=0}^\ell\Omega^k$ is a graded associative algebra.

Let us check that the above product is really an element of $\Omega^{k+k'}_{N+N'}$. By (\ref{eq:wedgecov}),
\begin{align*}
\sigma_{k+k'}(S(h_{(1)}))(\omega\wprod\omega')\za h_{(2)}
&=(a\za h_{(3)})\cdot (a'\za h_{(4)})
\sigma_{k}(S(h_{(2)}))v\wprod \sigma_{k'}(S(h_{(1)}))v' \\
&=
\sigma_{k}(S(h_{(2)}))\omega\za h_{(3)}\wprod
\sigma_{k'}(S(h_{(1)}))\omega'\za h_{(4)} \\
\intertext{and since by hypothesis $\omega\in\Omega^k_N$, i.e.
$\sigma_{k}(S(h_{(2)}))\omega\za h_{(3)}=\epsilon(h_{(2)})\omega$, this is}
&=\omega\wprod\sigma_{k'}(S(h_{(1)}))\omega'\za h_{(2)} \\
\intertext{that using $\omega'\in\Omega^{k'}_{N'}$ gives}
&=\epsilon(h)\omega\wprod\omega' \;.
\end{align*}
Since we have also $(\omega\wprod\omega')\za\hat{K}
=(\omega\za\hat{K})\wprod(\omega'\za\hat{K})=q^{k+k'-\frac{\ell}{\ell+1}(N+N')}$,
we conclude that $\omega\wprod\omega'\in\Omega^{k+k'}_{N+N'}$ for all
$\omega\in\Omega^k_N$ and $\omega'\in\Omega^{k'}_{N'}$.
Associativity follows from associativity of both the products
in $\Oq$ and $\Gr$.

We need few more lemmas.

\begin{lemma}
We have
\begin{equation}\label{eq:using}
\omega_{1,2,\ldots,k,r}\za M_{r,s-1}=\omega_{1,2,\ldots,k,s} \;.
\end{equation}
for all $\omega\in\Omega^k_N$ and for all $0\leq k<r<s\leq\ell$.
\end{lemma}

\begin{proof}
By definition of $\Omega^k_N$, $\omega_{\underline{i}}\za E_j=\{\sigma_k(E_j)\omega\}_{\underline{i}}$
for all $\omega\in\Omega^k_N$ and all $1\leq j\leq\ell-1$, and by (\ref{ar:E}) we have:
$$
\omega_{1,2,\ldots,k,r}\za E_j=\delta_{r,j}\,
\omega_{1,2,\ldots,k,r+1} \;,
$$
for all $j>k$.
Using $M_{r,s-1}=[E_r,M_{r+1,s-1}]_q$ we get
$$
\omega_{1,2,\ldots,k,r}\za M_{r,s-1}=
\omega_{1,2,\ldots,k,r+1}\za M_{r+1,s-1} \;,
$$
which by iteration gives \eqref{eq:using}.
\end{proof}

\begin{lemma}\label{lemma:Dsquare}
For all $\omega\in\Omega^k_N$ and for all $1\leq r\leq k<s\leq\ell$, we have
$$
\omega_{1,\ldots,\hat{r},\ldots,k-1,k,s}\za S^{-1}(M_{r,s-1}^*)=
(-q)^{k-r+1}\omega_{1,\ldots,k} \;.
$$
Here $\hat{r}$ means that $r$ is omitted from the list.
\end{lemma}
\begin{proof}
By definition of $\Omega^k_N$, $\omega_{\underline{i}}\za F_j=\{\sigma_k(F_j)\omega\}_{\underline{i}}$
for all $\omega\in\Omega^k_N$ and all $1\leq j\leq\ell-1$, and by (\ref{ar:F}) we have
$$
\omega_{1,\ldots,\hat{i},\ldots,k+1}\za F_j=\delta_{i,j}\,
\omega_{1,\ldots,\,\widehat{i+1\!}\,,\ldots,k+1} \;.
$$
The equation $M_{r,s-1}=[E_r,M_{r+1,s-1}]_q$ implies
$S^{-1}(M^*_{r,s-1})=-q[F_r,S^{-1}(M^*_{r+1,s-1})]_q$. Using this,
we get
\begin{align*}
\omega_{1,\ldots,\hat{r},\ldots,k+1}\za S^{-1}(M_{r,k}^*)
&=-q\omega_{1,\ldots,\,\widehat{r+1\!}\,,\ldots,k+1}\za S^{-1}(M^*_{r+1,k}) \\
\intertext{and by iteration}
&=(-q)^{k-r}\omega_{1,\ldots,k-1,\hat{k},k+1}\za S^{-1}(M^*_{k,k}) \\
&=(-q)^{k-r+1}\omega_{1,\ldots,k-1,\hat{k},k+1}\za F_k \\
&=(-q)^{k-r+1}\omega_{1,\ldots,k} \;.
\end{align*}
This proves the case $s=k+1$. In the same way from
$$
\omega_{1,\ldots,\hat{r},\ldots,k-1,k,s}\za F_n=(1-\delta_{n,k})
\bigl(\delta_{n,r}\omega_{1,\ldots,\,\widehat{r+1\!}\,,\ldots,k-1,k,s}
+\delta_{n,s-1}\omega_{1,\ldots,\hat{r},\ldots,k-1,k,s-1}\bigr)
$$
we get
\begin{align*}
\omega_{1,\ldots,\hat{r},\ldots,k,s}\za S^{-1}(M_{r,k}^*)
&=-q\omega_{1,\ldots,\,\widehat{r+1\!}\,,\ldots,k+1}\za S^{-1}(M^*_{r+1,k}) \\
\intertext{and by iteration}
&=(-q)^{k-r+1}\omega_{1,\ldots,k-1,s}\za F_k=0 \;,\\
\intertext{and also}
\omega_{1,\ldots,\hat{r},\ldots,k,s}\za S^{-1}(M_{k+1,s-1}^*)
&=\omega_{1,\ldots,\hat{r},\ldots,k,s}\za F_{s-1}S^{-1}(M_{k+1,s-2}^*) \\
&=\omega_{1,\ldots,\hat{r},\ldots,k,s-1}\za S^{-1}(M_{k+1,s-2}^*) \\
\intertext{and by iteration}
&=-q\omega_{1,\ldots,\hat{r},\ldots,k+1} \;,
\end{align*}
for all $s>k+1$. For all $1\leq r\leq k$ and $k+1<s\leq\ell$ we have
$M_{r,s-1}=[M_{r,k},M_{k+1,s-1}]_q$, thus
$S^{-1}(M_{r,s-1}^*)=[S^{-1}(M_{r,k}^*),S^{-1}(M_{k+1,s-1}^*)]_q$ and
\begin{align*}
\omega_{1,\ldots,\hat{r},\ldots,k,s}\za S^{-1}(M_{r,s-1}^*)
&=\omega_{1,\ldots,\hat{r},\ldots,k,s}\za [S^{-1}(M_{r,k}^*),S^{-1}(M_{k+1,s-1}^*)]_q \\
&=-q^{-1}\omega_{1,\ldots,\hat{r},\ldots,k,s}\za S^{-1}(M_{k+1,s-1}^*)S^{-1}(M_{r,k}^*) \\
&=\omega_{1,\ldots,\hat{r},\ldots,k+1}\za S^{-1}(M_{r,k}^*) \\
&=(-q)^{k-r+1}\omega_{1,2,\ldots,k} \;.
\end{align*}
This concludes the proof.
\end{proof}

\begin{lemma}
We have
\begin{equation}\label{eq:etaM}
\eta_{i+1}\za S^{-1}(M^*_{j,k})=-q\delta_{i,k}\eta_j \;,
\end{equation}
for all $\eta\in\Omega^1_N$ and for all $0\leq i\leq\ell-1$
and $1\leq j\leq k\leq\ell-1$.
\end{lemma}
\begin{proof}
From $\eta_{i+1}\za F_n=\delta_{i,n}\eta_i$, we get
$$
\eta_{i+1}\za S^{-1}(M^*_{j,k})=0\qquad\mathrm{if}\;
i<j\;\mathrm{or}\;i>k \;.
$$
We have $M_{j,k}=[M_{j,n},M_{n+1,k}]_q$,
i.e.~$S^{-1}(M^*_{j,k})=[S^{-1}(M^*_{j,n}),S^{-1}(M^*_{n+1,k})]_q$
for all $j\leq n<k$.\linebreak
In particular $S^{-1}(M^*_{j,i})=-q[S^{-1}(M^*_{j,i-1}),F_i]_q$ applied
to $\eta_{i+1}$ gives
$$
\eta_{i+1}\za S^{-1}(M^*_{j,i})=
\eta_{i+1}\za F_i\,S^{-1}(M^*_{j,i-1})=
\eta_i\za S^{-1}(M^*_{j,i-1})
$$
that iterated gives
$$
\eta_{i+1}\za S^{-1}(M^*_{j,i})=
\eta_{j+1}\za S^{-1}(M^*_{j,j})=
-q\eta_{j+1}\za F_j=
-q\eta_j \;.
$$
Finally, for $j\leq i<k$ we have
\begin{align*}
\eta_{i+1}\za S^{-1}(M^*_{j,k}) &=
\eta_{i+1}\za [S^{-1}(M^*_{j,i}),S^{-1}(M^*_{i+1,k})]_q=
\eta_{i+1}\za S^{-1}(M^*_{j,i})\,S^{-1}(M^*_{i+1,k}) \\
&=-q\eta_j\za S^{-1}(M^*_{i+1,k})=0 \;.
\end{align*}
This concludes the proof.
\end{proof}

The following lemma simplifies considerably the computations with forms.

\begin{lemma}\label{lemma:trick}
Two elements $\omega,\omega'\in\Omega^k_N$ are equal if and only if
$\,\omega_{1,2,\ldots,k}=\smash[b]{\omega'}_{1,2,\ldots,k}$.
Moreover, for all $\omega,\omega'\in\Omega^k_N$ and $\underline{i}\in\Lambda_k$, we have
\begin{equation}\label{eq:trick}
\varphi(\omega_{\underline{i}}^*\omega'_{\underline{i}})=
\tfrac{q^{2|\underline{i}|-k(\ell+1)}}{\dim_q W_k}\inner{\omega,\omega'} \;,
\end{equation}
where the quantum dimension $\dim_q W_k$ is given explicitly in Lemma \ref{lemma:dimq}.
\end{lemma}

\begin{proof}
Since
$$
\omega_{\underline{i}}=
\omega_{1,2,\ldots,k}\za (E_kE_{k+1}\ldots E_{i_k-1})\ldots
(E_2E_3\ldots E_{i_2-1})(E_1E_2\ldots E_{i_1-1}) \;,
$$
for all $\omega\in\Omega^k_N$ and all $\underline{i}\in\Lambda_k$,
if $\omega_{1,2,\ldots,k}$ is zero, all components of $\omega$ are
zero. By linearity, this proves the first claim.

By inverting the previous transformation we get
\begin{equation}\label{eq:iF}
\omega_{1,2,\ldots,k}=
\omega_{\underline{i}}\za (F_{i_1-1}F_{i_1-2}\ldots F_1)
(F_{i_2-1}F_{i_2-2}\ldots F_2)\ldots
(F_{i_k-1}F_{i_k-2}\ldots F_k) \;.
\end{equation}
Assume $j\#\underline{i}=-1$, so that $\omega'_{\underline{i}}\za K_j=q^{-\frac{1}{2}}\omega'_{\underline{i}}$,
$\omega_{\underline{i}}^*\za K_j^{-1}=q^{-\frac{1}{2}}\omega_{\underline{i}}^*$,
$$
\omega'_{\underline{i}}\za F_jE_j=
\{\sigma_k(F_j)\omega'\}_{\underline{i}}\za E_j=
\omega'_{\underline{i}^{\,j,-}}\za E_j=
\{\sigma_k(E_j)\omega'\}_{\underline{i}^{\,j,-}}=
\omega_{(\underline{i}^{\,j,-})^{\,j,+}}=\omega_{\underline{i}} \;,
$$
and by covariance of the action
\begin{align*}
q^{\frac{1}{2}}\bigl\{\omega_{\underline{i}}^*(\omega'_{\underline{i}}\za F_j)\bigr\}\za E_j
&=\omega_{\underline{i}}^*(\omega'_{\underline{i}}\za F_jE_j)
-q^2(\omega_{\underline{i}}\za F_j)^*(\omega'_{\underline{i}}\za F_j) \\
&=\omega_{\underline{i}}^*\omega'_{\underline{i}}
-q^2(\omega_{\underline{i}}\za F_j)^*(\omega'_{\underline{i}}\za F_j) \;.
\end{align*}
Applying the Haar state $\varphi$ to both sides of this identity we get
$$
0=\varphi\bigl(\omega_{\underline{i}}^*\omega'_{\underline{i}}\bigr)
-q^2\varphi\bigl((\omega_{\underline{i}}\za F_j)^*(\omega'_{\underline{i}}\za F_j)\bigr) \;,
$$
where the left hand side vanishes due to invariance of $\varphi$.
By iterated use of the last equation, together with (\ref{eq:iF}), we get
$$
\varphi(\omega_{\underline{i}}^*\omega'_{\underline{i}})=
q^{2(i_1-1+i_2-2+\ldots+i_k-k)}\varphi(\omega_{1,2,\ldots,k}^*\omega'_{1,2,\ldots,k})=
q^{2|\underline{i}|-k(k+1)}\varphi(\omega_{1,2,\ldots,k}^*\omega'_{1,2,\ldots,k}) \;.
$$
Sum over $\underline{i}$, using Lemma \ref{lemma:dimq}, gives
$$
\inner{\omega,\omega'}=q^{k(\ell-k)}\dim_q\!W_k\;\varphi(\omega_{1,2,\ldots,k}^*\omega'_{1,2,\ldots,k}) \;,
$$
and $\varphi(\omega_{\underline{i}}^*\omega'_{\underline{i}})=
q^{2|\underline{i}|-k(\ell+1)}(\dim_q W_k)^{-1}\inner{\omega,\omega'}$.
This concludes the proof.
\end{proof}

\subsection{Harmonic decomposition of $\Omega^k_N$}
Since left and right canonical actions of $\Uq{\ell+1}$ are mutually commuting, the space
$\bigoplus_{k=0}^\ell\Omega^k_N$, $N\in\Z$, carries a left action of $\Uq{\ell+1}$.
We decompose it into irreducible representations using the Gelfand-Tsetlin basis.

\begin{prop}\label{prop:ha}
As left $\Uq{\ell+1}$-modules we have the following equivalences:{\allowdisplaybreaks%
\begin{align*}
\Omega^0_N&\simeq \begin{cases}
\bigoplus\nolimits_{m\in\N}V_{(m+N,0,\ldots,0,m)} \;, &\mathrm{if}\;N\geq 0\,,\\[5pt]
\bigoplus\nolimits_{m\in\N}V_{(m,0,\ldots,0,m-N)} \;, &\mathrm{if}\;N<0\,,
\end{cases} \\[8pt]
\Omega^k_N&\simeq \begin{cases}
\bigoplus\nolimits_{m\in\N}V_{(m+N-k,0,\ldots,0,m)+\,\underline{e}_k}
\oplus V_{(m+N-k-1,0,\ldots,0,m)+\,\underline{e}_{k+1}} 
\;, &\mathrm{if}\;1\leq k\leq\min(N,\ell)-1\,,\\[5pt]
\bigoplus\nolimits_{m\in\N}V_{(m,0,\ldots,0,m-N+k)+\,\underline{e}_k}
\oplus V_{(m,0,\ldots,0,m-N+k+1)+\,\underline{e}_{k+1}} 
\;, &\mathrm{if}\;\max(1,N)\leq k\leq \ell-1\,,\\
\end{cases} \\[8pt]
\Omega^\ell_N&\simeq \begin{cases}
\bigoplus\nolimits_{m\in\N}V_{(m,0,\ldots,0,m-N+\ell+1)} \;, &\mathrm{if}\;N\leq\ell\,,\\[5pt]
\bigoplus\nolimits_{m\in\N}V_{(m+N-\ell-1,0,\ldots,0,m)} \;, &\mathrm{if}\;N>\ell\,,
\end{cases}
\end{align*}
}where $\underline{e}_{\,k}=(e_k^1,\ldots,e_k^\ell)$ is the $\ell$-tuple with
components $e_k^j=\delta^j_k$ and $V_n$ is the vector space carrying the irreducible
representation $\rho_n$.
The trivial representation $V_{(0,0,\ldots,0)}$ appears (with multiplicity $1$) in
just two cases: as a subrepresentation of $\Omega^0_N$ for $N=0$, and as
a subrepresentation of $\Omega^\ell_N$ for $N=\ell+1$.
\end{prop}
\begin{proof}
By Peter-Weyl theorem, the algebra $\Oq$ is a multiplicity free direct sum
of representations $\rho_n\otimes\rho_n^*$ of $\Uq{\ell+1}\otimes\Uq{\ell+1}$,
where $n=(n_1,\ldots,n_\ell)$ runs in $\N^\ell$ and $\rho_n^*$ is the representation
dual to $\rho_n$. We have
$$
\Omega_N^k\simeq\bigoplus\nolimits_{n}(\sigma^N_k,\rho_n^*)V_n
$$
where $(\sigma^N_k,\rho_n^*)$ is the multiplicity of the irreducible representation
$\sigma^N_k$ of $\Kq{\ell}$ inside the representation $\rho_n^*$ of $\Uq{\ell+1}$.
We have $\rho_n^*\simeq\rho_{n'}$, where $n':=(n_\ell,\ldots,n_1)$.
A basis of $V_{n'}$ is given by Gelfand-Tsetlin tableaux
(GT tableaux, for short), which are arrays of integers of the form
\begin{equation}\label{eq:GTT}
\left(\begin{array}{cccccc}
r_{1,1} & r_{1,2} & \ldots & r_{1,\ell} & 0 \\
r_{2,1} & r_{2,2} & \ldots & r_{2,\ell} \\
\ldots & \ldots & \ldots \\
r_{\ell,1} & r_{\ell,2} \\
r_{\ell+1,1}
\end{array}\right)
\end{equation}
with $r_{i,j}\geq r_{i+1,j}\geq r_{i,j+1}$ for all
$i,j$ (from now on it is understood that $r_{ij}=0$ if $i,j$ are out of
range, i.e.~if $i<1$, $j<1$, or $i+j>\ell+2$). The highest weight $n'$
has entries $n'_i=r_{1,i}-r_{1,i+1}$.
Usually GT tableaux are defined modulo a global rescaling
(two arrays are equivalent if they differ by a constant).
Here for each equivalence class of GT tableaux, we choose the representative
which has zero in the top-right corner, $r_{1,\ell+1}=0$.
If we remove the first row from (\ref{eq:GTT}) we obtain a GT tableu of
$\Uq{\ell}$ with highest weight $m=(m_1,\ldots,m_{\ell-1})$ given
by $m_i=r_{2,i}-r_{2,i+1}$.
In particular, $\sigma_k$ appears (with multiplicity $1$) in $\rho_{n'}$ if and only if
\begin{equation}\label{eq:br}
n'=(n'_1,0,\ldots,0,n'_\ell)+\underline{e}_{\,\ell-k}
\qquad\mathrm{or}\qquad
n'=(n'_1,0,\ldots,0,n'_\ell)+\underline{e}_{\,\ell-k+1} \;,
\end{equation}
for $1\leq k\leq\ell-1$, and if and only if $n'=(n'_1,0,\ldots,0,n'_\ell)$ for $k=0,\ell$.

In Gelfand-Tsetlin notations~\cite{GT88a}, $K_i=q^{\frac{1}{2}H_i}$ is
represented by $q^{\frac{1}{2}(E_{i+1,i+1}-E_{i,i})}$, thus
$$
\hat{K}^{\ell+1}=q^{-\ell E_{\ell+1,\ell+1}+\sum_{i=1}^\ell E_{i,i}} \;.
$$
The eigenvalue of $E_{i,i}$ when applied to the generic GT tableau is
$$
\sum\nolimits_{j=1}^ir_{\ell+2-i,j}-\sum\nolimits_{j=1}^{i-1}r_{\ell+3-i,j} \;,
$$
and then the eigenvalue of $\hat{K}$ is
$$
q^{\sum_{j=1}^\ell(\frac{\ell}{\ell+1}r_{1,j}-r_{2,j})} \;.
$$
So, for $k=0,\ell$ we have $\hat{K}^{\ell+1}=q^{\ell(n'_1-n'_\ell)}\cdot id$, and
for $1\leq k\leq\ell-1$ we have, in the two cases listed in (\ref{eq:br}),
$\hat{K}^{\ell+1}=q^{\ell(n'_1-n'_\ell)-\ell+k}\cdot id\,$ resp.
$\hat{K}^{\ell+1}=q^{\ell(n'_1-n'_\ell)+k}\cdot id\,$.
In each case, the eigenvalue of $\hat{K}^{\ell+1}$ must be equal to $q^{(\ell+1)k-\ell N}$, that means
\begin{align*}
\mathrm{if}\;k=0 & \Rightarrow n'_1=n'_\ell-N \;, \\
\mathrm{if}\;1\leq k\leq\ell-1 & \Rightarrow
n'_1=n'_\ell+k+1-N \;\mathrm{or}\;n'_1=n'_\ell+k-N  \;, \\
\mathrm{if}\;k=\ell & \Rightarrow n'_1=n'_\ell+\ell+1-N \;.
\end{align*}
Recalling that $n_j=n'_{\ell+1-j}$, we get Proposition \ref{prop:ha}.
\end{proof}

\subsection{The Dolbeault operator}

We take as a $q$-analogue of the Dolbeault operator
\begin{equation}\label{eq:Dolbeault}
\deb:=\textstyle{\sum_{i=1}^\ell}\,\mL_{\hat{K}X_i}\otimes\mathfrak{e}^L_{e^i} \;,
\end{equation}
where $X_i$ are the operator in Lemma \ref{lemma:Xi} and $\mathfrak{e}^L$ is the exterior product
in \eqref{eq:extpro}.
Explicitly, for any $\omega\in\Omega^k_N$,
$$
(\deb\omega)_{\underline{i}}=\sum\nolimits_{r=1}^{k+1}\,(-q)^{1-r}\,
\omega_{\underline{i}\smallsetminus i_r}\za S^{-1}(\hat{K}X_{i_r}) \;.
$$
Being $h\mapsto\mL_h$ a $*$-representation, the Hermitian conjugate $\deb^\dag$ of $\deb$ is
\begin{equation}\label{eq:debdag}
\deb^\dag=\textstyle{\sum_{i=1}^\ell}\,\mL_{X^*_i\hat{K}}\otimes\mathfrak{i}^L_{e^i} \;.
\end{equation}

\begin{prop}
The operator $\deb$ maps $\Omega^k_N$ in $\Omega^{k+1}_N$ and satisfies $\deb^2=0$.
The operator $\deb^\dag$ maps $\Omega^{k+1}_N$ in $\Omega^k_N$ and satisfies $(\deb^\dag)^2=0$.
\end{prop}
\begin{proof}
Since $h\to\mL_h$ is a representation, the algebraic identity
$ht=\bigl(t\adj S^{-1}(h_{(2)})\bigr)h_{(1)}$ implies
$$
\mL_h\mL_t=\mL_{t\adj S^{-1}(h_{(2)})}\mL_{h_{(1)}}
$$
for all $h,t\in\Uq{\ell+1}$. By (\ref{eq:wedgecov}) we have also
$$
\sigma^N_{k+1}(h)(x\wprod v)=
\{\sigma^0_1(h_{(1)})x\}\wprod\{\sigma^N_k(h_{(2)})v\}
$$
for all $h\in\Kq{\ell}$, $x\in W_1$ and $v\in W_k$, that means
$\sigma^N_{k+1}(h)\mathfrak{e}^L_{x}=
\mathfrak{e}^L_{\sigma^0_1(h_{(1)})x}\cdot\sigma^N_k(h_{(2)})$. Thus, for all $h\in\Kq{\ell}$,
\begin{align*}
\bigl\{\mL_{h_{(1)}}\otimes\sigma^N_{k+1}(h_{(2)})\bigr\}\deb
&=\textstyle{\sum_{i=1}^\ell}\,\mL_{h_{(1)}}\mL_{\hat{K}X_i}\otimes\sigma^N_{k+1}(h_{(2)})\mathfrak{e}^L_{e^i} \\
&=\textstyle{\sum_{i=1}^\ell}\,\mL_{\hat{K}X_i\adj S^{-1}(h_{(2)})}\mL_{h_{(1)}}\otimes
\mathfrak{e}^L_{\sigma^0_1(h_{(3)})e^i}\cdot\sigma^N_k(h_{(4)}) \;.
\end{align*}
But $\hat{K}$ commutes with all $h\in\Kq{\ell}$, and
$\sum_iX_i\adj S^{-1}(h_{(2)})\otimes\sigma^0_1(h_{(3)})e^i=\epsilon(h_{(2)})\sum_iX_i\otimes e^i$
by Lemma \ref{lemma:Xi}. We conclude that
$$
\bigl\{\mL_{h_{(1)}}\otimes\sigma^N_{k+1}(h_{(2)})\bigr\}\deb
=\deb\bigl\{\mL_{h_{(1)}}\otimes\sigma^N_k(h_{(2)})\bigr\} \;,
$$
for all $h\in\Kq{\ell}$.
Hence, $\deb$ maps invariant elements into invariant elements, and
$\deb(\Omega^k_N)\subset\Omega^{k+1}_N$. By adjunction, being all representations
unitary,
$$
\deb^\dag\bigl\{\mL_{h_{(1)}^*}\otimes\sigma^N_{k+1}(h_{(2)}^*)\bigr\}
=\bigl\{\mL_{h_{(1)}^*}\otimes\sigma^N_k(h_{(2)}^*)\bigr\}\deb^\dag \;.
$$
Hence, $\deb^\dag$ maps invariant elements into invariant elements, and
$\deb^\dag(\Omega^{k+1}_N)\subset\Omega^k_N$.

Now we prove that $\deb^2=0$ (and by adjunction $(\deb^\dag)^2=0$). Using the
associativity of the wedge-product, and $X_i\hat{K}=q\hat{K}X_i$ we get
$$
\deb^2\omega=q\omega\wprod\za\, S^{-1}(X\wprod X)\hat{K}^{-2} \;,
$$
where $X=\sum_ie^iX_i$. But from (\ref{eq:debsquare}) it follows that
$$
(X\wprod X)_{i_1,i_2}=X_{i_1}X_{i_2}-q^{-1}X_{i_2}X_{i_1}=0
$$
for all $i_1<i_2$, thus $X\wprod X=0$ and $\deb^2=0$.
\end{proof}

Thus for any $N$, we have a left $\Uq{\ell+1}$-covariant cohomology complex $(\Omega_N^\bullet,\deb)$
over $\Omega^0_N$. In particular if $N=0$, this complex gives a differential calculus over $\Aq{\ell}$
(if $N\neq 0$, $\Omega_N^\bullet$ is not closed under the wedge product).
In fact, we now prove two different Leibniz properties of $\deb$.

\begin{lemma}
We have
\begin{equation}\label{eq:gLeibzeroN}
\deb(a\omega)=a(\deb\omega)+(\deb a)\wprod\omega \;,
\end{equation}
for all $a\in\Aq{\ell}$ and $\omega\in\Omega^k_N$.
\end{lemma}

\begin{proof}
Since $S^{-1}$ is anticomultiplicative, by applying $S^{-1}$ to (\ref{eq:Xcop}) we get
\begin{align}
\Delta(S^{-1}(\hat{K}X_i)) &=S^{-1}(\hat{K}X_i)\otimes 1+
N_{i,i-1}^{-1}\hat{K}^{-1}\otimes S^{-1}(\hat{K}X_i) \notag \\
&+q^{-\frac{1}{2}}(q-q^{-1})\sum_{j=i}^{\ell-1}
\hat{K}^{-1}N_{i,j}^{-1}S^{-1}(M_{i,j}^*)\otimes S^{-1}(\hat{K}X_{j+1}) \;.\label{eq:XcopS}
\end{align}
With this, by using the covariance of the right action and the right $\Kq{\ell}$-invariance of $a$, we prove the identity in (\ref{eq:gLeibzeroN}):
\begin{align*}
\{\deb(a\omega)\}_{\underline{i}} &=
\sum\nolimits_{r=1}^{k+1}\,(-q)^{1-r}\,
(a\omega_{\underline{i}\smallsetminus i_r})\za S^{-1}(\hat{K}X_{i_r}) \\
&=\sum\nolimits_{r=1}^{k+1}\,(-q)^{1-r}\,
\Big\{(a\za S^{-1}(\hat{K}X_i))\omega_{\underline{i}\smallsetminus i_r}
+(a\za N_{i,i-1}^{-1}\hat{K}^{-1})(\omega_{\underline{i}\smallsetminus i_r}\za S^{-1}(\hat{K}X_i))\Big\} \\
&+q^{-\frac{1}{2}}(q-q^{-1})\sum\nolimits_{r=1}^{k+1}\,(-q)^{1-r}\,\sum\nolimits_{j=i}^{\ell-1}
(a\za \hat{K}^{-1}N_{i,j}^{-1}S^{-1}(M_{i,j}^*))(\omega_{\underline{i}\smallsetminus i_r}\za S^{-1}(\hat{K}X_{j+1})) \\
&=\sum\nolimits_{r=1}^{k+1}\,(-q)^{1-r}\,
\Big\{(a\za S^{-1}(\hat{K}X_i))\omega_{\underline{i}\smallsetminus i_r}
+a(\omega_{\underline{i}\smallsetminus i_r}\za S^{-1}(\hat{K}X_i))\Big\} \\
&=\{(\deb a)\wprod\omega+a(\deb\omega)\}_{\underline{i}} \;. &\!\!\qedhere
\end{align*}
\end{proof}

Hence, the commutator of $\deb$ with the operator of left multiplication by a
`function' $a$ gives the left multiplication by the differential of $a$, and
this will be useful to construct spectral triples. The second Leibniz property
-- that is more difficult to prove -- tells us that $\deb$ is a graded derivation
on $\Omega^\bullet_0$.

\begin{prop}\label{prop:5.4}
The datum $(\Omega^\bullet_0,\deb)$ is a left $\Uq{\ell+1}$-covariant differential calculus over $\Aq{\ell}$.
That is, for all $\omega,\omega'\in\Omega^\bullet_0$:
\begin{equation}\label{eq:gLeib}
\deb(\omega\wprod\omega')=(\deb\omega)\wprod\omega'+(-1)^k\omega\wprod(\deb\omega') \;,
\end{equation}
where $k$ is the degree of $\omega$.
\end{prop}

\begin{proof}
If $k=0$, (\ref{eq:gLeib}) is a particular case of (\ref{eq:gLeibzeroN}).

Now we consider the case $k=1$. Let $\eta\in\Omega^1_0$ and $\omega'\in\Omega^{k'}_0$.
By Lemma \ref{lemma:trick}, in order to prove the equation
\begin{equation}\label{eq:gLeibkone}
\deb(\eta\wprod\omega')=(\deb\eta)\wprod\omega'-\eta\wprod(\deb\omega')
\end{equation}
it is enough to show that
$$
\{\deb(\eta\wprod\omega')\}_{1,2,\ldots,k'+2}=
\{(\deb\eta)\wprod\omega'\}_{1,2,\ldots,k'+2}
-\{\eta\wprod(\deb\omega')\}_{1,2,\ldots,k'+2} \;.
$$
We have
\begin{align*}
\{\deb(\eta\wprod\omega')\}_{1,2,\ldots,k'+2}
&=\sum_{1\leq s<r\leq k'+2}\,(-q)^{2-r-s}\,
(\eta_s\,\omega'_{(1,\ldots,k'+2)\smallsetminus\{r,s\}})\za S^{-1}(\hat{K}X_r) \\
&+\sum_{1\leq r<s\leq k'+2}\,(-q)^{3-r-s}\,
(\eta_s\,\omega'_{(1,\ldots,k'+2)\smallsetminus\{r,s\}})\za S^{-1}(\hat{K}X_r) \;,
\end{align*}
and by (\ref{eq:XcopS}):
\begin{align*}
\{\deb(\eta\wprod\omega')-(\deb\eta)\wprod\omega'\}_{1,\ldots,k'+2}
=\!\!\sum_{1\leq s<r\leq k'+2}\!\!(-q)^{2-r-s}\,(\eta_s\za N_{r,r-1}^{-1}\hat{K}^{-1})
(\omega'_{(1,\ldots,k'+2)\smallsetminus\{r,s\}}\za S^{-1}(\hat{K}X_r))
\\
+q^{-\frac{1}{2}}(q-q^{-1})\!\!\sum_{1\leq s<r\leq k'+2}\sum_{j=r}^{\ell-1}\,
(-q)^{2-r-s}\,(\eta_s\za \hat{K}^{-1}N_{r,j}^{-1}S^{-1}(M_{r,j}^*))
(\omega'_{(1,\ldots,k'+2)\smallsetminus\{r,s\}}\za S^{-1}(\hat{K}X_{j+1}))
\\
+\!\!\sum_{1\leq r<s\leq k'+2}\!\!(-q)^{3-r-s}\,(\eta_s\za N_{r,r-1}^{-1}\hat{K}^{-1})
(\omega'_{(1,\ldots,k'+2)\smallsetminus\{r,s\}}\za S^{-1}(\hat{K}X_r))
\\
+q^{-\frac{1}{2}}(q-q^{-1})\!\!\sum_{1\leq r<s\leq k'+2}\sum_{j=r}^{\ell-1}\,
(-q)^{3-r-s}\,(\eta_s\za \hat{K}^{-1}N_{r,j}^{-1}S^{-1}(M_{r,j}^*))
(\omega'_{(1,\ldots,k'+2)\smallsetminus\{r,s\}}\za S^{-1}(\hat{K}X_{j+1})) \;.
\end{align*}
By (\ref{eq:etaM}), all the terms with $j\neq s-1$ are zero. Thus
\begin{align*}
\{\deb(\eta\wprod\omega')-(\deb\eta)\wprod\omega'\}_{1,\ldots,k'+2}
=\!\!\sum_{1\leq s<r\leq k'+2}\!\!(-q)^{2-r-s}\,(\eta_s\za N_{r,r-1}^{-1}\hat{K}^{-1})
(\omega'_{(1,\ldots,k'+2)\smallsetminus\{r,s\}}\za S^{-1}(\hat{K}X_r))
\\
\qquad +\!\!\sum_{1\leq r<s\leq k'+2}\!\!(-q)^{3-r-s}\,(\eta_s\za N_{r,r-1}^{-1}\hat{K}^{-1})
(\omega'_{(1,\ldots,k'+2)\smallsetminus\{r,s\}}\za S^{-1}(\hat{K}X_r))
\\
\qquad +q^{-\frac{1}{2}}(q-q^{-1})\!\!\sum_{1\leq r<s\leq k'+2}\!\!
(-q)^{3-r-s}\,(\eta_s\za \hat{K}^{-1}N_{r,s-1}^{-1}S^{-1}(M_{r,s-1}^*))
(\omega'_{(1,\ldots,k'+2)\smallsetminus\{r,s\}}\za S^{-1}(\hat{K}X_s)) \;.
\end{align*}
We have
\begin{align*}
\eta_s\za \hat{K}^{-1}N_{r,r-1}^{-1} &=q^{-1}\eta_s
&\mathrm{if}\; r>s \;,\\
\eta_s\za \hat{K}^{-1}N_{r,r-1}^{-1} &=q^{-1}\eta_s
&\mathrm{if}\; r<s \;,\\
\eta_s\za \hat{K}^{-1}N_{r,s-1}^{-1} &=q^{-\frac{3}{2}}\eta_s
&\mathrm{if}\; r<s \;.
\end{align*}
We have also $\eta_s\za S^{-1}(M_{r,s-1}^*)=-q\eta_r$ by (\ref{eq:etaM}).
Thus,
\begin{align*}
\{\deb(\eta\wprod\omega')-(\deb\eta)\wprod\omega'\}_{1,\ldots,k'+2}
=-q^{-2}\sum_{1\leq s<r\leq k'+2}\,(-q)^{3-r-s}\,\eta_s
(\omega'_{(1,\ldots,k'+2)\smallsetminus\{r,s\}}\za S^{-1}(\hat{K}X_r))
\\
\qquad -\sum_{1\leq r<s\leq k'+2}\,(-q)^{2-r-s}\,\eta_s
(\omega'_{(1,\ldots,k'+2)\smallsetminus\{r,s\}}\za S^{-1}(\hat{K}X_r))
\\
\qquad +(q^{-2}-1)\sum_{1\leq r<s\leq k'+2}\,
(-q)^{3-r-s}\,\eta_r
(\omega'_{(1,\ldots,k'+2)\smallsetminus\{r,s\}}\za S^{-1}(\hat{K}X_s)) \;.
\end{align*}
If we exchange $r$ with $s$ in last summation, we get
\begin{align*}
\{\deb(\eta\wprod\omega')-(\deb\eta)\wprod\omega'\}_{1,\ldots,k'+2}
&=-\sum_{1\leq s<r\leq k'+2}\,(-q)^{3-r-s}\,\eta_s
(\omega'_{(1,\ldots,k'+2)\smallsetminus\{r,s\}}\za S^{-1}(\hat{K}X_r))
\\
&\quad -\sum_{1\leq r<s\leq k'+2}\,(-q)^{2-r-s}\,\eta_s
(\omega'_{(1,\ldots,k'+2)\smallsetminus\{r,s\}}\za S^{-1}(\hat{K}X_r)) \\
&=-\{\eta\wprod(\deb\omega')\}_{1,\ldots,k'+2} \;.
\end{align*}
This proves (\ref{eq:gLeibkone}).

Now let $k\geq 1$. Since $\Omega^k_0$ is generated by $1$-forms, we can write
any $k$-form as a sum of elements $\omega=\eta^1\wprod\eta^2\wprod\ldots\wprod\eta^k$,
with $\eta^j\in\Omega^1_0$. By iterated use of (\ref{eq:gLeibkone}) we get
\begin{align*}
\deb(\omega\wprod\omega') &=
\deb(\eta^1\wprod\eta^2\wprod\ldots\wprod\eta^k\wprod\omega') \\
&=(\deb\eta^1)\wprod(\eta^2\wprod\ldots\wprod\eta^k\wprod\omega')
-\eta^1\wprod\deb(\eta^2\wprod\ldots\wprod\eta^k\wprod\omega') \\
&=\sum_{j=1}^k(-1)^{j-1}\eta^1\wprod\ldots\wprod(\de\eta^j)\wprod\ldots\wprod\eta^k\wprod\omega'
+(-1)^k\eta^1\wprod\ldots\wprod\eta^k\wprod(\deb\omega') \\
&=(\deb\omega)\wprod\omega'+(-1)^k\omega\wprod(\deb\omega') \;.
\end{align*}
This concludes the proof.
\end{proof}

\subsection{Reality and the first order condition}
If we tensor the operator $J$ in (\ref{eq:JonW}) with the involution `$*$' on $\Oq$ we
get an antilinear map $*\otimes J:\Omega^\bullet_N\to \Oq\otimes\Gr$.
One can check the the image of
$\Omega^k_N=\Oq\boxtimes_{\sigma^N_k}W_k$ through $*\otimes J$ is in the
space $W_{\ell-k}\boxtimes_{\sigma^{\ell+1-N}_{\ell-k}}\Oq$, that is defined
like $\Omega_{\ell+1-N}^{\ell-k}$ but for the order of the factors in the
Hopf tensor product.
Thus, due to anticocommutativity of $\Uq{\ell}$, the image of $*\otimes J$
is in general not an element of $\Omega^\bullet_{N'}$, for any value of $N'$.
To cure this problem one could compose $*\otimes J$ with the $R$-matrix of $\Uq{\ell}$,
that is trivial if $q=1$ or if $\ell=1$ (so, in the commutative case and also on the standard
Podle\'s sphere we don't have this problem). We prefer to proceed in the following simpler
way.

We introduce
\begin{equation}\label{eq:Rstructure}
\Ja:=(*\otimes J)(\mL_{\maut^{-1}}\otimes id) \;,
\end{equation}
where $\maut$ is defined in (\ref{eq:Ssquare}).

\smallskip

\begin{prop}
The operator $\Ja$ maps $\Omega^k_N$ into $\Omega^{\ell-k}_{\ell+1-N}$.
\end{prop}

\begin{proof}
Let $h\in\Uq{\ell}$.
By (\ref{eq:Jx}) we have
$\sigma_{\ell-k}(h)J=J\sigma_k(S(h^*))$, and $\mL_h\,*=*\,\mL_{S(h^*)}$
being $\Oq$ a right $\Uq{\ell+1}$-module $*$-algebra. Thus, for all $h\in\Uq{\ell}$,
\begin{align*}
\bigl\{\mL_{h_{(1)}}\otimes\sigma_{\ell-k}(h_{(2)})\bigr\}\Ja &=
(*\otimes J)\bigl\{\mL_{S(h_{(1)}^*)}\otimes\sigma_k(S(h_{(2)}^*))\bigr\}
(\mL_{\maut^{-1}}\otimes id) \\
&=\Ja\bigl\{\mL_{\maut t_{(2)}\maut^{-1}}\otimes\sigma_k(t_{(1)})\bigr\}
 =\Ja\bigl\{\mL_{S^2(t_{(2)})}\otimes\sigma_k(t_{(1)})\bigr\} \;,
\end{align*}
where we used \eqref{eq:Sdue} and called $t=S(h^*)$. But
$$
\bigl\{\mL_{S^2(t_{(2)})}\otimes\sigma_k(t_{(1)})\bigr\}\omega=
\sigma_k(t_{(1)})\omega\za S(t_{(2)})=
\sigma_k(t_{(1)}S(t_{(2)}))\omega=
\epsilon(t)\omega
$$
for all $\omega\in\Omega^k_N$ and $t\in\Uq{\ell}$, where in the second equality we used
Lemma~\ref{lemma:equiv}.
Thus, $\Ja$ maps $\Omega^k_N$
into $\bigoplus_{N'\in\Z}\Omega^{\ell-k}_{N'}$, the elements of $\Oq\otimes W_{\ell-k}$
that are invariant under $\mL_{h_{(1)}}\otimes\sigma_{\ell-k}(h_{(2)})$,
$h\in\Uq{\ell}$. Moreover, since $a^*\za\hat{K}=(a\za\hat{K}^{-1})^*$, we have:
$$
q^{\ell-k-\frac{\ell}{\ell+1}N'}\Ja\omega=
(\Ja\omega)\za\hat{K}=\Ja(\omega\za\hat{K}^{-1})=
q^{-\left(k-\frac{\ell}{\ell+1}N\right)}\Ja\omega \;.
$$
It follows that $\Ja(\Omega^k_N)\subset\Omega^{\ell-k}_{N'}$ with $N'=\ell+1-N$.
\end{proof}

The map $\Ja$ is equivariant:
\begin{equation}\label{eq:Jaequiv}
h\az(\Ja\omega)=\Ja\bigl(S(h)^*\az\omega\bigr) \;,\qquad
\forall\;h\in\Uq{\ell+1}\;.
\end{equation}
This follows immediately from the fact that $h\az$ commutes with any endomorphism
of $W_k$ (it acts diagonally on $\Omega^k_N$), while $h\az a^*=\{S(h)^*\az a\}^*$
since $\Oq$ is a left $\Uq{\ell+1}$-module $*$-algebra. The antiunitary part of $\Ja$ is a
natural candidate for the real structure. We are going to compute it explicitly.

\begin{lemma}\label{lemma:Jaomega}
We have
$$
(\Ja\omega)_{\underline{i}}=
(-1)^{|\underline{i}|}q^{\frac{\ell}{2}(N-\frac{\ell+1}{2})}
(\omega_{\,\underline{i}^c}\za \maut^{\frac{1}{2}}\hat{K}^{\ell+1})^* \;,
$$
for all $\omega\in\Omega^k_N$.
\end{lemma}

\begin{proof}
By (\ref{eq:zaK}) we have
$$
\omega_{\underline{i}}\za\maut\hat{K}^{-\ell-1}=q^{k(\ell+1)-2|\underline{i}|}\omega_{\underline{i}} \;.
$$
Thus
\begin{equation}\label{eq:zamaut}
\omega_{\underline{i}^c}\za\maut^{-\frac{1}{2}}\hat{K}^{\ell+1}=
\omega_{\underline{i}^c}\za(\maut\hat{K}^{-\ell-1})^{-\frac{1}{2}}\hat{K}^{\frac{\ell+1}{2}}=
q^{|\underline{i}^c|-\frac{\ell}{2}N}\omega_{\underline{i}^c}=
q^{-|\underline{i}|+\frac{\ell}{2}(\ell+1-N)}\omega_{\underline{i}^c}
\;,
\end{equation}
for all $\omega\in\Omega^k_N$.
With this, we can compute $\Ja\omega$. By (\ref{eq:Rstructure}) and (\ref{eq:JonW}), we have
\begin{align*}
(\Ja\omega)_{\underline{i}} &=
(-q^{-1})^{|\underline{i}|}
q^{\frac{1}{4}\ell(\ell+1)}
(\omega_{\underline{i}^c}\za\maut)^* \\
\intertext{and by (\ref{eq:zamaut})}
&=(-q^{-1})^{|\underline{i}|}q^{\frac{1}{4}\ell(\ell+1)}q^{|\underline{i}|-\frac{\ell}{2}(\ell+1-N)}
(\omega_{\underline{i}^c}\za\maut \maut^{-\frac{1}{2}}\hat{K}^{\ell+1})^* \\
&=(-1)^{|\underline{i}|}q^{\frac{\ell}{2}(N-\frac{\ell+1}{2})}
(\omega_{\underline{i}^c}\za\maut^{\frac{1}{2}}\hat{K}^{\ell+1})^* \;.
\end{align*}
This concludes the proof.
\end{proof}

\begin{lemma}\label{lemma:iso}
For all $\omega,\omega'\in\Omega^k_N$,
$$
\big<\Ja(\maut^{-\frac{1}{2}}\az\omega\za\hat{K}^{-2(\ell+1)}),\Ja(\maut^{-\frac{1}{2}}\az\omega'\za\hat{K}^{-2(\ell+1)})\big>=
q^{\ell(N-\frac{\ell+1}{2})}\inner{\omega',\omega} \;.
$$
\end{lemma}

\begin{proof}
From Lemma \ref{lemma:Jaomega} and equation (\ref{eq:modprop}) we get
\begin{align*}
& \big<\Ja(\maut^{-\frac{1}{2}}\az\omega),\Ja(\maut^{-\frac{1}{2}}\az\omega')\big>
=\sum\nolimits_{\underline{i}\in\Lambda_{\ell-k}}
\varphi\bigl( (\Ja \maut^{\smash[t]{-\frac{1}{2}}}\az\omega)_{\underline{i}}^*
(\Ja \maut^{\smash[t]{-\frac{1}{2}}}\az\omega')_{\underline{i}} \bigr) \\ & \qquad\qquad
=\sum\nolimits_{\underline{i}\in\Lambda_{\ell-k}}
q^{\ell(N-\frac{\ell+1}{2})}
\varphi\bigl(
\{\maut^{\smash[t]{-\frac{1}{2}}}\az\omega_{\,\underline{i}^c}\za \maut^{\frac{1}{2}}\hat{K}^{\ell+1}\}
\{\maut^{\smash[t]{\frac{1}{2}}}\az(\omega'_{\,\underline{i}^c})^*\za \maut^{-\frac{1}{2}}\hat{K}^{-\ell-1}\}
\bigr) \\
\intertext{(by the modular property of the Haar state, see \eqref{eq:modprop})}
 & \qquad\qquad
=\sum\nolimits_{\underline{i}\in\Lambda_{\ell-k}}
q^{\ell(N-\frac{\ell+1}{2})}
\varphi\bigl(
\{\maut^{\smash[t]{\frac{1}{2}}}\az(\omega'_{\,\underline{i}^c})^*\za \maut^{-\frac{1}{2}}\hat{K}^{-\ell-1}\}
\{\maut^{\smash[t]{\frac{1}{2}}}\az\omega_{\,\underline{i}^c}\za \maut^{\frac{3}{2}}\hat{K}^{\ell+1}\}
\bigr) \\
\intertext{(by invariance of the Haar state)}
 & \qquad\qquad
=\sum\nolimits_{\underline{i}\in\Lambda_{\ell-k}}
q^{\ell(N-\frac{\ell+1}{2})}
\varphi\bigl(
(\omega'_{\,\underline{i}^c})^*
(\omega_{\,\underline{i}^c}\za \maut^2\hat{K}^{2(\ell+1)})
\bigr) \\
\intertext{(using (\ref{eq:zaK}))}
 & \qquad\qquad
=\sum\nolimits_{\underline{i}\in\Lambda_{\ell-k}}
q^{\ell(N-\frac{\ell+1}{2})+6k(\ell+1)-4\ell N-4|\underline{i}^c|}
\varphi\bigl(
(\omega'_{\,\underline{i}^c})^*\omega_{\,\underline{i}^c}
\bigr) \\
 & \qquad\qquad
=\sum\nolimits_{\underline{i}\in\Lambda_k}
q^{\ell(N-\frac{\ell+1}{2})+6k(\ell+1)-4\ell N-4|\underline{i}|}
\varphi\bigl(
(\omega'_{\underline{i}})^*\omega_{\underline{i}}
\bigr) \\
\intertext{(using (\ref{eq:trick}))}
 & \qquad\qquad
=q^{\ell(N-\frac{\ell+1}{2})+4k(\ell+1)-4\ell N}\tfrac{1}
{\dim_q W_k}\inner{\omega',\omega}\sum\nolimits_{\underline{i}\in\Lambda_k}
q^{k(\ell+1)-2|\underline{i}|} \\
 & \qquad\qquad
=q^{\ell(N-\frac{\ell+1}{2})+4k(\ell+1)-4\ell N}\inner{\omega',\omega} \\
 & \qquad\qquad
=q^{\ell(N-\frac{\ell+1}{2})}\inner{\omega'\za\smash[t]{\hat{K}^{2(\ell+1)},\omega\za\hat{K}^{2(\ell+1)}}} \;.
\end{align*}
This concludes the proof.
\end{proof}

As a consequence of Lemma \ref{lemma:iso}, the antiunitary part of $\Ja$ is the map $\Jb:\Omega^k_N\to\Omega^{\ell-k}_{\ell+1-N}$
given by
\begin{equation}\label{eq:Jb}
\Jb\omega=q^{\frac{\ell}{2}(\frac{\ell+1}{2}-N)}\maut^{\frac{1}{2}}\az(\Ja\omega)\za\hat{K}^{2(\ell+1)} \;,
\end{equation}
for all $\omega\in\Omega^k_N$.

\begin{prop}\label{prop:rstru}
The operator $\Jb$ satisfies
\begin{itemize}
\item[(i)] $\Jb^2=(-1)^{\lfloor\frac{\ell+1}{2}\rfloor}$;
\item[(ii)] $\Jb a\Jb^{-1}\omega=\omega\cdot(\maut^{\frac{1}{2}}\az a^*)$ for
   all $\omega\in\Omega^k_N$ and $a\in\Aq{\ell}$;
\item[(iii)] $\Jb\deb\Jb^{-1}|_{\Omega^k_N}=q^{k-N}\deb^\dag$.
\end{itemize}
In particular, we find the following alternative expression for the operator $\deb^\dag$
in \eqref{eq:debdag}:
\begin{equation}\label{eq:samedag}
\deb^\dag\big|_{\Omega^k_N}=-q^{N-k+1}\textstyle{\sum_{i=1}^\ell}\,q^{-2i}
\mL_{S(X_i^*\hat{K})}\otimes \mathfrak{i}^R_{e^i} \;.
\end{equation}
\end{prop}

\smallskip

\begin{proof}
Point (i) follows from Proposition \ref{prop:Jsquare} and the observation that
$\Jb^2=\Ja^2=(*\otimes J)^2$. Concerning point (ii), we have
$$
\Jb a\Jb^{-1}\omega
=\maut^{\frac{1}{2}}\az\bigl(a\,\maut^{\frac{1}{2}}\az\omega^*\za\hat{K}^{2(\ell+1)}\bigr)^*\za\hat{K}^{2(\ell+1)}
=\bigl\{(\maut^{-\frac{1}{2}}\az a)\omega^*\bigr\}^*
=\omega (\maut^{\frac{1}{2}}\az a^*)\;.
$$
It remains point (iii).
By definition $\deb=\sum_{i=1}^\ell\,\mL_{\hat{K}X_i}\otimes\mathfrak{e}^L_{e^i}$,
and
$$
\Jb\deb\Jb^{-1}=
\textstyle{\sum_{i=1}^\ell}\,
\mL_{\maut\hat{K}^{-2(\ell+1)}}
\mL_{S(X_i^*\hat{K})}
\mL_{\maut^{-1}\hat{K}^{2(\ell+1)}}
\otimes J\mathfrak{e}^L_{e^i}J^{-1} \;.
$$
Using Proposition \ref{prop:3.12}, and the identities $\maut h\maut^{-1}=S^2(h)$ and
$\hat{K}X_i^*\hat{K}^{-1}=qX_i^*$, we get
$$
\Jb\deb\Jb^{-1}=
-q^{-2(\ell+1)+1}\textstyle{\sum_{i=1}^\ell}\,
\mL_{S^3(X_i^*\hat{K})}\otimes \mathfrak{i}^R_{e^i} \;.
$$
Notice that $S^2(\hat{K})=\hat{K}$. Moreover, $S^2$ is an
algebra morphism, $S^2(E_j)=q^2E_j$, and $X_i^*$ is the
product of $\ell+1-i$ operators $E_j$'s. Thus
$S^2(X_i^*)=q^{2(\ell+1-i)}X_i^*$ and
$$
\Jb\deb\Jb^{-1}=
-q\textstyle{\sum_{i=1}^\ell}\,q^{-2i}
\mL_{S(X_i^*\hat{K})}\otimes \mathfrak{i}^R_{e^i} \;.
$$
We want to prove that this operator is just $q^{k-N}\deb^\dag$; this will give
as a corollary (\ref{eq:samedag}).
By Lemma \ref{lemma:trick}, it is enough to prove that for all $\omega\in\Omega^{k+1}_N$
we have
$$
(\deb^\dag\omega)_{1,2,\ldots,k}=(q^{N-k}\Jb\deb\Jb^{-1}\omega)_{1,2,\ldots,k} \;.
$$
Using (\ref{eq:internal}) we find the explicit formul{\ae}:
\begin{align*}
(\deb^\dag\omega)_{1,\ldots,k} &=
\sum_{i=k+1}^\ell(-q)^{-k}\omega_{1,2,\ldots,k,i}\za S^{-1}(X_i^*\hat{K}) \;,\\
(\Jb\deb\Jb^{-1}\omega)_{1,\ldots,k} &=
\sum_{i=k+1}^\ell(-q)^{k+1-2i}\omega_{1,2,\ldots,k,i}\za X_i^*\hat{K} \;.
\end{align*}
From the equation $h_{(2)}S^{-1}(h_{(1)})=\epsilon(h)$, applied to $h=X_r^*\hat{K}$,
and using (\ref{eq:Xcop}) for the coproduct, we get:
$$
-X_r^*\hat{K}=\hat{K}N_{r,r-1}\,S^{-1}(X_r^*\hat{K})
+q^{-\frac{1}{2}}(q-q^{-1})\sum_{s=r+1}^\ell
\hat{K}N_{r,s-1}M_{r,s-1}\,
S^{-1}(X_s^*\hat{K}) \;.
$$
Plugging this into previous formula we get
\begin{multline*}
(-\Jb\deb\Jb^{-1}\omega)_{1,2,\ldots,k}=
\sum_{r=k+1}^\ell(-q)^{k+1-2r}\omega_{1,2,\ldots,k,r}\za 
\hat{K}N_{r,r-1}\,S^{-1}(X_r^*\hat{K})
\\
+q^{-\frac{1}{2}}(q-q^{-1})
\sum_{r=k+1}^\ell(-q)^{k+1-2r}
\sum_{s=r+1}^\ell\omega_{1,2,\ldots,k,r}\za 
\hat{K}N_{r,s-1}M_{r,s-1}\,S^{-1}(X_s^*\hat{K}) \;.
\end{multline*}
From the definition of $\Omega^{k+1}_N$ and $\sigma_{k+1}^N$ we deduce
that $\omega\za\hat{K}=q^{k+1-\frac{\ell}{\ell+1}N}\omega$, and
since
\begin{equation}\label{eq:Nrs}
N_{r,s-1}^\ell=(K_r\ldots K_{\ell-1})^\ell(K_s\ldots K_{\ell-1})^\ell
(K_1K_2^2\ldots K_{\ell-1}^{\ell-1})^{-2}\hat{K} \;,
\end{equation}
we have also
$$
\omega_{1,\ldots,k,r}\za\hat{K}N_{r,s-1}
=q^{\frac{1}{2}(1+\delta_{r,s})+k-N}\omega_{1,\ldots,k,r} \;.
$$
Thus,
\begin{multline*}
(-q^{N-k-1}\Jb\deb\Jb^{-1}\omega)_{1,2,\ldots,k}=
\sum_{r=k+1}^\ell(-q)^{k+1-2r}\omega_{1,2,\ldots,k,r}\za S^{-1}(X_r^*\hat{K})
\\
+(1-q^{-2})\sum_{r=k+1}^\ell(-q)^{k+1-2r}
\sum_{s=r+1}^\ell\omega_{1,2,\ldots,k,r}\za 
M_{r,s-1}\,S^{-1}(X_s^*\hat{K}) \;.
\end{multline*}
Now we use (\ref{eq:using}) and get:
\begin{multline*}
(-q^{N-k-1}\Jb\deb\Jb^{-1}\omega)_{1,2,\ldots,k}=
\sum_{r=k+1}^\ell(-q)^{k+1-2r}\omega_{1,2,\ldots,k,r}\za S^{-1}(X_r^*\hat{K})
\\
+(1-q^{-2})\sum_{r=k+1}^\ell(-q)^{k+1-2r}
\sum_{s=r+1}^\ell\omega_{1,2,\ldots,k,s}\za S^{-1}(X_s^*\hat{K}) \;.
\end{multline*}
We use
$\sum_{r=k+1}^\ell\sum_{s=r+1}^\ell=\sum_{s=k+2}^\ell\sum_{r=k+1}^{s-1}$
to change the order of the summations, and use
$$
\sum_{r=k+1}^{s-1}q^{-2r}=\frac{q^{-2(k+1)}-q^{-2s}}{1-q^{-2}}
$$
to get
\begin{align*}
(-q^{N-k-1}\Jb\deb\Jb^{-1}\omega)_{1,2,\ldots,k} &=
\sum_{r=k+1}^\ell(-q)^{k+1-2r}\omega_{1,2,\ldots,k,r}\za S^{-1}(X_r^*\hat{K})
\\ &
+\sum_{s=k+2}^\ell (-q)^{k+1}(q^{-2(k+1)}-q^{-2s})
\omega_{1,2,\ldots,k,s}\za S^{-1}(X_s^*\hat{K}) \\
&=\sum_{s=k+1}^\ell (-q)^{-k-1}
\omega_{1,2,\ldots,k,s}\za S^{-1}(X_s^*\hat{K}) \\
&=-q^{-1}(\deb^\dag\omega)_{1,\ldots,k} \;.
\end{align*}
This concludes the proof.
\end{proof}

\section{A family of spectral triples for $\CP^\ell_q$}\label{sec:quattro}
In this section we present spectral triples for $\CP^\ell_q$.
Let $\HH_N$ be the Hilbert space completion of $\Omega^\bullet_N$,
and let $\gamma_N$ be $+1$ on even forms and $-1$ on odd forms.
A bounded $*$-representation of $\Aq{\ell}$ is given by (the completion
of) left multiplication on $\Omega^\bullet_N$, and a densely defined
$*$-representation of $\Uq{\ell}$ is given by the left action on $\Omega^\bullet_N$.
To construct Dirac operators, there are basically two possibilities. The first is
to use the operator $\deb|_{\Omega^\bullet_N}$ and its Hermitian conjugate.
Since $\Omega^k_N\simeq\Omega^k\otimes_{\Aq{\ell}}\Gamma_N$, another possibility
is to use the natural connection $\nabla_{\!N}\!:\Omega^k_N\to\Omega^{k+1}_N$ defined
by twisting the flat connection on $\Omega^k$ with the Grassmannian connection of
$\Gamma_N$. That is, we define $\nabla_{\!N}$ as a composition:

\smallskip

\begin{center}\begin{tabular}{c}
\begindc{\commdiag}[30]
 \obj(1,3)[A]{$\Omega^k_N$}
 \obj(5,3)[B]{$\qquad\! (\Omega^k_0)^{\otimes r}$}
 \obj(5,1)[C]{$\qquad (\Omega^{k+1}_0)^{\otimes r}$}
 \obj(1,1)[D]{$\Omega^{k+1}_N$}
 \mor{A}{B}[15,15]{{\footnotesize $\,\cdot\,\Psi^\dag_N$}}
 \mor{B}{C}[15,15]{{\footnotesize $\;\deb\otimes 1_r$}}
 \mor{C}{D}[20,20]{{\footnotesize $\,\cdot\,\Psi_N$}}[\atright,\solidarrow]
 \mor{A}{D}[15,15]{{\footnotesize $\,\nabla_{\!N}$}}[\atright,\dasharrow]
\enddc
\end{tabular}\end{center}

\smallskip

\noindent
where $\Psi_N$ is the vector given in (\ref{eq:Psi}), $r:=\binom{|N|+\ell}{\ell}$ is its size,
and $1_r$ is the identity matrix of size $r$. Explicitly
$$
\nabla_{\!N}\omega:=\big\{\deb(\omega\Psi^\dag_N)\big\}\cdot\Psi_N \;,
$$
for all $\omega\in\Omega^k_N$. The two choices coincide, that is $\nabla_{\!N}=\deb|_{\Omega^\bullet_N}$
as shown in the next lemma.

\begin{lemma}
We have
$$
\nabla_{\!N}\omega=\deb\omega \;,
$$
for all $\omega\in\Omega^k_N$.
\end{lemma}

\begin{proof}
We give the proof for $N\geq 0$.
By definition, the components of $\Psi^\dag_N$ are monomials in $z_i$.
As $\pi^{\ell+1}_i(F_j)=0$, $z_i\za F_j=0$ and so $\Psi^\dag_N\za F_j=0$.
This means that $\Psi^\dag_N\za S^{-1}(X_i)=0$. Then by (\ref{eq:Xcop})
$$
\{\deb(\omega\Psi^\dag_N)\}_{\underline{i}}=\sum\nolimits_{r=1}^{k+1} (-q)^{1-r}
\bigr(\omega_{\underline{i}\smallsetminus i_r}\za S^{-1}(\hat{K}X_{i_r})\bigr)
\Psi^\dag_N=(\deb\omega)_{\underline{i}}\Psi^\dag_N
$$
and $\nabla_{\!N}\omega=\deb\omega$.
\end{proof}

Thus, for $c_{\ell,N,k}\in\C$ we define the following Dirac-type operator
$$
D(c)|_{\Omega^k_N}:=c_{\ell,N,k}\deb+\bar{c}_{\ell,N,k-1}\deb^\dag \;.
$$
The antilinear map $\Jb$ given by \eqref{eq:Jb} extends to an isometry $\HH_N\to\HH_{\ell+1-N}$. In particular,
for odd $\ell$ it is an automorphism of $\HH_{\frac{1}{2}(\ell+1)}$. A preferred Dirac operator $D_N$ on $\HH_N$
-- belonging to the class $D(c)$ -- is given by
\begin{equation}\label{eq:DN}
D_N:=q^{\frac{1}{2}(k-N)}\deb+q^{\frac{1}{2}(k-N-1)}\deb^\dag \;.
\end{equation}
By Proposition \ref{prop:rstru}, we have $\Jb D_N=D_{\ell+1-N}\Jb$ on $\Omega^\bullet_N$, that for odd $\ell$
and $N=\frac{1}{2}(\ell+1)$ means that $D_N$ commutes with $\Jb$.

Remark: the operator $D_0$ is a $q$-analogue of the Dolbeault-Dirac operator of $\CP^\ell$,
while $D_N$ is the twist of $D_0$ with the Grassmannian connection of the line
bundle $\Gamma_N$. If $\ell$ is odd and $N=\frac{1}{2}(\ell+1)$, $D_N$ is a
$q$-analogue of the Dirac operator of the Fubini-Study metric.

The main result of this section is the following theorem.

\begin{thm}\label{thm}
The datum $(\Aq{\ell},\HH_N,D_N,\gamma_N)$ is a $0^+$-dimensional equivariant
even spectral triple. If $\ell$ is odd and $N=\frac{1}{2}(\ell+1)$, the spectral
triple $(\Aq{\ell},\HH_N,D_N,\gamma_N,\Jb)$ is real with KO-dimension
$\,2\ell\!\!\mod 8$.
\end{thm}

The rest of this section is devoted to the proof of this theorem.
Let us check the conditions of a spectral triple, as recalled in Appendix \ref{app:A}. 
Equivariance holds by construction: $\Omega^\bullet_N$ is
dense in $\HH_N$, it is a left $\Aq{\ell}\rtimes\Uq{\ell+1}$-module, the Dirac-type operators
(symmetric by construction since $\deb^\dag$ is the Hermitian conjugated of $\deb$) are defined
on the dense domain $\Omega^\bullet_N$ and commute with the action of $\Uq{\ell+1}$ (since $\deb$ does),
and $\Jb$ is the antiunitary part of $\Ja$ -- given by \eqref{eq:Rstructure} -- that is equivariant
due to \eqref{eq:Jaequiv}.

Next, for any $a\in\Aq{\ell}$,
$$
[D(c),a]=c_{\ell,N,k}[\deb,a]-\bar{c}_{\ell,N,k-1}\,[\deb,a^*]^\dag \;.
$$
But by \eqref{eq:gLeibzeroN}, $[\deb,a]$ is the operator of left multiplication by $\deb a$,
and this shows the bounded commutator condition.

As far as the grading is concerned, the action of $\Aq{\ell}\rtimes\Uq{\ell+1}$ does not change the parity
of forms, while the Dirac operator does.
The operator $\Jb$ sends $\Omega^k_N$ to $\Omega^{\ell-k}_{\ell+1-N}$, that is it exchanges the
parity of forms exactly when $\ell$ is odd, that shows the last condition
in (\ref{eq:J}). The remaining two conditions in (\ref{eq:J}), $D_N\Jb=\Jb D_N$ and $\Jb^2=(-1)^{\frac{1}{2}(\ell+1)}$,
follow from Proposition \ref{prop:rstru} for any odd $\ell$ and $N=\frac{1}{2}(\ell+1)$. 

We now pass to the commutant and first order condition, cf.~\eqref{eq:real}.
Proposition \ref{prop:rstru} tells us that for odd $\ell$ and $N=\frac{1}{2}(\ell+1)$, $\Jb a\Jb^{-1}$ is the
operator of right multiplication by $\,\maut^{\frac{1}{2}}\az a^*$. Since left and right $\Aq{\ell}$-module
structure of $\Omega^\bullet_N$ commute, this proves the commutant condition.
Moreover, due to the modular property of the Haar state (\eqref{eq:modprop}),
\begin{align*}
\big<\omega,\omega'(\maut^{\frac{1}{2}}\az a^*)\big>
&=\varphi\bigl( \omega^\dag\omega'(\maut^{\frac{1}{2}}\az a^*) \bigr)
=\varphi\bigl( (\maut^{-\frac{1}{2}}\az a^*)\omega^\dag\omega' \bigr) \\
&=\varphi\bigl( \{\omega(\maut^{\frac{1}{2}}\az a)\}^\dag\omega' \bigr)
=\big<\omega(\maut^{\frac{1}{2}}\az a),\omega'\big>
=\inner{\Jb a^*\Jb^{-1}\omega,\omega'} \;.
\end{align*}
Hence $(\Jb a\Jb^{-1})^\dag=\Jb a^*\Jb^{-1}$. Since $[\deb,a]=(\deb a)\wprod$, from the associativity of
$\wprod$ we deduce the first order condition for $\deb$. This condition for $D_N$ follows
from the identity
$$
[[D_N,a],\Jb b\Jb^{-1}]=c_{\ell,N,k}[[\deb,a],\Jb b\Jb^{-1}]
+\bar{c}_{\ell,N,k-1}\,[[\deb,a^*],\Jb b^*\Jb^{-1}]^\dag \;,
$$
that is valid for all $a,b\in\Aq{\ell}$.

Next, we show that $D_N$ is diagonalizable by relating it to the Casimir of $\Uq{\ell+1}$ (then,
being a symmetric operator, $D_N$ has a canonical selfadjoint extension). We shall also prove that
the eigenvalues of $D_N$ diverge exponentially (while their multiplicities are only polynomially
divergent). This implies that $(D_N^2+1)^{-1}$ is compact and that $D_N$ is $0^+$-summable, which
will complete the proof of Theorem \ref{thm}. We need first two technical lemmas.

\begin{lemma}\label{lemma:dddag}
We have
\begin{equation}\label{eq:dddag}
(\deb\deb^\dag+\deb^\dag\deb)\omega=\omega\za\left(
q^{\frac{2N}{\ell+1}}\textstyle{\sum_{i=1}^\ell}q^{-2i}X_iX^*_i
+q^{N-\ell-k}[k][\ell+1-N]\right)
\end{equation}
for all $\omega\in\Omega^k_N$.
\end{lemma}

\begin{proof}
Using (\ref{eq:Dolbeault}) for $\deb$ and (\ref{eq:samedag}) for $\deb^\dag$, we compute
$$
(\deb\deb^\dag+\deb^\dag\deb)\big|_{\Omega^k_N}=-q^{N-k}\sum_{i,j=1}^\ell q^{-2j}\Big(
q\,\mL_{\hat{K}X_i}\mL_{S(X_j^*\hat{K})}\otimes \mathfrak{e}^L_{e^i}\mathfrak{i}^R_{e^j}
+\mL_{S(X_j^*\hat{K})}\mL_{\hat{K}X_i}\otimes\mathfrak{i}^R_{e^j}\mathfrak{e}^L_{e^i}\Big) \;.
$$
Since $X_i\hat{K}=q\hat{K}X_i$, we have $X_i^*\hat{K}=q^{-1}\hat{K}X_i^*$ and
\begin{align*}
\mL_{\hat{K}X_i}\mL_{S(X_j^*\hat{K})} &=\mL_{\hat{K}X_i S(X_j^*\hat{K})}
     =\mL_{\hat{K}X_i\hat{K}^{-1} S(X_j^*)}=q^{-1}\mL_{X_i S(X_j^*)} \;,\\
\mL_{S(X_j^*\hat{K})}\mL_{\hat{K}X_i} &=\mL_{S(X_j^*\hat{K})\hat{K}X_i}
     =q^{-1}\mL_{S(\hat{K}X_j^*)\hat{K}X_i}=q^{-1}\mL_{S(X_j^*)X_i} \;.
\end{align*}
Thus
$$
(\deb\deb^\dag+\deb^\dag\deb)\big|_{\Omega^k_N}=-q^{N-k}\sum_{i,j=1}^\ell q^{-2j}
\Big(
\mL_{X_i S(X_j^*)}\otimes\mathfrak{e}^L_{e^i}\mathfrak{i}^R_{e^j}+
q^{-1}
\mL_{S(X_j^*)X_i}\otimes\mathfrak{i}^R_{e^j}\mathfrak{e}^L_{e^i}
\Big) \;.
$$
Using (\ref{eq:internal}) we get
\begin{align*}
\sum_{i,j=1}^\ell q^{-2j}(\mL_{X_i S(X_j^*)}\otimes\mathfrak{e}^L_{e^i}\mathfrak{i}^R_{e^j}\omega)_{\underline{i}}
=\!\!\sum_{r\in\underline{i} \, ,\, s\notin(\underline{i}\smallsetminus r)}
\!\!(-q)^{-L(r,\underline{i})+L(s,(\underline{i}\smallsetminus r)\cup s)-2s}
\,\omega_{(\underline{i}\smallsetminus r)\cup s}\za q^{-2s}X^*_s S^{-1}(X_r) \\
=\sum_{r\in\underline{i}}\,\omega_{\underline{i}}\za q^{-2r}X^*_r S^{-1}(X_r)
+\!\!\sum_{r\in\underline{i} \, ,\, s\notin\underline{i} \, ,\, r<s}
\!\!(-q)^{-L(r,\underline{i})+L(s,\underline{i}\cup s)-1}
\,\omega_{\underline{i}\smallsetminus r\cup s}\za q^{-2s}X^*_s S^{-1}(X_r) \\
+\!\!\sum_{r\in\underline{i} \, ,\, s\notin\underline{i} \, ,\, r>s}
\!\!(-q)^{-L(r,\underline{i})+L(s,\underline{i}\cup s)}
\,\omega_{\underline{i}\smallsetminus r\cup s}\za q^{-2s}X^*_s S^{-1}(X_r) \\
\intertext{and}
\sum_{i,j=1}^\ell q^{-2j}(\mL_{S(X_j^*)X_i}\otimes\mathfrak{i}^R_{e^j}\mathfrak{e}^L_{e^i}\omega)_{\underline{i}}
=\!\!\sum_{s\notin\underline{i} \, ,\, r\in(\underline{i}\cup s)}
\!\!(-q)^{L(s,\underline{i}\cup s)-L(r,\underline{i}\cup s)}
\,\omega_{(\underline{i}\cup s)\smallsetminus r}\za q^{-2s}S^{-1}(X_r)X^*_s \qquad\quad\! \\
=\sum_{s\notin\underline{i}}\,\omega_{\underline{i}}\za q^{-2s}S^{-1}(X_s)X^*_s
+\!\!\sum_{s\notin\underline{i} \, ,\, r\in\underline{i} \, , \, r<s}
\!\!(-q)^{-L(r,\underline{i})+L(s,\underline{i}\cup s)}
\,\omega_{\underline{i}\cup s\smallsetminus r}\za q^{-2s}S^{-1}(X_r)X^*_s \\
+\!\!\sum_{s\notin\underline{i} \, ,\, r\in\underline{i} \, , \, r>s}
\!\!(-q)^{-L(r,\underline{i})+L(s,\underline{i}\cup s)-1}
\,\omega_{\underline{i}\cup s\smallsetminus r}\za q^{-2s}S^{-1}(X_r)X^*_s \;.
\end{align*}
In particular, we read off the component $\underline{i}=(1,2,\ldots,k)$,
\begin{align*}
-q^{k-N}\{(\deb\deb^\dag+\deb^\dag\deb)\omega\}_{1,\ldots,k}
& =\sum_{r=1}^k\omega_{1,\ldots,k}\za q^{-2r}X^*_rS^{-1}(X_r)
+q^{-1}\!\!\!\sum_{s=k+1}^\ell\!\!\omega_{1,\ldots,k}\za q^{-2s}S^{-1}(X_s)X^*_s
 \\
& \quad +\sum_{r=1}^k\sum_{s=k+1}^\ell(-q)^{k-r-2s}
\,\omega_{1,\ldots,\hat{r},\ldots,k-1,k,s}\za [X^*_s,S^{-1}(X_r)] \;.
\end{align*}
We use (\ref{eq:SXXA}) and Lemma \ref{lemma:Dsquare} to rewrite last term:
\begin{align*}
\sum_{r=1}^k\sum_{s=k+1}^\ell(-q)^{k-r-2s}
\,\omega_{1,\ldots,\hat{r},\ldots,k-1,k,s}\za [X^*_s,S^{-1}(X_r)]
=-q^{\frac{3}{2}}\sum_{r=1}^k\sum_{s=k+1}^\ell
\,q^{2(k-r-s)}\omega_{1,\ldots,k}\za (\hat{K}N_{r,s-1})^{-1}
\\
=-q^{N+1-k}\omega_{1,\ldots,k}\sum_{r=1}^k\sum_{s=k+1}^\ell\,q^{2(k-r-s)}
=-q^{N-\ell-k-1}[k][\ell-k]\omega_{1,\ldots,k} \;.
\end{align*}
Thus,
\begin{align*}
-q^{k-N}\{(\deb\deb^\dag+\deb^\dag\deb)\omega\}_{1,2,\ldots,k}
&=\sum_{r=1}^k\omega_{1,2,\ldots,k}\za q^{-2r}X^*_rS^{-1}(X_r)
+q^{-1}\!\!\!\sum_{s=k+1}^\ell\!\!\omega_{1,2,\ldots,k}\za q^{-2s}S^{-1}(X_s)X^*_s \\
&\quad -q^{N-\ell-k-1}[k][\ell-k]\omega_{1,2,\ldots,k} \;.
\end{align*}
Now, we rewrite the first summation using (\ref{eq:SXXB}) and get:
\begin{align*}
-q^{k-N}\{(\deb\deb^\dag+\deb^\dag\deb)\omega\}_{1,2,\ldots,k}
=\sum_{r=1}^k\omega_{1,2,\ldots,k}\za q^{-2r}S^{-1}(X_r)X^*_r
+q^{-1}\!\!\!\sum_{s=k+1}^\ell\!\!\omega_{1,2,\ldots,k}\za q^{-2s}S^{-1}(X_s)X^*_s
\\
-q\sum_{r=1}^kq^{-2r}\omega_{1,2,\ldots,k}\za \tfrac{(K_r\ldots K_\ell)^2-(K_r\ldots K_\ell)^{-2}}{q-q^{-1}}
-q^{N-\ell-k-1}[k][\ell-k]\omega_{1,\ldots,k} \;.
\end{align*}
The last sum is
$$
\sum_{r=1}^kq^{-2r}\omega_{1,\ldots,k}\za \tfrac{(K_r\ldots K_\ell)^2-(K_r\ldots K_\ell)^{-2}}{q-q^{-1}}
=[k-N+1]\omega_{1,\ldots,k}\sum_{r=1}^kq^{-2r}=
q^{-k-1}[k][k-N+1]\omega_{1,\ldots,k} \;.
$$
Thus
\begin{align*}
-q^{k-N}\{(\deb\deb^\dag+\deb^\dag\deb)\omega\}_{1,\ldots,k}
&=\sum_{r=1}^k\omega_{1,\ldots,k}\za q^{-2r}S^{-1}(X_r)X^*_r
+q^{-1}\!\!\sum_{r=k+1}^\ell\!\omega_{1,\ldots,k}\za q^{-2r}S^{-1}(X_r)X^*_r \\
&\quad -q^{-\ell}[k][\ell+1-N]\omega_{1,\ldots,k} \;.
\end{align*}
Applying the identity $S^{-1}(h_{(2)})h_{(1)}=\epsilon(h)$ to $h=X_r$,
and using (\ref{eq:Xcop}) for $\Delta(X_r)$, we get
$$
-S^{-1}(X_r)=qN_{r,r-1}^{-1}\hat{K}X_r
+q^{\frac{1}{2}}(q-q^{-1})\sum_{s=r+1}^\ell N_{r,s-1}^{-1}\hat{K}S^{-1}(M_{r,s-1}^*)X_s  \;.
$$
Since $\omega_{1,2,\ldots,k}\za F_n=0$ for all $1\leq n\leq\ell-1$, 
we have $\omega_{1,2,\ldots,k}\za S^{-1}(M_{r,s-1}^*)$ for all $s\leq\ell$. Thus
$$
\omega_{1,\ldots,k}\za S^{-1}(X_r)=
-q\omega_{1,\ldots,k}\za N_{r,r-1}^{-1}\hat{K}X_r=
\begin{cases}
-q^{k-N+\frac{2N}{\ell+1}}\omega_{1,\ldots,k}\za X_r &\mathrm{if}\;1\leq r\leq k \;,\\
-q^{k-N+\frac{2N}{\ell+1}+1}\omega_{1,\ldots,k}\za X_r &\mathrm{if}\;k+1\leq r\leq\ell \;.
\end{cases}
$$
With this, we obtain
$$
\{(\deb\deb^\dag+\deb^\dag\deb)\omega\}_{1,2,\ldots,k}
=q^{\frac{2N}{\ell+1}}\sum_{r=1}^\ell\omega_{1,2,\ldots,k}\za q^{-2r}X_rX^*_r
+q^{N-\ell-k}[k][\ell+1-N]\omega_{1,\ldots,k} \;.
$$
By Lemma \ref{lemma:trick}, the last equality implies
$$
\{(\deb\deb^\dag+\deb^\dag\deb)\omega\}_{\underline{i}}
=q^{\frac{2N}{\ell+1}}\sum_{r=1}^\ell\omega_{\underline{i}}\za q^{-2r}X_rX^*_r
+q^{N-\ell-k}[k][\ell+1-N]\omega_{\underline{i}} \;,
$$
and this concludes the proof.
\end{proof}

\begin{lemma}
We have
\begin{equation}\label{eq:dddagB}
q^{2(k-N)}(\deb\deb^\dag+q^2\deb^\dag\deb)\omega=\omega\za\left(
q^{2\frac{\ell+1-N}{\ell+1}}\textstyle{\sum_{i=1}^\ell}q^{-2i}S^{-1}(X_iX^*_i)
+q^{-N-\ell+k+1}[\ell-k][N]\right)
\end{equation}
for all $\omega\in\Omega^k_N$.
\end{lemma}

\begin{proof}
We substitute $k\to\ell-k$ and $N\to\ell+1-N$ in \eqref{eq:dddag} and apply conjugation by $\Jb$.
By Proposition \ref{prop:rstru} the left hand side becomes the left hand side of \eqref{eq:dddagB}.
Since for any $h\in\Uq{\ell+1}$,
$(\Jb^{-1}\omega)\za h=\Jb^{-1}(\omega\za \hat{K}^{-2(\ell+1)}S^{-1}(h^*)\hat{K}^{2(\ell+1)})^*$,
and since $X_iX_i^*$ commutes with $\hat{K}$, the right hand side of \eqref{eq:dddag}
becomes the right hand side of \eqref{eq:dddagB}.
\end{proof}

Using the identity
$$
D_N^2\big|_{\Omega^k_N} =q^{k-N-1}(\de\deb^\dag+q\deb^\dag\de)
=q^{k-N-1}(q\tfrac{\de\deb^\dag+\deb^\dag\de}{1+q}+\tfrac{\de\deb^\dag+q^2\deb^\dag\de}{1+q}) \;,
$$
from \eqref{eq:dddag} and \eqref{eq:dddagB} we get the following expression for the square of
the Dirac operator
\begin{align}
(1+q)D_N^2\big|_{\Omega^k_N} &=
q^{k-\frac{\ell-1}{\ell+1}N}\textstyle{\sum_{i=1}^\ell}q^{-2i}(\,\cdot\,)\za X_iX^*_i
+q^{\frac{\ell-1}{\ell+1}N-k+1}\textstyle{\sum_{i=1}^\ell}q^{-2i}(\,\cdot\,)\za S^{-1}(X_iX^*_i) \notag\\ &\quad
+q^{-\ell}[k][\ell+1-N]+q^{-\ell}[\ell-k][N] \;. \label{eq:DNsquare}
\end{align}
Since $D_N^2$ leaves the degree of forms invariant, it is the direct sum of
finitely many $D_N^2|_{\Omega^k_N}$. Therefore, it suffices to prove that
the eigenvalues of $D_N^2|_{\Omega^k_N}$ diverge exponentially and the multiplicities
diverge polynomially. For that, we now study the first term in \eqref{eq:DNsquare}.

\begin{lemma}\label{lemma:6.5}
We have
\begin{align*}
\za\big(\mathcal{C}_q-\textstyle{\sum\nolimits_{i=1}^\ell}q^{\ell+1-2i} X_iX_i^*\big)\big|_{\Omega^k_N}
&=q^{1+\frac{2k}{\ell}-\frac{2N}{\ell+1}}\left(
\tfrac{1}{2}[k][\ell+1-\tfrac{\ell+2}{\ell}k]+\tfrac{1}{2}[\ell][\tfrac{k}{\ell}]^2
+\tfrac{[\ell]}{(q-q^{-1})^2}\right) \\
&\quad +\tfrac{1}{(q-q^{-1})^2}\,q^{-2k+\ell(\frac{2N}{\ell+1}-1)}
-\tfrac{[\ell+1]}{(q-q^{-1})^2} \;.
\end{align*}
\end{lemma}

\begin{proof}
From (\ref{eq:CCprime}) we get
$$
\mathcal{C}_q-\textstyle{\sum\nolimits_{i=1}^\ell}q^{\ell+1-2i}X_iX_i^*=
q\hat{K}^{\frac{2}{\ell}}\Big(\mathcal{C}'_q+\tfrac{[\ell]}{(q-q^{-1})^2}\Big)
+\tfrac{q^{-\ell}}{(q-q^{-1})^2}\,\hat{K}^{-2}
-\tfrac{[\ell+1]}{(q-q^{-1})^2} \;,
$$
where $\mathcal{C}_q$ is the Casimir of $\Uq{\ell+1}$ as in (\ref{eq:Cq}) and
$\mathcal{C}'_q$ is the Casimir of $\Uq{\ell}$ as in (\ref{eq:CqK}).
The representation $\sigma^N_k$ in the definition of $\Omega^k_N$
has highest weight $n_i=\delta_{i,k}$.
By Lemma \ref{lemma:multeig} (with the replacement $\ell\to\ell-1$ and $n_1=n_\ell=0$)
the eigenvalue of $\mathcal{C}'_q$ in this representation is
\begin{equation}\label{eq:zaCprimeq}
\za 2\mathcal{C}'_q\big|_{\Omega^k_N}=
[k][\ell+1-\tfrac{\ell+2}{\ell}k]+[\ell][\tfrac{k}{\ell}]^2 \;.
\end{equation}
Moreover $\za\,\hat{K}|_{\Omega^k_N}=q^{k-\frac{\ell}{\ell+1}N}$. Combining
these equations concludes the proof.
\end{proof}

We are interested in the asymptotic behaviour of the eigenvalues of $D_N$.
From \eqref{eq:DNsquare} and the last lemma we deduce that
$$
(1+q)D_N^2\big|_{\Omega^k_N}\simeq q^{k-\frac{\ell-1}{\ell+1}N-\ell-1}\mathcal{C}_q
+q^{\frac{\ell-1}{\ell+1}N-k-\ell}S^{-1}(\mathcal{C}_q) \;,
$$
where `$\simeq$' means equality modulo constant multiples of the identity operator
on $\Omega^k_N$.
Since the coefficients in front of $\mathcal{C}_q$ and $S^{-1}(\mathcal{C}_q)$ in
$(1+q)D_N^2|_{\Omega^k_N}$ are positive (so that it is not possible that
the dominating contributions cancel), it is enough to show that
$\mathcal{C}_q|_{\Omega^k_N}$ and $S^{-1}(\mathcal{C}_q)|_{\Omega^k_N}$ separately
have eigenvalues that diverge exponentially.

Recall that for central elements, like $\mathcal{C}_q$ and $S^{-1}(\mathcal{C}_q)$, the
left and right canonical actions coincide. The decomposition of $\Omega^k_N$ with respect of the left
action of $\Uq{\ell}$ was given in Proposition \ref{prop:ha}. We set $n_\ell=n$ and $n_1=n+N-k$, where $n\gg 1$.
We use \eqref{eq:eigA} and, for $1\leq k\leq\ell$, we find two series of eigenvalues
for $\mathcal{C}_q|_{\Omega^k_N}$ given by
$$
q^{-2n}\frac{q^{-\frac{2\ell}{\ell+1}N-\ell+2k-2}+q^{\frac{2}{\ell+1}N-\ell}}{2(q-q^{-1})^2}+O(1) \;,\qquad
q^{-2n}\frac{q^{-\frac{2\ell}{\ell+1}N-\ell+2k}+q^{\frac{2}{\ell+1}N-\ell}}{2(q-q^{-1})^2}+O(1) \;,
$$
where by $O(1)$ we mean terms that remain bounded when $n\to\infty$. These eigenvalues diverge like $q^{-2n}$
and have multiplicity that, according to \eqref{eq:mult}, is polynomial in $n$. If $k=0$ or $k=\ell$,
the same holds true except one of the two series is absent.

A similar reasoning applies to $S^{-1}(\mathcal{C}_q)|_{\Omega^k_N}$. The eigenvalue of $S^{-1}(\mathcal{C}_q)$
in the representation with highest weight $(n_1,n_2,\ldots,n_\ell)$ is the eigenvalue of $\mathcal{C}_q$ in
the representation with highest weight $(n_\ell,\ldots,n_2,n_1)$. Thus the two series of eigenvalues for
$S^{-1}(\mathcal{C}_q)|_{\Omega^k_N}$ are given asymptotically by
$$
q^{-2n}\frac{q^{\frac{2\ell}{\ell+1}N-\ell-2k}+q^{-\frac{2}{\ell+1}N-\ell+2}}{2(q-q^{-1})^2}+O(1) \;,\qquad
q^{-2n}\frac{q^{\frac{2\ell}{\ell+1}N-\ell-2k+2}+q^{-\frac{2}{\ell+1}N-\ell+2}}{2(q-q^{-1})^2}+O(1) \;,
$$
and one of the two is absent if $k=0$ or $k=\ell$. With this, the proof of Theorem \ref{thm}
is completed.

Our asymptotic analysis can be extended to the explicit computation of the spectrum of $D_N$,
but we are not giving it since the formul{\ae} are complicated and not particularly illuminating.

We conclude with few remarks on the case $\ell=2$.
The combination of Lemma \ref{lemma:dddag} and Lemma \ref{lemma:6.5} gives
(after a tedious check in all the three cases $k=0,1,2$) the simple formula
$$
(\deb+\deb^\dag)^2|_{\Omega^\bullet_N}=q^{\frac{2}{3}N-3}\mathcal{C}_q
+\tfrac{1}{1-q^2}\big(q^N[N]-q^{\frac{N}{3}+2}[3][\tfrac{N}{3}]\big) \;.
$$
From this, the explicit expression of the spectrum of the Dirac operator $\deb+\deb^\dag$
follows, which extends the case $N=0$ studied in \cite{DDL08b}. For $N=0$ we get
$$
(\deb+\deb^\dag)^2|_{\Omega^\bullet_0}=q^{-3}\mathcal{C}_q \;,
$$
which is the square of the Dirac operator considered in \cite{DDL08b}, although in
different conventions which explains the $q^{-3}$ factor.

\smallskip

{\small%
\subsection*{Acknowledgments}
L.D.~was partially supported by PRIN-2006 (I), MKTD-CT-2004-509794 (EU) and N201177033 (PL). }

\appendix\section{Some general definitions}\label{app:A}
We recall the notion of equivariant (unital) spectral triple, see e.g.~\cite{Con94,Sit03}.
Let $\A$ be a (complex, associative) involutive algebra with unity,
$(\U,\epsilon,\Delta,S)$ a Hopf $*$-algebra and suppose $\A$ is a
left $\U$-module $*$-algebra, which means that the left action
`$\az$' of $\U$ on $\A$ satisfies
$$
h\az ab=(h_{(1)}\az a)(h_{(2)}\az b)\;,\qquad
h\az 1=\varepsilon(h)1\;,\qquad
h\az a^*=\{S(h)^*\az a\}^*\;,
$$
for all $h\in\U$ and $a,b\in\A$. As usual we use Sweedler notation for the
coproduct, $\Delta(h)=h_{(1)}\otimes h_{(2)}$ with summation understood.
The left crossed product $\A\rtimes\U$ is the $*$-algebra generated
by $\A$ and $\U$ with crossed commutation relations
$$
ha=(h_{(1)}\az a)h_{(2)}\;,\quad\forall\;h\in\U,\;a\in\A\,.
$$
The data $(\A,\HH,D)$ is called an $\U$-equivariant spectral triple
if (i) there is a dense subspace $\M$ of $\HH$ carrying a $*$-representation
$\pi$ of $\A\rtimes\U$, (ii) $D$ is a (unbounded) selfadjoint operator
with compact resolvent and with domain containing $\M$, (iii) $\pi(a)$
and $[D,\pi(a)]$ extend to bounded operators on $\HH$ for all $a\in\A$,
(iv) $[D,\pi(h)]=0$ on $\M$ for any $h\in\U$.
As usual we refer to $D$ as the `Dirac operator', in analogy with the
commutative situation where spectral triples are canonically associated to
spin structures. The representation symbol $\pi$ will be omitted.

An equivariant spectral triple is called \emph{even} if there exists a grading
$\gamma$ on $\HH$ (i.e.~a bounded operator satisfying $\gamma=\gamma^*$ and $\gamma^2=1$)
such that the Dirac operator is odd and the crossed product algebra is even:
$$
\gamma D+D\gamma=0\;,\quad\qquad t\gamma=\gamma t \quad \forall\;t\in\A\rtimes\U\;.
$$
An even spectral triple is called \emph{real} if there exists an antilinear
isometry $J$ on $\HH$ such that
$$
J^2=\pm 1\;,\qquad
JD=\pm DJ\;,\qquad
J\gamma=\pm\gamma J\;,
$$
and such that for all $a,b\in\A$
\begin{equation}\label{eq:real}
[a,JbJ^{-1}]=0\;,\qquad
[[D,a],JbJ^{-1}]=0\;.
\end{equation}
The signs `$\pm$' in (\ref{eq:J}) are determined by the dimension of
the geometry~\cite{Con96}. For a $2\ell$-dimensional space, with $\ell$ odd,
we have
\begin{equation}\label{eq:J}
J^2=(-1)^{\frac{1}{2}(\ell+1)} \;,\qquad
JD=DJ \;,\qquad
J\gamma=-\gamma J \;.
\end{equation}

Finally, we call $(\A,\HH,D,\gamma,J)$ a `real equivariant even spectral triple'
if $(\A,\HH,D,\gamma)$ is an equivariant even spectral triple, $J$ is a real
structure and there is an (unbounded) antilinear operator $T$ on $\HH$ whose
antiunitary part is $J$ and satisfying
$$
Th=S(h)^*T\;,
$$
on the joint domain of $h$ and $T$, and for all $h\in\U$.

\section{Twisted Dirac operators on $\CP^\ell$}\label{app:B}

In this section we consider the limit $q=1$.
If we work with formal power series in $\hbar:=\log q$, the Cartan generators $H_i$ are obtained from $K_i$ through the rescaling $K_i=q^{H_i/2}$. In the $q\to 1$ limit, the elements $\{H_i,E_i,F_i\}_{i=1,\ldots,\ell}$ satisfy the well known Chevalley-Serre relations
(see~\cite{Ser01}, Sec.~VI.4), and the elements $\{H_i,M_{jk},M_{jk}^*\}_{i,j,k=1,\ldots,\ell}$
form a linear basis of the Lie algebra $\mathfrak{su}(\ell+1)$ called Cartan-Weyl basis (see~\cite{KS97}, Sec.~6.1.1).

The $q\to 1$ limit of the Casimir (\ref{eq:Cq}) is
$$
\mathcal{C}_{q=1}=\tfrac{1}{2}A+\sum_{1\leq i\leq j\leq\ell}M_{ij}^*M_{ij}
$$
where
\begin{align*}
A &=\sum\nolimits_{i=1}^{\ell+1}
\left(\tfrac{\sum_{j=1}^{i-1}jH_j-\sum_{j=i}^\ell
(\ell+1-j)H_j}{\ell+1}+i-\tfrac{\ell+2}{2}\right)^2
-\sum\nolimits_{i=1}^{\ell+1}
\left(i-\tfrac{\ell+2}{2}\right)^2 \\
&=\sum\nolimits_{i=1}^{\ell+1}
\left(\tfrac{\sum_{j=1}^{i-1}jH_j-\sum_{j=i}^\ell
(\ell+1-j)H_j}{\ell+1}\right)^2
+\sum\nolimits_{i=1}^{\ell+1}
2(i-\tfrac{\ell+2}{2})\tfrac{\sum_{j=1}^{i-1}jH_j-\sum_{j=i}^\ell
(\ell+1-j)H_j}{\ell+1} \;.
\end{align*}
The term $-2\tfrac{\ell+2}{2}\sum\nolimits_{i=1}^{\ell+1}
\tfrac{\sum_{j=1}^{i-1}jH_j-\sum_{j=i}^\ell
(\ell+1-j)H_j}{\ell+1}$ is zero, as one can see inverting
the order in the summation. Furthermore
\begin{align*}
\frac{1}{\ell+1}\sum_{i=1}^{\ell+1}2i\left(\sum_{j=1}^{i-1}jH_j
-\sum_{j=i}^\ell(\ell+1-j)H_j\right) &=
\frac{2}{\ell+1}\sum\nolimits_{j=1}^\ell\left(
\sum_{i=j+1}^{\ell+1}ij-\sum_{i=1}^ji(\ell+1-j)\right)H_j \\
&=\sum\nolimits_{j=1}^\ell j(\ell+1-j)H_j\;.
\end{align*}
Now we develop the square in $A$ using the formula
$$
\left(\sum\nolimits_{i=1}^n a_i\right)^2=
\sum\nolimits_{i=1}^n a_i^2+
2\sum\nolimits_{1\leq i<j\leq n}a_ia_j
$$
and get
$$
A=\sum_{i=1}^\ell\tfrac{i(\ell+1-i)}{\ell+1}H_i(\ell+1-H_i)
+2\sum_{1\leq i\leq j\leq\ell}\tfrac{i(\ell+1-j)}{\ell+1}\,H_iH_j \;.
$$
Hence
\begin{equation}\label{eq:B1}
\mathcal{C}_{q=1}=\tfrac{1}{2}\sum_{i=1}^\ell\tfrac{i(\ell+1-i)}{\ell+1}H_i(\ell+1-H_i)
 +\sum_{1\leq i\leq j\leq\ell}(\tfrac{i(\ell+1-j)}{\ell+1}H_iH_j+M_{ij}^*M_{ij}) \;.
\end{equation}
The relation between the Casimir of $U(\mathfrak{su}(\ell+1))$ and the
Casimir of $U(\mathfrak{su}(\ell))$ is:
$$
\textstyle{\sum_{i=1}^\ell}X_iX_i^*=\mathcal{C}_{q=1}-\mathcal{C}'_{q=1}-\tfrac{\ell+1}{8\ell}\hat{H}(\hat{H}+\ell)
$$
where $\hat{H}=\frac{2}{\ell+1}\sum_{i=1}^\ell iH_i$. This can be obtained both as $q\to 1$
limit of \eqref{eq:CCprime} or using \eqref{eq:B1}.

The limit of Dirac operator in \eqref{eq:DN} is $D_N=\deb+\deb^\dag$, and the formul{\ae}
\eqref{eq:dddag} gives
$$
D_N^2\big|_{\Omega^k_N}=
\mathcal{C}_{q=1}-\mathcal{C}'_{q=1}-\tfrac{\ell+1}{8\ell}\hat{H}(\hat{H}+\ell)
+k(\ell+1-N) \;,
$$
where the right action is understood.
The operator $\za\,\hat{H}|_{\Omega^k_N}$ is given by $2k-\frac{2\ell}{\ell+1}N$ times the identity
operator. The operator $\za\,\mathcal{C}'_{q=1}$ is constant on $\Omega^k_N$ and given by (cf.~\eqref{eq:zaCprimeq})
$$
\za\,\mathcal{C}'_{q=1}=\tfrac{\ell+1}{2\ell}k(\ell-k) \;.
$$
Thus
$$
D_N^2\big|_{\Omega^k_N}=
\mathcal{C}_{q=1}+\tfrac{\ell}{2(\ell+1)}N(\ell+1-N) \;.
$$
Notice that the constant on the right hand side does not depend on $k$. If we forget
Cumulatively, the results of Proposition \ref{prop:ha} can be can be simplified as follows.
If $N\leq 0$,
$$
\Omega^\bullet_N\simeq
V_{(0,0,\ldots,0,-N)}\,\oplus\,
2\,\bigoplus\nolimits_{k=1}^\ell
\bigoplus\nolimits_{m\in\N}
V_{(m,0,\ldots,0,m-N+k)\,+\,\underline{e}_k} \;,
$$
if $N>\ell$,
$$
\Omega^\bullet_N\simeq
V_{(N-\ell-1,0,\ldots,0,0)}\,\oplus\,
2\,\bigoplus\nolimits_{k=1}^\ell
\bigoplus\nolimits_{m\in\N}
V_{(m+N-k,0,\ldots,0,m)\,+\,\underline{e}_k} \;,
$$
and if $1\leq N\leq\ell$,
\begin{align*}
\Omega^\bullet_N &\simeq
2\bigoplus\nolimits_{k=1}^{N-1}
\bigoplus\nolimits_{m\in\N}
V_{(m+N-k,0,\ldots,0,m)\,+\,\underline{e}_k} \\
&\oplus 2\bigoplus\nolimits_{m\in\N}
V_{(m,0,\ldots,0,m)\,+\,\underline{e}_N} \\
&\oplus
2\bigoplus\nolimits_{k=N+1}^{\ell}
\bigoplus\nolimits_{m\in\N}
V_{(m,0,\ldots,0,m-N+k)\,+\,\underline{e}_k} \;.
\end{align*}
Eigenvalues of $D_N$ are computed using the $q\to 1$ limit of 
Lemma \ref{lemma:multeig}. Let
\begin{subequations}
\begin{align}
\lambda_{m,k}^N &:=\sqrt{(m+N)(m+k)} \;, \label{eq:Dolan} \\
\mu_{m,k}^N &:=\frac{k(2m+k+N)}{(m+N)(m+k)}
\binom{m+\ell}{\ell}\binom{m+k+N-1}{\ell}\binom{\ell}{k} \;.
\end{align}
\end{subequations}
If $N\leq 0$, $D_N$ has a kernel of dimension $\binom{-N+\ell}{\ell}$ while
non-zero eigenvalues $\lambda$'s and their multiplicities $\mu$'s are given by
$$
\mathrm{Sp}(D_N)\smallsetminus\ker(D_N)
=\bigl\{(\pm\lambda_{m,k}^{\ell+1-N},\mu_{m,k}^{\ell+1-N})\,;\;1\leq k\leq\ell\,,\;m\in\N\bigr\} \;.
$$
If $N>\ell$, $D_N$ has a kernel of dimension $\binom{N-1}{\ell}$ and
$$
\mathrm{Sp}(D_N)\smallsetminus\ker(D_N)
=\bigl\{(\pm\lambda_{m,k}^{N},\mu_{m,k}^{N})\,;\;1\leq k\leq\ell\,,\;m\in\N\bigr\} \;.
$$
Notice that in the cases above $\mathrm{Sp}(D_N)=\mathrm{Sp}(D_{\ell+1-N})$.
If $1\leq N\leq\ell$, $D_N$ is invertible with spectrum
\begin{align*}
\mathrm{Sp}(D_N)
&=\bigl\{(\pm\lambda_{m,\ell+1-k}^N,\mu_{m,\ell+1-k}^N)\,;\;1\leq k<N\,,\;m\in\N\bigr\} \cup \\
&\quad \cup
\bigl\{(\pm\lambda_{m,k}^{\ell+1-N},\mu_{m,k}^{\ell+1-N})\,;\;N\leq k\leq\ell\,,\;m\in\N\bigr\} \;.
\end{align*}
In particular, if $\ell$ is odd and $N=\frac{1}{2}(\ell+1)$, the Dirac operator of the Fubini-Study
metric $\D=D_{\frac{1}{2}(\ell+1)}$ has eigenvalues:
$$
\pm\sqrt{(m+\tfrac{\ell+1}{2})(m+k)} \;,\qquad
\tfrac{\ell+1}{2}\leq k\leq \ell\,,\;m\in\N\,.
$$

\smallskip

Let us compare these results with the literature.

For odd $\ell$, the spectrum of the Dirac operator of the Fubini-Study metric on $\CP^\ell$
(i.e.~$D_N$ when $N=\frac{1}{2}(\ell+1)$) has been computed in \cite{CFG89} (cf.~also
\cite{SS93}). The same spectrum has been computed with a different method in \cite{AB98},
and by Theorem 4.6 in \cite{AB98} it coincides with the one in \cite{CFG89,SS93} after the right
parameter substitutions. Finally, in \cite{DHMOC08} the spectrum of $D_N$ for arbitrary $N$ is
computed (for both odd and even $\ell$); when $\ell$ is odd and $N=\frac{1}{2}(\ell+1)$,
coincides with the spectrum in \cite{AB98}.

We obtain the same result.
The comparison with the notations of \cite[App.~B]{DHMOC08} is as follows: $n:=\ell$ is the dimension (over $\C$) of the space,
$q:=\ell+1-N$ is the `charge' of the bundle used to construct the Dirac operator, $l:=m+k+N-\ell-1$
and $k':=k-1$ is the integers that label the eigenvalues. Modulo a misprint in last line of \cite[App.~B]{DHMOC08}
(the constraint is $l+q-k'\geq 1$, with a $-$ sign, cf.~(105) of their paper), their spectrum coincide
with our.


\providecommand{\bysame}{\leavevmode\hbox to3em{\hrulefill}\thinspace}


\begin{thebibliography}{10}

\bibitem{AB98}
B.~Ammann and C.~B{\"a}r, \emph{{The Dirac Operator on Nilmanifolds and
  Collapsing Circle Bundles}}, Ann. Global Anal. Geom. \textbf{16} (1998),
  no.~3, 221--253, \arxiv{math/9801091v1}.

\bibitem{Bin91}
A.M. Bincer, \emph{{Casimir operators for $su_q(n)$}}, J.~Phys. \textbf{A24}
  (1991), no.~19, L1133--L1138, \doi{10.1088/0305-4470/24/19/002}.

\bibitem{BB05}
A.~Bj{\"o}rner and F.~Brenti, \emph{{Combinatorics of Coxeter Groups}},
  Springer, 2005.

\bibitem{CFG89}
M.~Cahen, A.~Franc, and S.~Gutt, \emph{{Spectrum of the Dirac operator on
  complex projective space $P_{2q-1}(\mathbb{C})$}}, Lett. Math. Phys.
  \textbf{18} (1989), no.~2, 165--176, \doi{10.1007/BF00401871}, Erratum in
  Lett. Math. Phys.~32 (1994) 365-368.

\bibitem{Cha91}
A.~Chakrabarti, \emph{{$q$-analogs of $IU(n)$ and $U(n,1)$}}, J.~Math. Phys.
  \textbf{32} (1991), no.~5, 1227--1234, \doi{10.1063/1.529319}.

\bibitem{CHZ96}
C.-S. Chu, P.-M. Ho, and B.~Zumino, \emph{{Geometry of the quantum complex
  projective space $CP_q(N)$}}, Eur. Phys. J. \textbf{C72} (1996), no.~1,
  163--170, \doi{10.1007/s002880050233}, \arxiv{q-alg/9510021v1}.

\bibitem{Con94}
A.~Connes, \emph{{Noncommutative Geometry}}, Academic Press, 1994.

\bibitem{Con96}
\bysame, \emph{{Gravity coupled with matter and the foundation of
  non-commutative geometry}}, Commun. Math. Phys. \textbf{182} (1996), no.~1,
  155--176, \doi{10.1007/BF02506388}, \arxiv{hep-th/9603053}.

\bibitem{Dab08}
L.~D{\k a}browski, \emph{The local index formula for quantum $SU(2)$},
  in ``Traces in Number Theory, Geometry and Quantum Fields'',
  S.~Albeverio et al. (Eds), 
  Aspects of Mathematics \textbf{E38}, Vieweg Verlag, 2008.

\bibitem{DS03}
L.~D{\k a}browski and A.~Sitarz, \emph{{Dirac operator on the standard Podle\'s
  quantum sphere}}, Noncommutative geometry and quantum groups, vol.~61, Banach
  Center Publ., 2003, pp.~49--58, \arxiv{math/0209048}.

\bibitem{DS94}
L.~D{\k a}browski and J.~Sobczyk, \emph{{Left regular representation and
  contraction of $sl_q(2)$ to $e_q(2)$}}, Lett. Math. Phys. \textbf{32} (1994),
  no.~3, 249--258, \doi{10.1007/BF00750667}.

\bibitem{DDL08}
F.~D'Andrea, L.~D{\k a}browski, and G.~Landi, \emph{{The Isospectral Dirac
  Operator on the $4$-dimensional Orthogonal Quantum Sphere}}, Commun. Math.
  Phys. \textbf{279} (2008), no.~1, 77--116, \doi{10.1007/s00220-008-0420-x},
  \arxiv{math.QA/0611100v2}.

\bibitem{DDL08b}
\bysame, \emph{{The Noncommutative Geometry of the Quantum Projective Plane}},
  Rev. Math. Phys. \textbf{20} (2008), no.~8, 979--1006,
  \doi{10.1142/S0129055X08003493}, \arxiv{0712.3401}.

\bibitem{DDLW07}
F.~D'Andrea, L.~D{\k a}browski, G.~Landi, and E.~Wagner, \emph{{Dirac operators
  on all Podle\'s spheres}}, J. Noncomm. Geom. \textbf{1} (2007), no.~2,
  213--239, \arxiv{math/0606480}.

\bibitem{DL08}
F.~D'Andrea and G.~Landi, \emph{{Antiself-dual Connections on the Quantum
  Projective Plane: Monopoles}}, in preparation, 2009.

\bibitem{DHMOC08}
B.P. Dolan, I.~Huet, S.~Murray, and D.~O'Connor, \emph{{A universal Dirac
  operator and noncommutative spin bundles over fuzzy complex projective
  spaces}}, JHEP \textbf{03} (2008), 029, \doi{10.1088/1126-6708/2008/03/029},
  \arxiv{0711.1347}.

\bibitem{GT88a}
I.M. Gelfand and M.L. Tsetlin, \emph{{Finite-dimensional representations of the
  group of unimodular matrices}}, {I.M.~Gelfand: Collected papers, vol. II},
  Springer-Verlag, 1988, pp.~653--656, English translation of the paper: Dokl.
  Akad. Nauk SSSR 71 (1950) 825-828.

\bibitem{HK04}
I.~Heckenberger and S.~Kolb, \emph{{The locally finite part of the dual
  coalgebra of quantized irreducible flag manifolds}}, Proc. London Math. Soc.
  \textbf{89} (2005), no.~2, 457--484, \doi{10.1112/S0024611504014777},
  \arxiv{math/0301244}.

\bibitem{HK06b}
\bysame, \emph{{De Rham Complex for Quantized Irreducible Flag Manifolds}},
  J.~Algebra \textbf{305} (2006), no.~2, 704--741,
  \doi{10.1016/j.jalgebra.2006.02.001}, \arxiv{math.QA/0307402v1}.

\bibitem{IN66}
C.~Itzykson and M.~Nauenberg, \emph{{Unitary Groups: Representations and
  Decompositions}}, Rev. Mod. Phys. \textbf{38} (1966), no.~1, 95--120,
  \doi{10.1103/RevModPhys.38.95}.

\bibitem{KS97}
A.~Klimyk and K.~Schm{\"u}dgen, \emph{{Quantum groups and their
  representations}}, Springer, 1997.

\bibitem{Kra04}
U.~Kr{\"a}hmer, \emph{{Dirac Operators on Quantum Flag Manifolds}}, Lett. Math.
  Phys. \textbf{67} (2004), no.~1, 49--59,
  \doi{10.1023/B:MATH.0000027748.64886.23}, \arxiv{math/0305071}.

\bibitem{NT07}
S.~Neshveyev and L.~Tuset,
  \emph{The Dirac operator on compact quantum groups},
  \arxiv{math/0703161}, 2007.
 
\bibitem{Res87}
N.Y. Reshetikhin, \emph{{Quantized universal enveloping algebras, the
  Yang-Baxter equation and invariants of links I and II}}, {Preprint LOMI
  E-4-87 E-17-87}, 1987.

\bibitem{Ser01}
J.-P. Serre, \emph{Complex Semisimple Lie Algebras}, Springer, 2001.

\bibitem{SW04}
K.~Schm{\"u}dgen and E.~Wagner, \emph{{Dirac operator and a twisted cyclic
  cocycle on the standard Podle\'s quantum sphere}}, J.~Reine Angew. Math.
  \textbf{574} (2004), 219--235, \doi{10.1515/crll.2004.072},
  \arxiv{math/0305051}.

\bibitem{SS93}
S.~Seifarth and U.~Semmelmann, \emph{{The Spectrum of the Dirac Operator on the
  Odd Dimensional Complex Projective Space $P^2_{m-1}(C)$}}, {SFB 288 Preprint
  \textbf{95}}, 1993.

\bibitem{Sit03}
A.~Sitarz, \emph{{Equivariant spectral triples}}, Noncommutative Geometry and
  Quantum Groups, vol.~61, Banach Centre Publ., 2003, pp.~231--263.

\bibitem{VS91}
L.~Vaksman and Ya. Soibelman, \emph{{The algebra of functions on the quantum
  group $SU(n+1)$ and odd-dimensional quantum spheres}}, Leningrad Math.~J.
  \textbf{2} (1991), 1023--1042.

\end{thebibliography}
\end{document}